\documentclass{article}
\usepackage[a4paper,marginratio=1:1,width=150mm,top=25mm,bottom=25mm]{geometry}
\usepackage[T1]{fontenc}
\usepackage[utf8]{inputenc}

\usepackage{amsmath}
\usepackage{amssymb}
\usepackage{amsthm}
\usepackage{amsfonts}

\usepackage{bbold}

\usepackage{tikz}
\usetikzlibrary{calc}
\usepackage{tikz-cd}
\usepackage{todonotes}

\usepackage{quiver}

\usepackage[citestyle=alphabetic,bibstyle=alphabetic,maxbibnames=9]{biblatex}
\usepackage{graphicx}
\usepackage{xcolor}
\usepackage{thmtools}
\usepackage{hyperref}
\usepackage{cleveref}

\usepackage{parskip}
\usepackage{tcolorbox}
\tcbuselibrary{most}

\definecolor{RUBBlue}{RGB}{128,0,50}
\definecolor{RUBGrey}{RGB}{231, 231, 231}
\definecolor{RUBLightGrey}{RGB}{252, 252, 252}
\definecolor{RUBGreen}{RGB}{0, 49, 83}

\theoremstyle{plain}
\newtheorem{theorem}{Theorem}[section]
\crefname{theorem}{theorem}{theorems}
\newtheorem{proposition}[theorem]{Proposition}
\crefname{proposition}{proposition}{propositions}
\newtheorem{lemma}[theorem]{Lemma}
\crefname{lemma}{lemma}{lemmas}
\newtheorem{corollary}[theorem]{Corollary}
\crefname{corollary}{corollary}{corollaries}

\theoremstyle{definition}
\newtheorem{definition}[theorem]{Definition}
\crefname{definition}{definition}{definitions}
\newtheorem{construction}[theorem]{Construction}
\crefname{construction}{construction}{constructions}
\newtheorem{notation}[theorem]{Notation}
\crefname{notation}{notation}{notations}

\crefname{convention}{convention}{conventions}

\theoremstyle{remark}
\newtheorem{example}[theorem]{Example}
\crefname{example}{example}{examples}
\newtheorem{remark}[theorem]{Remark}
\crefname{remark}{remark}{remarks}

\newtheorem{introthm}{Theorem}

\crefname{introthm}{theorem}{theorems}
\tcolorboxenvironment{introthm}{colback=white, colframe=RUBBlue}

\tcolorboxenvironment{theorem}{colback=white, colframe = RUBBlue}
\tcolorboxenvironment{proposition}{colback=white, colframe = RUBBlue}
\tcolorboxenvironment{lemma}{colback=white, colframe = RUBBlue}
\tcolorboxenvironment{corollary}{colback=white, colframe = RUBBlue}
\tcolorboxenvironment{definition}{colframe= RUBGreen, colback=white}
\tcolorboxenvironment{construction}{colframe= RUBGreen, colback=white,breakable}
\tcolorboxenvironment{notation}{colframe= RUBGreen, colback=white}
\tcolorboxenvironment{convention}{colframe= RUBGreen, colback=white}
\tcolorboxenvironment{proof}{blanker,breakable,left=1mm,before skip=10pt,after skip=10pt
}
\tcolorboxenvironment{remark}{blanker,breakable,
before skip=10pt,after skip=10pt}

\DeclareMathOperator{\im}{im}

\DeclareMathOperator*{\colim}{colim}

\DeclareMathOperator{\id}{id}

\newcommand{\st}{\mathrm{st}}
\newcommand{\op}{\mathrm{op}}

\newcommand{\nc}{\mathrm{nc}}
\newcommand{\fpqc}{\mathrm{fpqc}}
\newcommand{\et}{\mathrm{\acute{e}t}}

\newcommand{\fin}{\mathrm{fin}}
\newcommand{\fab}{\mathrm{fab}}
\newcommand{\gl}{\mathrm{gl}}
\newcommand{\ab}{\mathrm{ab}}

\DeclareMathOperator{\Spec}{Spec}
\DeclareMathOperator{\Spet}{Sp\acute{e}t}
\DeclareMathOperator{\Fun}{Fun}
\DeclareMathOperator{\QCoh}{QCoh}
\DeclareMathOperator{\CAlg}{CAlg}
\DeclareMathOperator{\Shv}{Shv}
\DeclareMathOperator{\Orb}{Orb}

\DeclareMathOperator{\map}{map}
\DeclareMathOperator{\Mod}{Mod}
\DeclareMathOperator{\Map}{Map}

\DeclareMathOperator{\Sp}{Sp}
\DeclareMathOperator{\Spc}{Spc}
\DeclareMathOperator{\Cat}{Cat}

\DeclareMathOperator{\Ar}{Ar}
\DeclareMathOperator{\Fin}{Fin}
\DeclareMathOperator{\SpDM}{SpDM}
\DeclareMathOperator{\SpSch}{SpSch}
\newcommand{\PrLst}{{\mathrm{Pr}^L_\st}}
\DeclareMathOperator{\Stk}{Stk}
\DeclareMathOperator{\Top}{Top}
\DeclareMathOperator{\Glo}{Glo}

\renewcommand{\O}{\mathcal{O}}
\newcommand{\X}{\mathrm{X}}
\newcommand{\Y}{\mathrm{Y}}

\newcommand{\ko}{\mathrm{ko}}

\newcommand{\ku}{\mathrm{ku}}
\newcommand{\KU}{\mathrm{KU}}

\newcommand{\nb}[1]{\left(#1 \right)}

\newcommand{\F}{\mathbb{F}}

\newcommand{\oo}{\mathcal{O}}

\newcommand{\zz}{\mathbb{Z}}

\newcommand{\qq}{\mathbb{Q}}

\newcommand{\mm}{\mathfrak{m}}
\newcommand{\nn}{\mathfrak{n}}
\renewcommand{\gg}{\mathbb{G}}
\renewcommand{\ss}{\mathbb{S}}
\newcommand{\bb}{\mathbb{B}}
\newcommand{\einfty}{\mathbb{E}_\infty}

\newcommand{\spec}{\mathrm{Spec}}

\begin{document}
\title{A family completion theorem for tempered cohomology}
\author{Leonard Tokic}
\date{\today}
\maketitle
\begin{abstract}
  Let \(\gg\) be an oriented \(\mathbb{P}\)-divisible group over a noetherian \(\einfty\)-ring \(R\), let \(G\) be a finite group, and let \(\mathcal{F}\) be a family of subgroups of \(G\).
  We show that completion of \(R(\gg)_G\)-modules at \(\mathcal{F}\) agrees with algebraic completion at the ideal
  \[I_\gg(\mathcal{F})=\bigcap_{H\in\mathcal{F}}\ker\nb{\pi_0R(\gg)^{ G}\to \pi_0R(\gg)^{ H}}.\]
  For \(\gg=\mu_{\mathbb{P}^\infty}\) over \(\KU\) this recovers the family completion theorem of Adams, Haeberly, Jackowski, and May, and for the trivial family the classical Atiyah-Segal completion theorem.
  The main input is a theory of support for points of the tempered character stack \(\gg\{\bb G\}\), in the spirit of Segal's analysis of the prime spectrum of the complex representation ring: we show that the support of a point is a single conjugacy class of abelian subgroups of \(G\), and that the points supported inside \(\mathcal{F}\) are exactly the preimage of \(V(I_\gg(\mathcal{F}))\) under the affinization map.
  We also prove a version over locally noetherian geometric base stacks, in which the ideal is replaced by an open substack of \(\gg(\bb G)\), the analogue over such a base of \(\spec\,R_\gg^{\bb G}\), and which applies for instance to genuine equivariant topological modular forms.
\end{abstract}

\tableofcontents

\section{Introduction}
 
Let \(G\) be a compact Lie group and let \(X\) be a finite \(G\)-CW complex.
The Atiyah-Segal completion theorem \cite{AtiyahSegal_1969_Equivariant$K$theoryCompletion} states that the projection \(EG\times X\to X\) induces an isomorphism
\[\KU_G^*(X)^\wedge_I\xrightarrow{\ \sim\ }\KU^*_G(EG\times X)=\KU^*(X_{hG}),\]
where \(I\subset R(G)\simeq \KU_G^0(*)\) is the augmentation ideal.
One can read this as saying that two rather different ways of forgetting equivariance agree: completing the coefficients at an ideal, and passing to the Borel construction.
 
Adams, Haeberly, Jackowski, and May \cite{AdamsHaeberlyJackowskiMay_1988_GeneralizationAtiyahsegalCompletion} showed that this is the first case of a much more flexible statement, in which the trivial subgroup is replaced by an arbitrary family \(\mathcal{F}\) of closed subgroups of \(G\).
They work with pro-group valued equivariant \(K\)-theory, \(\mathcal{KU}^*_G(X)=\{\KU^*_G(X_\alpha)\}\) with \(X_\alpha\) running over the finite subcomplexes of \(X\), define the \(\mathcal{F}\)-adic completion to be the pro-group
\[\{\KU^*_G(X_\alpha)/J\KU^*_G(X_\alpha)\},\]
where \(J\) runs over the finite products of the ideals \(I^G_H=\ker \KU^0_G(*)\to \KU_H^0(*)\) with \(H\in\mathcal{F}\), and prove that this is isomoprhic to \(\mathcal{KU}_G^*(X\times E\mathcal F)\).

The family \(\mathcal{F}=\{e\}\) recovers the Atiyah-Segal completion theorem, and the family of all subgroups gives nothing at all, so the interest lies in between.
Statements of this form are also known, for example, for equivariant stable cohomotopy \cite{AdamsHaeberlyJackowskiMay_1988_GeneralizationSegalConjecture}.
 
Tempered cohomology, introduced by Lurie in \cite{Lurie_2019_EllipticCohomologyIII}, is rich source of equivariant cohomology theories.
To a pre-oriented \(\mathbb{P}\)-divisible group \(\gg\) over an \(\einfty\)-ring \(R\) it associates a limit preserving functor
\[R_\gg^\bullet:(\Spc^\gl_\fin)^\op\to \CAlg(\Sp)_{R/},\]
which for \(\gg=\mu_{\mathbb{P}^\infty}\) over \(\KU\) recovers equivariant \(K\)-theory, for finite groups of equivariance.
Another key example of interest is given by taking \(\gg\) to be the torsion in an oriented elliptic curve, giving a version of equivariant elliptic cohomology.

Our goal in this paper is prove a family completion theorem for tempered cohomology theories.
 
\subsection{Main results}
 
Fix a noetherian \(\einfty\)-ring \(R\) and an oriented \(\mathbb{P}\)-divisible group \(\gg\) over \(\Spec R\).
Under these hypotheses, Gepner, Linskens, and Pol showed that the functor \(R_\gg^\bullet\) is representable by a commutative algebra \(R(\gg)\) in global spectra.

Given a finite group \(G\) and a family \(\mathcal{F}\) of subgroups of \(G\), there are two recollements of \(\Mod_{R(\gg)_G}(\Sp_G)\) at hand.
The first is geometric: the idempotent \(\mathbb{E}_0\)-algebra \(\widetilde{E\mathcal{F}}\in\Sp_G\) gives a recollement \(\mathcal{R}(\widetilde{E\mathcal{F}})\) of \(\Sp_G\), which may be tensored up.
The second is algebraic: the ideal \(I_\gg(\mathcal{F})\subset \pi_0R(\gg)^G\) gives, by the theory of \cite[§~7]{Lurie_SAG}, a recollement \(\mathcal{R}(I_\gg(\mathcal{F}))\) of \(\Mod_{R(\gg)^G}(\Sp)\), which may likewise be tensored up.
Our first main result is that these agree.
 
\begin{introthm}[\cref{thm:betterTemperedFamilyCompletion}]
  Let \(G\) be a finite group, let \(\mathcal{F}\) be a family of subgroups of \(G\), and let
  \[I_\gg(\mathcal{F})=\bigcap_{H\in\mathcal{F}}\ker\nb{\pi_0R(\gg)^G\to \pi_0R(\gg)^H}.\]
  Then we have an equivalence of recollements
  \[\mathcal{R}(\widetilde{E\mathcal{F}})\otimes_{\Sp_G}\Mod_{R(\gg)_G}(\Sp_G)\simeq \mathcal{R}(I_\gg(\mathcal{F}))\otimes_{\Mod_{R(\gg)^G}(\Sp)}\Mod_{R(\gg)_G}(\Sp_G).\]
  In particular, an \(R(\gg)_G\)-module is \(\mathcal{F}\)-complete if and only if it is \(I_\gg(\mathcal{F})\)-complete.
\end{introthm}
 
Unwinding the last sentence for a module \(M\), the theorem says that the natural map
\[M^\wedge_{I_\gg(\mathcal{F})}\to \underline{\map}_G\nb{E\mathcal{F}_+,M}\]
is an equivalence.
For \(\gg=\mu_{\mathbb{P}^\infty}\) over \(\KU\) this recovers a variant of the theorem of Adams, Haeberly, Jackowski, and May.
 
Strictly speaking, their theorem is a little stronger than this.
Their isomorphism of pro-groups says not merely that the two completions agree, but that the natural filtrations on either side are compatible, which a statement about the completed objects alone does not see.
On the other hand, our version directly gives an equivalence of \(G\)-spectra.

Our method of proof relies heavily on the basechange properties of recollements we establish in \cref{sec:recollements}, and it is unclear how one would incorporate information about the filtrations.

One direction of \cref{thm:betterTemperedFamilyCompletion} is elementary and needs none of the machinery we develop: an \(\mathcal{F}\)-complete module is \(I_\gg(\mathcal{F})\)-complete, essentially because inverting an element of \(I_\gg(\mathcal{F})\) already kills the relevant geometric fixed points (\cref{prop:FCompleteImpliesIFComplete}).
Everything else in this paper is in service of the converse.
 
As an application we obtain a finiteness statement in bounded height, to the effect that in height at most \(n\) one only needs to know tempered cohomology on groups with at most \(n\) generators.
 
\begin{introthm}[\cref{prop:rightKanExtension}]
  Let \(n\geq 0\) and suppose that, for every point \(\mathfrak{p}\in\spec\,R\) and every \(\ell\in\mathbb{P}\), the \(\ell\)-divisible group \((\gg_{\widehat{R}_\mathfrak{p}})^\et_{(\ell)}\) has height at most \(n\).
  Then the functor
  \[\Glo_\fab^\op\to \CAlg(\Sp),\,\bb V\mapsto R_\gg^{\bb V}\]
  is right Kan extended from the full subcategory on those \(\bb H\) with \(H\) generated by at most \(n\) elements.
\end{introthm}
 
Finally we would like a version over a general base stack, and here one immediately runs into the difficulty that there is no ring in which to take an ideal.
What survives is the closed subset that the ideal cuts out.
We call a stack \(M\) a {\it locally noetherian geometric} stack, or LNG-stack, if it admits a small colimit presentation by noetherian affines with flat transition maps (\cref{def:LNG_stack}); the stacks \(\mathrm{M}^\mathrm{or}_\mathrm{Ell}\) and \(\mathrm{M}^\mathrm{or}_\mathrm{Tori}\simeq(\Spec \KU)//C_2\) are of this form.
Over such a base the role of \(\spec\,R_\gg^{\bb G}\) is played by the relatively affine stack \(\gg(\bb G)\) over \(M\) associated to \(\O_M(\gg)^G\).
The union \(V_\gg(\mathcal{F})\) of the images of the \(|\gg(\bb H)|\) for \(H\in\mathcal{F}\) is then a closed subset of \(|\gg(\bb G)|\), and the open immersion associated to its complement is locally quasi-compact, so that it gives rise to a recollement \(\mathcal{R}(I_\gg(\mathcal{F}))\).
 
\begin{introthm}[\cref{thm:family_completion_LNG_version}]
  Let \(M\) be an LNG-stack and let \(\gg\) be an oriented \(\mathbb{P}\)-divisible group over \(M\).
  For any finite group \(G\) and family \(\mathcal{F}\) of subgroups of \(G\), the basechanges of \(\mathcal{R}(I_\gg(\mathcal{F}))\) and of \(\mathcal{R}(\widetilde{E\mathcal{F}})\) to \(\Mod_{\O_M(\gg)_G}(\Sp_G\otimes\QCoh(M))\) are equivalent.
\end{introthm}
 
Applied to the universal oriented elliptic curve over \(\mathrm{M}^\mathrm{or}_\mathrm{Ell}\), which is \(0\)-affine, this is a family completion theorem for genuine \(G\)-equivariant topological modular forms; combining it with the height bound of \cref{rmk:boundedHeightBoundedGeneration} recovers, for finite abelian groups of equivariance, the fact that \(\bb G\mapsto \mathrm{TMF}^G\) is right Kan extended from abelian groups generated by at most two elements.

\subsection{Proof ideas}
 
The basic structure of our proof of theorem A is heavily inspired by that of Adams, Haeberly, Jackowski, and May, and it may help the reader to keep their argument in mind.
Theirs proceeds by induction over the subgroups of \(G\).
Having reduced to a vanishing statement, they smash with a cofibre sequence \(S^0\to Y\to Y/S^0\) in which \(Y\) is a colimit of representation spheres with \(Y^G=S^0\) and \(Y^H\) contractible for proper \(H\), that is, with the isotropy separation sequence for the family \(\mathcal{P}\) of proper subgroups.
The \(Y\)-part is then handled outright, by an argument with Bott classes and Euler classes, while the \(Y/S^0\)-part is built out of cells \(G/H_+\) with \(H\) proper and is handled by the inductive hypothesis together with a change of groups isomorphism.
That change of groups step is the crux of the matter, and it rests on the purely algebraic assertion that the \(\mathcal{F}\)-adic and \(\mathcal{F}\vert_H\)-adic topologies on \(R(H)\) coincide, which they in turn deduce from Segal's theory of supports.
 
Our argument has the same shape.
We may assume \(G\notin\mathcal{F}\), and let \(M\) be an \(I_\gg(\mathcal{F})\)-complete module.
Mapping the isotropy separation sequence for the family \(\mathcal{P}\) of proper subgroups into \(M\) splits the problem in two.
The \(\widetilde{E\mathcal{P}}\)-part needs no induction, and here we can be more efficient than in the classical case.
If \(G\) is abelian it follows from a comparison with the work of Gepner, Linskens, and Pol, whose ideal from \cite[def.~15.10]{GepnerLinskensPol_2024_Global2ringsGenuine} coincides with ours and whose \cite[prop.~15.14]{GepnerLinskensPol_2024_Global2ringsGenuine} places any nilpotent module in the localizing tensor ideal generated by modules induced from proper subgroups.
If \(G\) is not abelian it is immediate from \(R(\gg)_G\) being nilpotent with respect to the family of abelian subgroups in the sense of \cite{MathewNaumannNoel_2017_NilpotenceDescentEquivariant}, which we deduce in \cref{lem:temperedIsMNNAbNilpotent} from the vanishing of geometric fixed points at nonabelian subgroups.
 
The \(E\mathcal{P}_+\)-part reduces, along the orbits \(G/H_+\) with \(H\) proper, to showing that \(\mathrm{Res}^G_HM\) is \(\mathcal{F}\vert_H\)-complete, which is where we want to apply the inductive hypothesis; for that we need to know that \(\mathrm{Res}^G_HM\) is \(I_\gg(\mathcal{F}\vert_H)\)-complete.
Restriction commutes with completions tensored up from \(\Mod_{R(\gg)^H}(\Sp)\), so we know it is complete for the extended ideal \(J=\mathrm{res}^G_H(I_\gg(\mathcal{F}))\pi_0R(\gg)^H\), and comparing \(J\) with \(I_\gg(\mathcal{F}\vert_H)\) is where our prove differs significantly:
We do not compare the two ideals directly.

Instead, we introduce the {\it character stack} \(\gg\{\bb G\}\), whose global sections agree with \(R(\gg)^G\).
Since \cref{lem:factors_over_character_stack} tells us that the functor \(\Mod_{R(\gg)^H}(\Sp)\to \Mod_{R(\gg)_H}(\Sp_H)\) factors over \(\QCoh(\gg\{\bb H\})\), it suffices to compare the two recollements after base change along the affinization map to the character stack.
As we show in \cref{ssec:support}, the points in \(\gg\{\bb G\}\) behave nicely under change of groups.
In particular, both \(V(J)\) and \(V(I_\gg(\mathcal{F}\vert_H))\) have the same preimage in \(|\gg\{\bb H\}|\), so the two completions of \(\mathrm{Res}^G_HM\) agree.

Another consequence of our study of the character stacks is a close connection between the heights of \(\gg\) and the subgroups of \(G\) on which a point of \(\gg\{\bb G\}\) is supported.
As a corollary, we obtain theorem B.
 
Theorem C is then deduced from theorem A by descent along an LNG-presentation.
The point is that the formation of \(V(I_\gg(\mathcal{F}))\) commutes with flat basechange between noetherian affines (\cref{lem:VIF_flat_basechange}), so that the closed subsets glue to a closed subset of \(|\gg(\bb G)|\) whose complementary open immersion is locally quasi-compact, and the resulting recollement restricts on each chart to the affine one.
 
\subsection{Outline}
 
In \cref{sec:recollements} we collect what we need about recollements.
After recalling the basic notions we introduce, in \cref{ssec:linear_recollements}, the notion of a \(C\)-linear recollement of a \(C\)-module \(X\): one whose local objects are closed under both tensors and cotensors by objects of \(C\).
\Cref{prop:ConsequencesOfLinearRecollement} identifies these with cofibre sequences in \(\Mod_C(\PrLst)\) whose functors are internal left adjoints, from which stability under basechange and under tensoring follows without much trouble.
The remaining subsections treat the four sources of recollements we use: idempotent \(\mathbb{E}_0\)-algebras in \cref{ssec:idempotent}, finitely generated ideals in \cref{ssec:ideal_recollement}, families of subgroups in \cref{ssec:family_recollement}, and open immersions of stacks in \cref{ssec:open_immersions}.
A reader who is willing to grant that these all give \(C\)-linear recollements may safely skip ahead.
 
\Cref{sec:tempered} contains the main results.
In \cref{ssec:support} we develop the theory of support for points in the character stacks, most notably that the support consists of a single conjugacy class of an abelian subgroup (\cref{prop:temperedSupportsAreConjugate}).
In \cref{ssec:family_completion_affine} we prove theorem A (\cref{thm:betterTemperedFamilyCompletion}) and deduce theorem B (\cref{prop:rightKanExtension}) from it.
In \cref{ssec:LNG} we turn to general base stacks: we introduce LNG-stacks, assemble the genuine \(G\)-equivariant algebra \(\O_M(\gg)_G\) out of the structure sheaf of \cite{BalderramaDaviesLinskens_2026_AmbidextrousGlobalSpectra}, show that \(V_\gg(\mathcal{F})\) is closed, and prove theorem C (\cref{thm:family_completion_LNG_version}).
 
Two appendices collect material which would otherwise interrupt the exposition.
\Cref{sec:stacks} gathers what we need about the stacks of \cite{BalderramaDaviesLinskens_2025_AffinenessReconstructionComplexperiodic}: underlying locales and topological spaces, local and finitary properties of morphisms, and open immersions, including the existence of the open substacks used to define \(U_\gg(\mathcal{F})\).
\Cref{sec:lifting_monos} is independent of the rest of the paper and gives a criterion for lifting a diagram across a monomorphism in an operadic setting: for a fibration of \(\infty\)-operads and a monomorphism in a fibre, the space of lifts is empty or contractible, and it is nonempty precisely when the evident factorization condition holds.
We use this in \cref{prop:ConsequencesOfLinearRecollement} to produce \(C\)-linear structures on adjoints appearing in \(C\)-linear recollements.
 
\subsection{Conventions}
 
We work throughout in the language of \(\infty\)-categories, the standard references being \cite{Lurie_HTT} and \cite{Lurie_HA}, whose content we use freely; all algebras, modules, and monoidal structures are understood in this sense, and we drop the prefix `\(\infty\)-' from `category' where no confusion can arise.
For stacks and quasi-coherent sheaves we follow the conventions of \cite{BalderramaDaviesLinskens_2025_AffinenessReconstructionComplexperiodic}; in particular \(|X|\) denotes the underlying topological space of a stack \(X\).
For tempered cohomology we follow \cite{Lurie_2019_EllipticCohomologyIII} and \cite{BalderramaDaviesLinskens_2026_AmbidextrousGlobalSpectra}, and for global spectra \cite{GepnerLinskensPol_2024_Global2ringsGenuine}.
In particular, \(\fin\) and \(\fab\) denote the global families of finite and finite abelian groups.
Families of subgroups are always assumed to be closed under subgroups and conjugation, and for a subgroup \(H\subset G\) we write \(\mathcal{F}\vert_H\) for the family of those subgroups of \(H\) which lie in \(\mathcal{F}\).
 
\subsection*{Acknowledgements}
Over the course of this project, I enjoyed various clarifying and enlightening conversations.
Thus, I want to thank Jack Davies, Daniel Garrido, Alexander Ivanov, Gerd Laures, Georg Lehner, Sil Linskens, Maxime Ramzi, Björn Schuster, and Stefan Schwede, to name a few.

As always, I'm also thankful for the algebraic topology community in general.

\section{Prerequisites about recollements}\label{sec:recollements}
The formulation we chose for the family completion theorems is as an equivalence of recollements.
In this section we will introduce some machinery to handle (linear) recollements, and discuss various ways to obtain recollements for the kind of categories we are interested in.

\begin{notation}
Let \(X\) be a stable category.
Following the conventions in \cite{Rezk__WHATARERECOLLEMENTS}, we say a recollement \(\mathcal{R}\) of \(X\) is diagram of categories
\[\begin{tikzcd}
	U & X & Z
	\arrow["{i_*}"{description}, hook, from=1-1, to=1-2]
	\arrow["{j^*}"{description}, two heads, from=1-2, to=1-3]
\end{tikzcd}\]
such that
\begin{enumerate}
  \item \(i_*\) is fully faithful, and both left and right adjoint,
  \item \(j^*\) is a localization, and both left and right adjoint, and
  \item the essential image of \(i_*\) agrees with the kernel of \(j^*\), that is the objects \(x\in X\) such that \(j^*x\to j^*0\) is an equivalence.
\end{enumerate}

  Consider the resulting diagram
\[\begin{tikzcd}
	U & X & Z
	\arrow[""{name=0, anchor=center, inner sep=0}, "{i_*}"{description}, hook, from=1-1, to=1-2]
	\arrow[""{name=1, anchor=center, inner sep=0}, "{i^!}"{description}, shift left, curve={height=-12pt}, two heads, from=1-2, to=1-1]
	\arrow[""{name=2, anchor=center, inner sep=0}, "{i^\#}"{description}, shift right, curve={height=12pt}, two heads, from=1-2, to=1-1]
	\arrow[""{name=3, anchor=center, inner sep=0}, "{j^*}"{description}, two heads, from=1-2, to=1-3]
	\arrow[""{name=4, anchor=center, inner sep=0}, "{j_!}"{description}, shift left, curve={height=-12pt}, hook, from=1-3, to=1-2]
	\arrow[""{name=5, anchor=center, inner sep=0}, "{j_\#}"{description}, shift right, curve={height=12pt}, hook, from=1-3, to=1-2]
	\arrow["\dashv"{anchor=center, rotate=-90}, draw=none, from=0, to=1]
	\arrow["\dashv"{anchor=center, rotate=-90}, draw=none, from=3, to=4]
	\arrow["\dashv"{anchor=center, rotate=-90}, draw=none, from=2, to=0]
	\arrow["\dashv"{anchor=center, rotate=-90}, draw=none, from=5, to=3]
\end{tikzcd}\]
  We will denote the essential images of \(j_\#\), \(i_*\), and \(j_!\) by \(Nil(\mathcal{R})\), \(Loc(\mathcal{R})\), and \(Cpl(\mathcal{R})\) and call their objects \(\mathcal{R}\)-nilpotent, \(\mathcal{R}\)-local, and \(\mathcal{R}\)-complete.
\end{notation}
Let us recall some basic properties of these objects:
\begin{proposition}[{\cite[prop.~1.2]{Rezk__WHATARERECOLLEMENTS}}]\label{prop:basic_recollement_omnibus}\,
  \begin{enumerate}
    \item \(U\) and \(Z\) are stable categories.
    \item The essential image of \(j_\#\) agrees with the kernel of \(i^\#\).
    \item The essential image of \(j_!\) agrees with the kernel of \(i^!\).
    \item All three horizontal sequences are Verdier sequences.
    \item We have equivalences
    \[Loc(\mathcal{R})=Nil(\mathcal{R})^\perp,\,Loc(\mathcal{R})=\,\!^\perp Cpl(\mathcal{R}),\,Nil(\mathcal{R})=\,\!^\perp Loc(\mathcal{R}),\,Cpl(\mathcal{R})=Loc(\mathcal{R})^\perp\]
    of subcategories of \(X\).
  \end{enumerate}
\end{proposition}

\begin{definition}
  We say that two recollements \(\mathcal{R}\) and \(\mathcal{R}'\) of \(X\) are equivalent, and write \(\mathcal{R}\simeq \mathcal{R}'\), if any (and thus all) of the classes of nilpotent, local, or complete objects agree.
\end{definition}

\subsection{Presentable linear recollements}\label{ssec:linear_recollements}
Let us fix a presentably symmetric monoidal stable category \(C\in\CAlg(\PrLst)\).
Given a \(C\)-module \(X\), we are interested in an appropriate notion of \(C\)-linear recollements of \(X\).
On the one hand, it should be a property of a recollement of \(X\) to be \(C\)-linear, on the other hand we would like these to have a good interaction with tensor products over \(C\).

\begin{definition}\label[definition]{def:CLinearRecollement}
  Let \(C\) be a presentably stable symmetric monoidal category, and let \(X\) be a \(C\)-module in \(\PrLst\).
  We call a recollement
\[\begin{tikzcd}
	U & X & Z
	\arrow[""{name=0, anchor=center, inner sep=0}, "{i_*}"{description}, hook, from=1-1, to=1-2]
	\arrow[""{name=1, anchor=center, inner sep=0}, "{i^!}"{description}, shift left, curve={height=-12pt}, two heads, from=1-2, to=1-1]
	\arrow[""{name=2, anchor=center, inner sep=0}, "{i^\#}"{description}, shift right, curve={height=12pt}, two heads, from=1-2, to=1-1]
	\arrow[""{name=3, anchor=center, inner sep=0}, "{j^*}"{description}, two heads, from=1-2, to=1-3]
	\arrow[""{name=4, anchor=center, inner sep=0}, "{j_!}"{description}, shift left, curve={height=-12pt}, hook, from=1-3, to=1-2]
	\arrow[""{name=5, anchor=center, inner sep=0}, "{j_\#}"{description}, shift right, curve={height=12pt}, hook, from=1-3, to=1-2]
	\arrow["\dashv"{anchor=center, rotate=-90}, draw=none, from=0, to=1]
	\arrow["\dashv"{anchor=center, rotate=-90}, draw=none, from=3, to=4]
	\arrow["\dashv"{anchor=center, rotate=-90}, draw=none, from=2, to=0]
	\arrow["\dashv"{anchor=center, rotate=-90}, draw=none, from=5, to=3]
\end{tikzcd}\]
  of the underlying stable category of \(X\) {\it \(C\)-linear} if the essential image of \(i_*\) is closed both under tensors and cotensors by objects in \(C\).
\end{definition}
\begin{remark}
  While the requirement that the essential image of \(i_*\) be closed under tensors seems to us like a rather natural condition, being closed under powers is of a different nature.
  
  One point of justification is that it follows from closure under tensors in the case that \(C\) is generated under colimits by dualizables.
  This also shows that any presentable stable recollement is automatically \(\Sp\)-linear.

  Another point is that, under the assumption that the essential image of \(i_*\) is closed under tensors, being closed under powers is equivalent to the left adjoint \(i^\#\) admitting a \(C\)-linear structure.
\end{remark}
\begin{proposition}\label[proposition]{prop:ConsequencesOfLinearRecollement}
  In the situation of \cref{def:CLinearRecollement}, there is a unique lift of
  \[Z\xrightarrow{j_\#}X\xrightarrow{i^\#}U\]
  to a co-fiber sequence in \(\Mod_C(\PrLst)\), and the lifts of \(j_\#\) and \(i^\#\) are internal left adjoints.
  
  On the other hand, any such co-fiber sequence has an underlying \(C\)-linear recollement.
\end{proposition}
\begin{proof}
  The underlying category of \(X\) is, by assumption, presentable.
  Moreover, the localization endofunctors \(i_*\circ i^\#\) and \(j_\#\circ j^*\) are left adjoints, so accessible, so by  \cite[prop.~5.5.4.2(3)]{Lurie_HTT} both \(U\) and \(Z\) are accessible.
  Clearly they also admit small colimits, so \(U\) and \(Z\) are presentable.

  This gives us a unique lift to \(\PrLst\), and there it is a (co)fiber-sequence by \cite[prop.~A.20]{Ramzi_2024_DualizablePresentable$infty$categories}.

  Since the local objects are closed under cotensors, the nilpotent objects are closed under tensors.
  Since \(j_\#\) is fully faithful, \cref{lem:lifting_monos_into_modules} gives us a unique lift of \(j_\#:Z\to X\) to \(\Mod_C(\PrLst)\) compatible with the \(C\)-module structure on \(X\).
  Then, as the cofiber in \(\PrLst\) of the \(C\)-linear map \(j_\#\), also \(i^\#\) obtains a unique lift to \(\Mod_C(\PrLst)\) compatible with the \(C\)-module structure on \(X\).
  
  Since the forgetful functor \(\Mod_C(\PrLst)\to \PrLst\) is conservative and both left and right adjoint, we see that the resulting sequence is a co-fiber sequence in \(\Mod_C(\PrLst)\).

  Finally, since the local objects are closed under tensors, we find that the projection maps
  \[c\otimes i_*u\to i_* c\otimes u\text{ and }c\otimes j^*x\to j^*c\otimes x\]
  are equivalences, so \(i^\#\) and \(j_\#\) are internal left adjoints in \(\Mod_C(\PrLst)\).

  In the other direction, given a co-fiber sequence in \(\Mod_C(\PrLst)\) with internal left adjoint functors, one easily deduces that the underlying diagram of categories is (the left-most adjoint part of) a \(C\)-linear recollement.
\end{proof}

This allows us to define notions of basechanges and relative tensors of \(C\)-linear recollements:

\begin{lemma}\label[lemma]{lem:CLinearRecollementBasechange}
  Let \(f:C\to D\) be a map in \(\CAlg(\PrLst)\), and let
  \[\mathcal{R}:U\xrightarrow{i_*} X\xrightarrow{j^*}Z\]
  be a sequence in \(\Mod_C(\PrLst)\) that defines a \(C\)-linear recollement.
  Then
  \[f_*\mathcal{R}:D\otimes_C U\xrightarrow{D\otimes_C i_*}D\otimes_C X\xrightarrow{D\otimes_Cj^*}D\otimes_C Z\]
  defines a \(D\)-linear recollement.
\end{lemma}
\begin{proof}
  By \cref{prop:ConsequencesOfLinearRecollement} we know that \(i^\#\) and \(j_\#\) are internal left adjoints in \(Mod_C(Pr^L)\).
  This equips us with adjunctions
  \[D\otimes_C i^\#\dashv D\otimes_C i_*\text{ and }D\otimes_C j_\#\dashv D\otimes_C j^*.\]
  Since tensoring preserves localizations by \cite[cor.~1.46]{Ramzi_2024_DualizablePresentable$infty$categories}, we have that \(D\otimes_C i^\#\) and \(D\otimes_C j^*\) are localizations, so their adjoints are fully faithful.
  
  It remains to show that
  \[D\otimes_C U\xrightarrow{D\otimes_C i_*}D\otimes_C X\xrightarrow{D\otimes_Cj^*}D\otimes_C Z\]
  is a co-fiber sequence in \(\PrLst\).

  Since, by \cref{prop:ConsequencesOfLinearRecollement}, \(U\to X\to Z\) is a cofiber sequence in \(\Mod_C(\PrLst)\) , and both \(D\otimes_C\) and forgetting the \(D\)-module structure preserve colimits, we find it is a cofiber sequence.

  The same argument also shows that
  \[D\otimes_C Z\xrightarrow{D\otimes j_\#}D\otimes_C X\xrightarrow{D\otimes_C i^\#}D\otimes_CU\]
  is a cofiber sequence in \(\PrLst\).
  By the equivalence \((\PrLst)^{op}\simeq \mathrm{Pr}^R_{\mathrm{st}}\) this shows that the right adjoint sequence is a fiber sequence, and we are done.
\end{proof}

\begin{lemma}\label[lemma]{lem:CLinearRecollementTensor}
  Let \(C\in \CAlg(\PrLst)\), let \(Y\) be a \(C\)-module, and let
  and let
  \[\mathcal{R}:U\xrightarrow{i_*} X\xrightarrow{j^*}Z\]
  be a sequence in \(\Mod_C(\PrLst)\) that defines an \(C\)-linear recollement.
  Then
  \[Y\otimes_C\mathcal{R}:Y\otimes_C U\xrightarrow{Y\otimes_C i_*}Y\otimes_C X\xrightarrow{Y\otimes_Cj^*}Y\otimes_C Z\]
  defines a \(C\)-linear recollement.
\end{lemma}
\begin{proof}
  The proof is essentially the same as the one of \cref{lem:CLinearRecollementBasechange}.
\end{proof}

\subsection{Symmetric monoidal recollements and idempotent algebras}\label{ssec:idempotent}
Let \(C\in \CAlg(\PrLst)\).
Then \(C\) is itself a \(C\)-module, and we will refert to \(C\)-linear recollements of \(C\) itself as {\it multiplicative recollements} of \(C\).
We shall now see that these correspond exactly to idempotent \(\mathbb{E}_0\)-algebras in \(C\).

\begin{definition}
  Let \(C\) be a presentable symmetric monoidal stable category, and let \(\mathbb{1}_C\to e\) be an idempotent \(\mathbb{E}_0\)-algebra in \(C\).
  We call an object \(c\in C\)
  \begin{itemize}
      \item \(e\)-nilpotent, if \(e\otimes c\simeq 0\),
      \item \(e\)-local, if \(\mathbb{1}_C\otimes c\to e\otimes c\) is an equivalence, and
      \item \(e\)-complete, if \(\underline{\map}_C(e,c)\simeq 0\).
  \end{itemize}
  If the idempotent algebra is clear from the context, we will often simply call these nilpotent, local, and complete objects.
\end{definition}
\begin{proposition}\label{prop:idempotent_recollement_omnibus}
  Let \(C\) be a presentable symmetric monoidal stable category, and let \(i:\mathbb{1}_C\to e\) be an idempotent \(\mathbb{E}_0\)-algebra in \(C\).
  Then \(e\) admits a unique refinement to an \(\mathbb{E}_\infty\)-algebra in \(C\).
  Moreover, the forgetful functor
  \[i_*:\Mod_e(C)\to C\]
  is fully faithful and extends to a \(C\)-linear recollement
\[\begin{tikzcd}
	{\mathcal{R}(e):\Mod_e(C)} && C && {C/\Mod_e(C)}
	\arrow[""{name=0, anchor=center, inner sep=0}, "{i_*}"{description}, hook, from=1-1, to=1-3]
	\arrow[""{name=1, anchor=center, inner sep=0}, "{i^!}"{description}, curve={height=-18pt}, two heads, from=1-3, to=1-1]
	\arrow[""{name=2, anchor=center, inner sep=0}, "{i^\#}"{description}, curve={height=18pt}, two heads, from=1-3, to=1-1]
	\arrow[""{name=3, anchor=center, inner sep=0}, "{j^*}"{description}, two heads, from=1-3, to=1-5]
	\arrow[""{name=4, anchor=center, inner sep=0}, "{j_!}"{description}, curve={height=-18pt}, hook, from=1-5, to=1-3]
	\arrow[""{name=5, anchor=center, inner sep=0}, "{j_\#}"{description}, curve={height=18pt}, hook, from=1-5, to=1-3]
	\arrow["\dashv"{anchor=center, rotate=-117}, draw=none, from=0, to=1]
	\arrow["\dashv"{anchor=center, rotate=-63}, draw=none, from=2, to=0]
	\arrow["\dashv"{anchor=center, rotate=-69}, draw=none, from=3, to=4]
	\arrow["\dashv"{anchor=center, rotate=-111}, draw=none, from=5, to=3]
\end{tikzcd}\]
having the following properties/admitting the following structures:
\begin{enumerate}
  \item  \(j_\#\), \(i_*\), and \(j_!\) are fully faithful, and their essential images are the \(e\)-nilpotent, \(e\)-local, and \(e\)-complete objects.
  \item \(i^\#\), \(j^*\), and \(i^!\) are Verdier quotients, and their kernels are the \(e\)-nilpotent, \(e\)-local, and \(e\)-complete objects.
  \item There are unique symmetric monoidal structures on \(\Mod_e(C)\) and \(C/\Mod_e(C)\) such that \(i^\#\) and \(j^*\) are symmetric monoidal, and for the induced \(C\)-module structures \(i_*\) and \(j_\#\) are \(C\)-linear
\end{enumerate}
  If \(F:C\to D\) is a map in \(\CAlg(\PrLst)\), then also \(F(e)\) is an idempotent algebra, and we get a unique commutative diagram in \(\PrLst\)
\[\begin{tikzcd}
{\Mod_e(C)} & C & {C/\Mod_e(C)} \\
{\Mod_{F(e)}(D)} & D & {D/\Mod_{F(e)}(D)}
\arrow[hook, from=1-1, to=1-2]
\arrow["F_{loc}"{description},from=1-1, to=2-1]
\arrow[two heads, from=1-2, to=1-3]
\arrow["F"{description}, from=1-2, to=2-2]
\arrow["F_{nil}"{description},from=1-3, to=2-3]
\arrow[hook, from=2-1, to=2-2]
\arrow[two heads, from=2-2, to=2-3]
\end{tikzcd}\]
where both squares are horizontally left-adjointable.

Finally, if \(i_*:U\to C\) is fully faithful, admits both a left adjoint \(i^\#\) and a right adjoint \(i^!\), and its essential image is closed under tensors and cotensors, then the adjunction unit \(\mathbb{1}_C\to i_*i^\#\mathbb{1}_C\) is an idempotent algebra, and the essential image of \(i_*\) agrees with the \((i_*i^\#\mathbb{1}_C)\)-local objects.
\end{proposition}
\begin{proof}
  That \(e\) uniquely lifts to an \(\mathbb{E}_\infty\)-algebra is precisely \cite[prop.~4.8.2.9]{Lurie_HA}.

  That \(i_*\) is fully faithful, and that its essential image consists of the local objects, is precisely \cite[prop.~4.8.2.10]{Lurie_HA}.

  Now \(i_*\) admits both adjoints, since it is a restriction-of-scalars functor, so by \cite[cor.~A.2.11]{CalmesHarpazLandNardinSteimleDottoHebestreitMoiNikolaus_2025_HermitianKtheoryStable} it is a split Verdier inclusion.
  It follows that the Verdier quotient \(j^*\) also admits both adjoints, and that both \((j_\#,i^\#) \) and \((j_!,i^!)\) are Verdier sequences by \cite[cor.~A.2.8]{CalmesHarpazLandNardinSteimleDottoHebestreitMoiNikolaus_2025_HermitianKtheoryStable}.

  Since the adjunctions \(i^\#\dashv i_*\dashv i^!\) arise from the map of algebras \(i:\mathbb{1}_C\to e\), and since we already know that \(i_*\) is fully faithful, the kernels of \(i^\#\) and \(i^!\) agree with the kernels of
  \[i_*\circ i^\#(-)\simeq e\otimes(-)\text{ and }i_*\circ i^!(-)\simeq \underline{\map}_C(e,-),\]
  so they are exactly the nilpotent and complete objects.
  This finishes (1) and (2).

  Since \(i^\#\) is the extension-of-scalars along the unit of an \(\mathbb{E}_\infty\)-algebra and a localization, we find that the class of morphisms inverted by \(i^\#\) fulfills the assumptions of \cite[Pr.A.5]{NikolausScholze_2018_TopologicalCyclicHomology}.
  Thus, this particular symmetric monoidal structure on \(i^\#\) and \(\Mod_e(C)\) is also the unique one.

  Similarly, \cite[Pr.A.5]{NikolausScholze_2018_TopologicalCyclicHomology} also equips \(j^*\) and \(C/\Mod_e(C)\) with a unique symmetric monoidal structure.

  As adjoints of \(C\)-linear functors, \(i_*\) and \(j_\#\) come with canonical (op)lax \(C\)-linear structures.
  Let us verify that these are, in fact, \(C\)-linear:
  For \(i_*\) we need to show that
  \[c\otimes i_*(m)\xrightarrow{\eta}i_*i^\#(c\otimes i_*(m))\xrightarrow{\simeq}i_*(i^\#(c)\otimes i^\#i_*(m))\xrightarrow{\epsilon}i_*(i^\#(c)\otimes m)\]
  is an equivalence,
  where the middle and right morphisms are equivalences by \(i^\#\) being monoidal and \(i^\#\) being a localization.
  The left map, the adjunction unit of \(i^\#\dashv i_*\), is an equivalence exactly on the essential image of \(i_*\), which are the local objects.
  But by definition the local objects are clearly closed under tensors.

  Similarly, for \(j_\#\), we need to show that
  \[j_\#(j^*(c)\otimes m)\xrightarrow{\eta}j_\#(j^*(c)\otimes j^*j_\#(m))\xrightarrow{\simeq}j_\#j^*(c\otimes j_\#(m))\xrightarrow{\epsilon}c\otimes j_\#(m)\]
  is an equivalence.
  The left and middle morphisms are equivalences by \(j^*\) being a localization and \(j^*\) being monoidal.
  The right map, the adjunction counit of \(j_\#\dashv j^*\), is an equivalence exactly on the essential image of \(j_\#\), which are the nilpotent objects.
  But by definition the nilpotent objects are clearly closed under tensors.
  This finishes (3).

  Now, if \(F:C\to D\) is symmetric monoidal, it is clear that \(F(e)\) is again an idempotent algebra, and that \(F\) maps \(e\)-local to \(F(e)\)-local objects.
  This gives the (unique) left and right vertical functors yielding the commutative diagram.
  
  Regarding the adjointability of the squares, denote the recollement functors for \(F(e)\) in \(D\) with \(I\) and \(J\) with the same decorations.
  For the left square, we need to show that the Beck-Chevalley transformation
  \[I^\#\circ F\xrightarrow{I^\#\circ F\circ \eta}I^\#\circ F\circ i_*\circ i^\#\xrightarrow{\simeq}I^\#\circ I_*\circ F_{loc}\circ i^\#\xrightarrow{\epsilon\circ F_{loc}\circ i^\#}F_{loc}\circ i^\#\]

is an equivalence.
The right morphism is an equivalence since \(I^\#\) is a localization, and the middle is the commutativity of the square.
For the left map, we may check that it is an equivalence after applying \(I_*\).
On an object \(c\in C\) this becomes
\[F(e)\otimes F(c)\xrightarrow{F(e)\otimes F(i\otimes\id_c)}F(e)\otimes F(e\otimes c),\]
which is an equivalence since \(F\) is monoidal and \(i:\mathbb{1}_C\to e\) is an idempotent algebra.

For the right square, we need to show that the Beck-Chevalley transformation
\[J_\#\circ F_{nil}\xrightarrow{J_\#\circ F_{nil}\circ \eta}J_\#\circ F_{nil}\circ j^*\circ j_\#\xrightarrow{\simeq}J_\#\circ J^*\circ F\circ j_\#\xrightarrow{\epsilon\circ F\circ j_\#}F\circ j_\#\]
is an equivalence.
The left morphism in an equivalence since \(j^*\) is a localization, and the middle is the commutativity of the square.
For the right map, it suffices to check on objects of the form \(j^*(c)\).
Denote the fiber of \(i:\mathbb{1}_C\to e\) by \(q:x\to \mathbb{1}_C\).
The functor \(j_\#\circ j^*\) is equivalent to \(x\otimes(-)\), so we need to show that
\[F(x)\otimes F(x\otimes c)\xrightarrow{F(q)\otimes F(x\otimes c)}F(x\otimes c)\]
is an equivalence.
This, again, follows from \(F\) being monoidal and \(q:x\to \mathbb{1}_C\) being an idempotent coalgebra.

Finally, assume \(i_*:U\to C\) is fully faithful, admits adjoints \(i^\#\dashv i_*\dashv i^!\), and its essential image is closed under tensors and cotensors.
Then \cite[cor.~A.2.11]{CalmesHarpazLandNardinSteimleDottoHebestreitMoiNikolaus_2025_HermitianKtheoryStable} tells us that \(i_*\) is a split Verdier inclusion, so \(i^\#\) is a Verdier quotient with kernel the left orthogonal complement of \(\mathrm{im}(i_*)\).
Since \(\mathrm{im}(i_*)\) is closed under cotensors, its left orthogonal complement is closed under tensors.
Thus, the localization \(i^\#\) fulfills the assumptions of \cite[prop.~A.5]{NikolausScholze_2018_TopologicalCyclicHomology}, so \(i^\#\) and \(U\) admit a unique symmetric monoidal structure.
Arguing as for (3) we find that \(i_*\) is then \(C\)-linear, so the localization endofunctor \(i_*\circ i^\#\) is \(C\)-linear.
It follows that \(\eta:\id_C\implies i_*\circ i^\#\) is equivalent to tensoring with \(\eta_{\mathbb{1}_C}\), so we conclude.
\end{proof}

We can give rather explicit characterizations of the nilpotents, locals, and completes associated to basechanges and tensors of \(\mathcal{R}(e)\):
\begin{lemma}\label[lemma]{lem:characterizeNLCOfIdempotentRecollement}
  Let \(C\in \CAlg(\PrLst)\) and let \(e\) be an idempotent \(\mathbb{E}_0\)-algebra in \(C\).
  For any \(C\)-module \(M\), we have the following characterizations \(M\otimes_C \mathcal{R}(e)\):
  \begin{itemize}
    \item \(Nil(M\otimes_C \mathcal{R}(e))=\{m\in M\mid e\otimes m\simeq 0\}\),
    \item \(Loc(M\otimes_C \mathcal{R}(e))=\{m\in M\mid \text{the unit }\mathbb{1}_C\otimes m\to e\otimes m\text{ is an equivalence}\}\), and
    \item \(Cpl(M\otimes_C \mathcal{R}(e))=\{m\in M\mid m^e\simeq0\}\).
  \end{itemize}
\end{lemma}
\begin{proof}
  We will start by identifiying the class of locals.
  By construction, it is the essential image of the localization endofunctor
  \[(\id_M\otimes i_*)\circ (\id_M\otimes i^\#)\]
  which sends \(m\) to \(e\otimes m\), and the natural transformation from \(\id_M\simeq \id_M\otimes \id_C\) is given by
  \[\mathbb{1}_C\otimes m\to e\otimes C.\]
  Thus, the class of locals is exactly as claimed.

  The identifications of the nilpotent and complete objects then follow from point (5) of \cref{prop:basic_recollement_omnibus}.
\end{proof}

\begin{corollary}\label[corollary]{cor:idempotent_recollement_basechange}
  Let \(f:C\to D\) be a map in \(\CAlg(\PrLst)\), and let \(e\in C\) be an idempotent \(\mathbb{E}_0\)-algebra.
  Then we have an equivalence of recollements
  \[f_*\mathcal{R}(e)\simeq \mathcal{R}(f(e)).\]
\end{corollary}

\subsection{Recollements associated to ideals}\label{ssec:ideal_recollement}
Let \(A\) be a commutative algebra in spectra, and let \(I\subset \pi_0A\) be a finitely generated ideal.
For any \(\Mod_A(\Sp)\)-module \(M\), Lurie defines subcategories of \(I\)-nilpotent, \(I\)-local, and \(I\)-complete objects of \(M\).
Let us recall these now:
\begin{construction}\label[construction]{cons:ICompletionRecollement}
  Let \(A\in \CAlg(\Sp)\), let \(I\subset\pi_0A\) be a finitely generated ideal, and let \(M\) be a \(\Mod_A(\Sp)\)-module.
  Recall the definition of the following full subcategories of \(M\):
  \begin{itemize}
    \item \(M^{Nil(I)}=\{m\in M\mid A[x^{-1}]\otimes m\simeq 0\text{ for all }x\in I\}\),
    \item \(M^{Loc(I)}=\nb{M^{Nil(I)}}^\perp\), and
    \item \(M^{Cpl(I)}=\nb{M^{Loc(I)}}^\perp\).
  \end{itemize}
  From \cite[prop.~7.2.4.4, prop.~7.3.1.4]{Lurie_SAG} it then follows that this assembles to a presentable stable recollement
  \[\mathcal{R}(I,M):M^{Loc(I)}\xrightarrow{}M\xrightarrow{\Gamma_I^M}M^{Nil(I)}.\]
  By \cite[prop.~7.2.4.9(3)]{Lurie_SAG} the \(I\)-local objects of \(M\) are closed under tensors.
  By \cite[prop.~7.1.1.12(b)]{Lurie_SAG} the \(I\)-nilpotent objects in \(M\) are closed under tensors, so the \(I\)-local objects in \(M\) are closed under cotensors.

  Thus, \(\mathcal{R}(I,M)\) is a \(\Mod_A(\Sp)\)-linear recollement of \(M\).
  In the case that \(M=\Mod_A(\Sp)\), we will denote this recollement simply by \(\mathcal{R}(I)\), and the idempotent algebra from \cref{prop:idempotent_recollement_omnibus} by \(L_IA\).

  Note that, by (the proof of) \cite[prop.~7.1.1.5]{Lurie_SAG}, these notions only depend on the radical of \(I\).
  In particular, these constructions only need that the radical of \(I\) is the radical of a finitely generated ideal.
\end{construction}
\begin{lemma}\label[lemma]{lem:ICompletionIsTensoredUp}
  Let \(A\in \CAlg(Sp)\), let \(I\subset\pi_0A\) be a finitely generated ideal, and let \(M\) be a \(\Mod_A(\Sp)\)-module.
  Then we have an equivalence of linear recollements
  \[\mathcal{R}(I,M)\simeq \mathcal{R}(I)\otimes_{\Mod_A(\Sp)}M.\]
\end{lemma}
\begin{proof}
  Since both are, by construction, linear recollements of \(M\), it will suffice to show that their nilpotent objects agree.
  We shall do this by an induction on the number of generators of \(I\).

  First, consider the case that \(I=(x)\).
  Then, by \cref{lem:characterizeNLCOfIdempotentRecollement}, \(\mathcal{R}(I)\) agrees with the recollement \(\mathcal{R}(A[x^{-1}])\) of \cref{prop:idempotent_recollement_omnibus}.
  Again by \cref{lem:characterizeNLCOfIdempotentRecollement} we thus find that the \(I\)-nilpotents in \(M\) are exactly the nilpotents of \(\mathcal{R}(I)\otimes_{\Mod_A(\Sp)}M\).

  Now, assume that we have already proven the statement for \(I\), and let us extend it to \(I+(x)\).
  By our inductive hypothesis and \cite[rem.~7.1.2.4]{Lurie_SAG} we have
  \[M^{Nil(I+(x))}=\nb{M^{Nil(I)}}^{Nil((x))}=\nb{M\otimes_{\Mod_A(\Sp)} Nil(\mathcal{R}(I))}\otimes_{\Mod_A(\Sp)} Nil(\mathcal{R}((x))).\]
  Using associativity and the same trick backwards, we obtain
  \[M^{Nil(I+(x))}=M\otimes_{\Mod_A(\Sp)}Nil(\mathcal{R}(I+(x)))\]
  and we are done.
\end{proof}
\begin{lemma}\label[lemma]{lem:ICompletionAndRestrictingScalars}
  Let \(A\to B\) be a map of commutative algebras in spectra, \(I\subset \pi_0 A\) a finitely generated ideal, \(M\) an \(\Mod_B(\Sp)\)-module, and denote by \(f^*M\) the \(\Mod_A(\Sp)\)-module obtained by restriction along \(f\).
  Then we have an equivalence of recollements
  \[\mathcal{R}(I,f^*M)\simeq \mathcal{R}(f(I)\pi_0B,M).\]
\end{lemma}
\begin{proof}
  It suffices to show that the classes nilpotents agree, which is the case by \cite[rem.~7.1.1.10]{Lurie_SAG}.
\end{proof}

Let us record a source of \(\Mod_A(\Sp)\)-modules which we will frequently use:

\begin{proposition}\label[proposition]{prop:linear_over_endomorphisms}
  Let \(C\) be a presentably stable symmetric monoidal catgeory.
  Then \(C\) admits a canonical \(\Mod_{\mathrm{end}_C(\mathbb{1}_C)}(\Sp)\)-algebra structure, and the functor \(\map_C(\mathbb{1}_C,\bullet)\) comes with a lax \(\Mod_{\mathrm{end}_C(\mathbb{1}_C)}(\Sp)\)-linear factorization through \(\Mod_{\mathrm{end}_C(\mathbb{1}_C)}(\Sp)\).
  If the unit in \(C\) is compact, then the functor is even \(\Mod_{\mathrm{end}_C(\mathbb{1}_C)}(\Sp)\)-linear.
\end{proposition}
\begin{proof}
  By \cite[thm.~4.8.5.11, thm.~4.8.5.16(4)]{Lurie_HA}, there is a factorization
\[\begin{tikzcd}
	\Sp & C \\
	& {\Mod_{\mathrm{end}_{C}(\mathbb{1}_C)}(\Sp)}
	\arrow[from=1-1, to=1-2]
	\arrow[from=1-1, to=2-2]
	\arrow[dashed, from=2-2, to=1-2]
\end{tikzcd}\]
  in \(\CAlg(\PrLst)\), providing the algebra structure on \(C\).
  Looking at right adjoints provides the factorization of \(\map_C(\mathbb{1}_C,\bullet)\).
  This makes the dashed arrow above tautologically \(\Mod_{\mathrm{end}_{C}(\mathbb{1}_C)}(\Sp)\)-linear, which equips its right adjoint with a lax \(\Mod_{\mathrm{end}_C(\mathbb{1}_C)}(\Sp)\)-linear structure.

  Finally, if the unit \(\mathbb{1}_C\) is compact, \cite[obs.~4.51]{Ramzi_2026_LocallyRigid$infty$categories} tells us that the lax linear structure of the right adjoint is actually linear.
\end{proof}
\subsection{Recollements associated to families of subgroups}\label{ssec:family_recollement}

A key example of recollements in equivariant stable homotopy theory comes from families of subgroups:
Let \(G\) be a compact Lie group and let \(\mathcal{F}\) be a family of closed subgroups of \(G\).
Define an \(\mathbb{E}_0\)-algebra \(\widetilde{E\mathcal{F}}\) in \((\Spc_G)_*\) by
\[\widetilde{E\mathcal{F}}^H=\begin{cases}S^0 & H\notin\mathcal{F}\\ * & H\in \mathcal{F}\end{cases}\]
and unit map \(\eta:S^0\to \widetilde{E\mathcal{F}}\) pointwise either the identity on \(S^0\) or the unique map \(S^0\to *\).
Denote the suspension spectrum by \(\widetilde{E\mathcal{F}}\in\Sp_G\), as well, and let \(E\mathcal{F}_+\) be the fiber of \(\eta:\ss_G\to \widetilde{E\mathcal{F}}\) in \(\Sp_G\).

\(\widetilde{E\mathcal{F}}\) is an idempotent \(\mathbb{E}_0\)-algebra in \(\mathrm{Sp}_G\), so we obtain the stable recollement \(\mathcal{R}(\widetilde{E\mathcal{F}})\):
\[\begin{tikzcd}
	{\mathrm{Mod}_{\widetilde{E\mathcal{F}}}=\mathrm{Sp}^{\Phi\mathcal{F}}_G} && {\mathrm{Sp}_G} && {\mathrm{Sp}^{h\mathcal{F}}_G} \\
	\\
	\\
	&& {}
	\arrow[""{name=0, anchor=center, inner sep=0}, "{i_*}"{description}, hook, from=1-1, to=1-3]
	\arrow[""{name=1, anchor=center, inner sep=0}, "{i^*}"{description}, curve={height=18pt}, from=1-3, to=1-1]
	\arrow[""{name=2, anchor=center, inner sep=0}, "{i^!}"{description}, curve={height=-18pt}, from=1-3, to=1-1]
	\arrow[""{name=3, anchor=center, inner sep=0}, "{j^*}"{description}, from=1-3, to=1-5]
	\arrow[""{name=4, anchor=center, inner sep=0}, "{j_!}"{description}, curve={height=18pt}, hook', from=1-5, to=1-3]
	\arrow[""{name=5, anchor=center, inner sep=0}, "{j_*}"{description}, curve={height=-18pt}, hook', from=1-5, to=1-3]
	\arrow["\dashv"{anchor=center, rotate=-116}, draw=none, from=0, to=2]
	\arrow["\dashv"{anchor=center, rotate=-88}, draw=none, from=3, to=5]
	\arrow["\dashv"{anchor=center, rotate=-64}, draw=none, from=1, to=0]
	\arrow["\dashv"{anchor=center, rotate=-92}, draw=none, from=4, to=3]
\end{tikzcd}\]

For \(A\in \CAlg(\Sp_G)\) we will say that an \(A\)-module is \(\mathcal{F}\)-nilpotent, \(\mathcal{F}\)-complete, or \(\mathcal{F}\)-local if its underlying \(G\)-spectrum is.
Since the underlying functor \(\Mod_A(\Sp_G)\to \Sp_G\) is conservative and commutes with both tensors and powers by objects in \(\Sp_G\), \cref{lem:characterizeNLCOfIdempotentRecollement} tells us that these are the nilpotents, locals, and completes of the recollement \(\Mod_A(\Sp_G)\otimes_{\Sp_G}\mathcal{R}(\widetilde{E\mathcal{F}})\).

\begin{remark}
  Note that our use of '\(\mathcal{F}\)-nilpotent' does {\it not} agree with the notion discussed in \cite[§~6.6]{MathewNaumannNoel_2017_NilpotenceDescentEquivariant}.
  Later we will want to also talk about their notion \(\mathcal{F}\)-nilpotence, and we will always explicitly state if we mean their notion.
\end{remark}

\begin{lemma}\label[lemma]{lem:generatorsForFNilpotents}
  The fullsubcategory of \(\mathcal{F}\)-nilpotent \(G\)-spectra is, as a localizing stable subcategory, generated by \(G/H_+\) for \(H\in\mathcal{F}\).
\end{lemma}
\begin{proof}
  As the above is a stable recollement, we can identify the \(\mathcal{F}\)-nilpotent objects with the kernel of \(i^*\).
  Since \(i_*\) is fully faithful we can further identify this as the kernel of \(i_*\circ i^*\).
  This endofunctor is given by tensoring with \(\widetilde{E\mathcal{F}}\).

  Let us first show that the spectra \(G/H_+\) are \(\mathcal{F}\)-nilpotent for \(H\in\mathcal{F}\).
  Applying the projection formula we have
  \[G/H_+\otimes \widetilde{E\mathcal{F}}=\mathrm{Ind}_H^G\nb{\mathrm{Res}_H^G\nb{\widetilde{E\mathcal{F}}}}=0\]
  since all fixed points of \(\widetilde{E\mathcal{F}}\) for subgroups of \(H\) are contractible by definition.

  Now assume that \(X\) is \(\mathcal{F}\)-nilpotent and that
  \(\mathrm{map}(G/H_+, X)=X^H=0\text{ for all }H\in\mathcal{F}\).
  To show that the orbits \(G/H_+\) for \(H\in \mathcal{F}\) generate the \(\mathcal{F}\)-nilpotent objects, it suffices to show that these assumptions already imply \(X=0\).

  Since \(\mathcal{F}\) is closed under subgroups we can conclude that \(\mathrm{Res}^G_H(X)=0\) for all \(H\in\mathcal{F}\), and then that \(X^{\Phi H}=0\) for all \(H\in\mathcal{F}\).
  For \(H\notin\mathcal{F}\) we can then use that \(X\) is \(\mathcal{F}\)-nilpotent:
  \[X^{\Phi H}=X^{\Phi H}\otimes\ss=X^{\Phi H}\otimes \widetilde{E\mathcal{F}}^{\Phi H}=\nb{X\otimes \widetilde{E\mathcal{F}}}^{\Phi H}=0.\]

  As the geometric fixed point functors are jointly conservative we deduce that \(X=0\).
\end{proof}

We now want to connect this with the recollements considered in \cref{cons:ICompletionRecollement}, at least for \(G\) a finite group.
\begin{definition}
  Let \(G\) be a finite group and let \(\mathcal{F}\) be a family of subgroups of \(G\).
  For \(A\in \CAlg(\Sp_G)\) we define an ideal in \(\pi_0A^G\) by
\[I_A(\mathcal{F})=\bigcap_{H\in \mathcal F}I^G_H=\ker(r^G_H:\pi_0A^G\to \pi_0A^H).\]
\end{definition}

The endomorphism spectrum of \(A\) in \(\Mod_A(\Sp_G)\) is equivalent to \(A^G\in\CAlg(\Sp)\), so \cref{prop:linear_over_endomorphisms} gives us that \(\Mod_A(\Sp_G)\) is canonically \(\Mod_{A^G}(\Sp)\)-linear.
Assuming \(I_A(\mathcal{F})\) is finitely generated, this allows us to talk about \(I_A(\mathcal{F})\)-complete \(A\)-modules:

\begin{proposition}\label[proposition]{prop:FCompleteImpliesIFComplete}
  Let \(G\) be a finite group, \(\mathcal{F}\) a family of subgroups of \(G\), \(A\) a commutative algebra in \(G\)-spectra, and assume that \(I_A(\mathcal{F})\subset \pi_0 A^G\) is finitely generated (up to radicals).
  Let \(M\) be an \(\mathcal{F}\)-complete \(A\)-module.
  Then \(M\) is \(I_A(\mathcal{F})\)-complete.
\end{proposition}
\begin{proof}
  By \cite[cor.~7.3.2.2]{Lurie_SAG} we need to show that, for every \(x\in I_A(\mathcal{F})\), the limit \(T\) of
  \[M\xleftarrow{x}M\xleftarrow{x}M\xleftarrow{x}\cdots\]
  vanishes.
  By the mapping description of \(\mathcal{F}\)-completion we have
  \[T=\underline{\map}_A\nb{(A\otimes E\mathcal{F}_+)[x^{-1}],M}.\]
  We claim that already \(R=(A\otimes E\mathcal{F}_+)[x^{-1}]=0\).
  By construction we have that \(R^{\Phi K}=0\) for \(K\notin \mathcal{F}\).
  For \(K\in \mathcal{F}\) we have \(x\in I^G_K\).
  This means that, after restriction to \(K\), \(x\) equals \(0\), so \(\mathrm{Res}^G_K R =0\), thus \(R^{\Phi K}=0\).
  Since all geometric fixed points of \(R\) vanish, \(R\) itself must be zero, and we are done. 
\end{proof}

\subsection{Recollements associated to open immersions of stacks}\label{ssec:open_immersions}
Even in the affine case, our proof of the family completion theorem for tempered cohomology will crucially employ the language of stacks as introduced in \cite{BalderramaDaviesLinskens_2025_AffinenessReconstructionComplexperiodic}.
See their paper for basic definitions, and appendix~\ref{sec:stacks} for some further results.
In particular, recall (\cref{def:open_immersion_stk}) that a map \(f:X\to Y\) of stacks is a (quasi-compact) open immersion if \(f\) is a \(|-|\)-Cartesian map, where
\[|-|:\Stk\to \Top\]
is the underlying topological space functor, and \(|f|:|X|\to |Y|\) is a (quasi-compact) open embedding of topological spaces.

\begin{definition}\label[definition]{def:locally_quasicompact_open_immersion}
  We call a map of stacks \(i:U\to X\)  a {\it locally quasi-compact open immersion} if there exists an effective epimorphism
  \[Y\to X\]
  such that the basechange \(U_Y\to Y\) is a quasi-compact open immersion (see \cref{def:open_immersion_stk}).
\end{definition}

\begin{lemma}\label[lemma]{lem:lqc_open_immersion_recollement}
  Let \(i:U\to X\) be an locally quasi-compact open immersion.
  Then \(i\) is universally \(0\)-affine, the pushforward \(i_*\) is fully faithful, and the essential image of \(i_*\) is closed under both tensors and cotensors with objects in \(\QCoh(X)\).
\end{lemma}
\begin{proof}
  The claim that \(i\) is universally \(0\)-affine follows from \cref{lem:universal_zero_affine_local_finitary} and the fact that quasi-compact open immersions are quasi-affine, thus universally \(0\)-affine by \cite[thm.~3.3.2.1]{BalderramaDaviesLinskens_2025_AffinenessReconstructionComplexperiodic}.

  Let \(Y\to X\) be an effective epimorphism such that the basechange \(V\to Y\) is a quasi-compact open immersion.
  Choosing a small colimit presentation \(Y\simeq \colim_{j\in J}\Spec R_j\) we see that the further basechange
  \[\coprod_{j\in J} V_j \to \coprod_{j\in J}\Spec R_j\]
  is a quasi-compact open immersion, too.
  In particular, each \(V_i\to \Spec R_i\) is a quasi-compact open immersion.
  This gives us pullback squares
\[\begin{tikzcd}
	{V_j} & {\Spec R_j} \\
	U & X
	\arrow["{i_j}"{description}, from=1-1, to=1-2]
	\arrow["{\iota_j^U}"{description}, from=1-1, to=2-1]
	\arrow["\lrcorner"{anchor=center, pos=0.125}, draw=none, from=1-1, to=2-2]
	\arrow["{\iota_j}"{description}, from=1-2, to=2-2]
	\arrow["i"{description}, from=2-1, to=2-2]
\end{tikzcd}\]
  which are adjointable by \cite[prop.~2.2.2.5(1)]{BalderramaDaviesLinskens_2025_AffinenessReconstructionComplexperiodic}.
  That \(i_*\) is fully faithful is equivalent to the adjunction counit \(\epsilon_M^i:i^*i_*M\to M\) being an equivalence for each \(M\in\QCoh(U)\).
  Since the map \(\coprod_{j\in J}\Spec R_j\to X\) is an effective epimorphism, the functors \((\iota_j)^*\) are jointly conservative.
  Thus, it is enough to show that each
  \[(\iota_j)^*\epsilon_M^i:(\iota_j)^*i^*i_*M\to (\iota_j)^* M\]
  is an equivalence.
  The adjointability of the square above identifies this with the adjunction counit
  \[\epsilon^{i_j}_{(\iota^U_j)^*M}:(i_j)^*(i_j)_*(\iota^U_j)^*M\to (\iota^U_j)^*M.\]
  But \((i_j)_*\) is fully faithful by \cite[cor.~2.4.1.6]{Lurie_SAG}, so \(i_*\) is fully faithful.

  It remains to show that the essential image of \(i_*\) is closed under tensors and cotensors by objects in \(\QCoh(X)\).
  The projection formula for \(i_*\) immediately shows that the essential image of \(i_*\) is closed under tensors with objects in \(\QCoh(X)\).
  For cotensors, consider the chain of natural equivalences of mapping spaces
  \begin{align*}
    \QCoh(X)(T,F_X(N,i_*M))&\simeq \QCoh(X)(T\otimes N,i_* M)\simeq \QCoh(U)(i^* (T\otimes N),M)\\
    &\simeq \QCoh(U)(i^*T\otimes i^*N,M)\simeq \QCoh(U)(i^*T,F_U(i^*N,M))\\
    &\simeq \QCoh(X)(T,i_* F_U(i^*N,M)),
  \end{align*}
  where $F_{Y}(-,-)$ denotes the internal homomorphism object in $\QCoh(Y)$ and where we have used that $i^\ast$ is strong symmetric monoidal. The Yoneda lemma then identifies \(F_X(N,i_*M)\simeq i_* F_U(i^*N,M)\), which shows that $F_X(N,i_\ast M)$ lies in the essential image of \(i_*\).
\end{proof}

\begin{corollary}\label{cor:qc_open_immersion_idempotent}
  Let \(i:U\to X\) be a locally quasi-compact open immersion of stacks.
  Then \(\O_X\to i_*\O_U\) is an idempotent \(\mathbb{E}_0\)-algebra, and the essential image of \(i_*\) agrees with the local objects of the associated recollement \(\mathcal{R}(i_*\O_U)\).
\end{corollary}
\begin{proof}
  This holds for the local side of a multiplicative recollement  by \cref{prop:idempotent_recollement_omnibus}, of which the inclusion
  \[i_*:\QCoh(U)\to\QCoh(X)\]
  is an example by \cref{lem:lqc_open_immersion_recollement}.
\end{proof}

\begin{remark}
  If \(X=\Spec R\) is an affine stack, then a locally quasi-compact immersion \(i:U\to X\) is determined by a quasi-compact open subset \(|U|\subset |\Spec R|=|\Spec \pi_0 R|\).
  Then \(|U|\) is the complement of \(V(I)\) for some finitely generated ideal \(I\subset \pi_0 R\), and the recollement \(\mathcal{R}(i_*\O_U)\) of \(\QCoh(X)\simeq \Mod_R(\Sp)\) agrees with the recollement \(\mathcal{R}(I)\) considered in \cref{cons:ICompletionRecollement}.
\end{remark}

\section{Tempered cohomology}\label{sec:tempered}
Tempered cohomology, introduced by Lurie in \cite{Lurie_2019_EllipticCohomologyIII}, associates to a pre-oriented \(\mathbb{P}\)-divisible group \(\gg\) over an \(\einfty\)-ring \(R\) a limit preserving functor
\[R_\gg^\bullet:(\Spc^\gl_{\fin})^\op\to \CAlg(\Sp)_{R/}.\]
In the case of \(\mu_{\mathbb{P}^\infty}\) over \(\KU\) this recovers equivariant topological \(K\)-theory, for finite groups of equivariance.
Over all, the construction of \(R_\gg^\bullet\) mimics the kind of structures one observes in the cohomology theory represented by \(\KU\in\CAlg(\Sp_\fin^\gl)\).

Especially well-behaved is the case where \(\gg\) is oriented and \(R\) is noetherian.
Under these conditions, Gepner, Linskens, and Pol (\cite{GepnerLinskensPol_2024_Global2ringsGenuine}) showed that the cohomology funnctor \(R_\gg^\bullet\) is represented by a commutative algebra \(R(\gg)\in\CAlg(\Sp_\fin^\gl)\), crucially using Lurie's theory of tempered local systems (\cite[§~7]{Lurie_2019_EllipticCohomologyIII}).

Let us now recall some of the main constrcutions involved in this story.
Throughout this section we will fix a noetherian \(\einfty\)-ring \(R\) and an oriented \(\mathbb{P}\)-divisible group \(\gg\) over \(\Spec R\).
We will denote the associated global cohomology theory of \cite[cons.~4.0.3]{Lurie_2019_EllipticCohomologyIII} by
\[R_\gg^\bullet:(\Spc_\fin^\gl)^\op\to \CAlg(\Sp)_{R/},\]
and the associated commutative algebra in \(\fab\)-global spectra, as in \cite[thm.~E]{GepnerLinskensPol_2024_Global2ringsGenuine}, as well as its coinduction to a commutative algebra in \(\fin\)-global spectra, by \(R(\gg)\).

By construction \(R^\bullet_\gg\) is right Kan extended from its restriction to
\[\Glo_\fab^\op\subset\Glo_\fin^\op\subset (\Spc_\fin^\gl)^\op,\]
and on there it is naturally equivalent to the functor
\[\Glo_\fab^\op\to \CAlg_R,\,\bb A\mapsto R(\gg)^A.\]
\begin{construction}
  Denote the left Kan extension of
  \[\Glo_\fab\xrightarrow{R_\gg^\bullet}(\CAlg(\Sp)_{R/})^\op\xrightarrow{\Spec}\Stk_{/\Spec R}\]
  along \(\Glo_\fab\subset \Glo_\fin\) by
  \[\gg\{\bullet\}:\Glo_\fin\to \Stk_{\Spec R}.\]
  For a finite group \(G\) we call \(\gg\{\bb G\}\) the {\it tempered character stack} (of \(G\) with respect to \(\gg\)).

  Since the global sections functor \(\Gamma:\Stk^\op\to \CAlg(\Sp)\) preserves limits, and since its left adjoint \(\Spec\) is fully faithful, we obtain an equivalence from \(\Gamma\circ \gg\{\bullet\}\) to the restriction of \(R_\gg^\bullet\) to \(\Glo_\fin^\op\).

  Since we Kan extended from an inclusion, we get that, for abelian groups \(A\), the affinization map
  \[\gg\{\bb A\}\to \Spec \Gamma\gg\{\bb A\}\simeq \Spec R_\gg^{\bb A}\]
  is naturally equivalent to the canonical natural transformation defining the Kan extension, and is an equivalence.

  Finally, we will repeatedly use the pointwise formula for the left Kan extension together with the fact that
  \[\Orb_\ab(G)\to (\Glo_\fab)_{/\bb G},\,G/H\mapsto (\bb H\to \bb G)\]
  is colimit final (\cite[lem.~4.2]{GepnerMeier_EquivariantTMF1}) to obtain
  \[\gg\{\bb G\}\simeq \colim_{G/H\in \Orb_\ab(G)}\Spec R_\gg^{\bb H}.\]
\end{construction}

\subsection{Theory of support}\label{ssec:support}
A key input to the proof of the completion theorem is a comparison of the \(I_{R(\gg)_G}(\mathcal{F})\)-adic and \(I_{R(\gg)_H}(\mathcal{F}\vert_H)\)-adic topologies on \(\pi_0R(\gg)^H\).
Doing this directly turns out to be quite hard: only for abelian groups \(H\) do we have a good understanding of the structure of \(R(\gg)^H\).
For general finite groups \(G\) the limit defining \(R(\gg)^G\) from the fixed points with respect to the abelian subgroups of \(G\) might create 'new' prime ideals.

Instead, we will develope a theory of support for points in \(|\gg\{\bb G\}|\) (which, by constrcution, all come from \(|\gg\{\bb H\}|\) with \(H\subset G\) abelian) to study their behaviour under restriction, inspired by Segal's work on the prime spectra of complex representation rings \cite{Segal_1968_RepresentationringCompactLie}. 
Then we identify the preimage of \(V(I_{R(\gg)_G}(\mathcal{F}))\) under the affinization map
\[|a|:|\gg\{\bb G\}|\to |\Spec R(\gg)^G|\]
with the points whose support is contained in \(\mathcal{F}\).
We shall later see that, for the completion theorem, controlling this preimage suffices.

\begin{definition}
  Let \(G\) be a finite group, and let \(x\) be a point in \(|\gg\{\bb G\}|\).
  We will say that \(x\) {\it comes from} a subgroup \(H\subset G\) if
  \[x\in \im |\gg\{\bb H\}|\to |\gg\{\bb G\}|.\]
  We define the {\it support} of \(x\) to be the set
  \[supp(x)=\{H\subset G\mid x\text{ comes from }H\text{, but not from a proper subgroup of }H\}.\]
\end{definition}

\begin{lemma}\label[lemma]{lem:reduceSupportToCompleteLocalRings}
  Let G be a finite group, let \(x\) be a point in \(|\gg\{\bb G\}|\), and let \(\mm\) be the image of \(x\) in \(|\Spec R|\).
  Denote the completed local ring of \(R\) at \(\mm\) by \(\widehat{R}\).
  Then \(x\) unqiely lifts to a point \(x'\) of \(|\gg_{\widehat{R}}\{\bb G\}|\), and \(supp(x)=supp(x')\).
\end{lemma}
\begin{remark}
  By construction, we have
  \[\gg\{\bb G\}=\colim_{G/H\in \Orb_{\ab}(G)}\gg\{\bb H\}.\]
  Since the underlying topological space functor, by construction, commutes with colimits, every point comes from an {\it abelian} subgroup of \(G\).
  In particular, the support always consists of abelian subgroups.
\end{remark}

\begin{proof}
  Let \(H\subset G\) be an abelian subgroup.
  Since \(R\to \widehat{R}\) is flat, we have by \cite*[cor.~4.7.3]{Lurie_2019_EllipticCohomologyIII} that the diagram of classical schemes
\[\begin{tikzcd}
	{\gg\{\bb H\}^\heartsuit} & {\gg_{\widehat{R}}\{\bb H\}^\heartsuit} \\
	{\spec\,\pi_0R} & {\spec\,\pi_0\widehat{R}}
	\arrow[from=1-1, to=2-1]
	\arrow[from=1-2, to=1-1]
	\arrow[from=1-2, to=2-2]
	\arrow[from=2-2, to=2-1]
\end{tikzcd}\]
  is a pullback.
  By \cite[\href{https://stacks.math.columbia.edu/tag/01JT}{Tag 01JT}]{stacks-project} we then have that the top horizontal map induces a bijection of points over \(\mm\) and \(\mm\pi_0\widehat{R}\).
  The set \(|\gg\{\bb G\}|\) is a quotient of the set
  \[\coprod_{H\subset G,\,H\text{ abelian}}|\gg\{\bb H\}|\]
  by a fiberwise equivalence relation over  \(|\spec\,R|\), and similarly for \(|\gg_{\widehat{R}}\{\bb G\}|\) over \(|\spec\,\widehat{R}|\).
  Furthermore, the comparison map reflects the equivalence relation, so we also have that
  \[\gg_{\widehat{R}}\{\bb G\}\to\gg\{\bb G\}\]
  induces a bijection between points over \(\mm\) and \(\mm\pi_0\widehat{R}\).
  Thus, there is a unique lift \(x'\) of \(x\).

  Now, assume that \(x\) comes from \(H\subset G\).
  Then it will necessarily be the image of a point \(y\) that lies over \(\mm\), too, so by the same argument there is a unique lift \(y'\).
  The image of \(y'\) will then be a lift of \(x\), and thus \(x'\) by uniqueness, so also \(x'\) comes from \(H\).

  In the other direction, lets say \(x'\) is the image of some \(y'\).
  Then simply applying the horizontal maps shows that also \(x\) comes from \(H\).

  Since \(x\) and \(x'\) come from the same subgroups, also the sets of minimal such subgroups must agree, and we are done.
\end{proof}

\begin{lemma}\label[lemma]{lem:supportInAbelianGroupsIsInjStack}
  Let \(H\) be a finite abelian group and let \(x\) be a point in
  \[|\gg\{\bb H\}|=|\gg[\hat{H}]|.\]
  Then \(S\subset H\) is in the support of \(x\) if and only if
  \[x\in \im |Inj(\hat{S},\gg)|\to |\gg[\hat{H}]|,\]
  where \(Inj\) denotes the stack of injections of \cite[def.~5.5.10]{BalderramaDaviesLinskens_2026_AmbidextrousGlobalSpectra}.
\end{lemma}
\begin{proof}
  This follows from the definition of the stack of injections, and from the fact that for any inclusion of finite abelian groups \(A\subset B\), the morphism
  \(\gg[\hat{A}]\to\gg[\hat{B}]\)
  is a closed immersion between affines, so injective on points.
\end{proof}

\begin{lemma}\label[lemma]{lem:temperedAbelianSupportStaysSupport}
  Let \(V\subset W\) be finite abelian groups, and let \(x\) be a point in \(|\gg\{\bb V\}|\).
  If \(S\subset V\) is in the support of \(x\), then it is also in the support of \((|\gg\{\bb V\}|\to|\gg\{\bb W\}|)(x)\).
\end{lemma}
\begin{proof}
  By \cref{lem:supportInAbelianGroupsIsInjStack} we know that \(x\) is in the image of
  \[|Inj(\widehat{S},\gg)|\to|\gg\{\bb V\}|.\]
  Then \((|\gg\{\bb V\}|\to|\gg\{\bb W\}|)(x)\) is in the image of
  \[|Inj(\widehat{S},\gg)|\to|\gg\{\bb V\}|\to|\gg\{\bb W\}|,\]
  so \(S\) is in the support of \((|\gg\{\bb V\}|\to|\gg\{\bb W\}|)(x)\) by \cref{lem:supportInAbelianGroupsIsInjStack}
\end{proof}
\begin{proposition}\label[proposition]{prop:temperedSupportsAreConjugate}
  Let \(G\) be a finite group, and let \(x\) be a point in \(|\gg\{\bb G\}|\).
  If \(V,W\) are abelian subgroups of \(G\) both in the support of \(x\), then they are conjugate in \(G\).
\end{proposition}
\begin{remark}
  The following proof is inspired by Lurie's proof of \cite[lem.~4.9.6]{Lurie_2019_EllipticCohomologyIII}.
  In fact, the first few reduction steps are basically identical to it.
\end{remark}
\begin{proof}
  As a first reduction, we may use \cref{lem:reduceSupportToCompleteLocalRings} so that we may assume that \(R\) is noetherian complete local, and that \(x\) lies over the closed point \(\mm\) of \(\Spec R\).

  Pick preimages \(x_V,x_W\) of \(x\).
  Then the support of \(x_V\) is \(V\), so by \cref{lem:supportInAbelianGroupsIsInjStack} \(x_V\) lies in the open substack \(Inj(\hat{V},\gg)\), and similarly for \(x_W\).

  There is short exact sequence of \(\mathbb{P}\)-divisable groups
  \[0\to \gg^\circ\xrightarrow{f}\gg\xrightarrow{q}\gg^{\et}\to 0\]
  over \(R\) such that \(\gg^\circ\) is connected over the closed point, and \(\gg^{\et}\) is étale: If the characteristic of \(k(\mm)\) is zero, simply set \(\gg^\circ=0\), if it is \(p>0\) use the connected-étale sequence guaranteed by \cite[cor.~2.5.22]{Lurie_Ell2}.
  Since \(\Spec\pi_0R\) is connected the \(\ell\)-divisable groups \(\gg^{\et}_{(\ell)}\) have constant height \(h_\ell\).
  Setting
  \[\Lambda=\bigoplus_{\ell\in\mathbb{P}}(\qq_\ell/\zz_\ell)^{h_\ell}\]
  we get, by \cite[prop.~2.7.15]{Lurie_2019_EllipticCohomologyIII}, a faithfully flat \(R\)-algebra \(S=Split_\Lambda(f)\) such that
  \[\gg_S\simeq\gg^\circ_S\oplus \underline{\Lambda}_S.\]

  Now fix a point \(\nn\in|\spec\, S|\) over \(\mm\).
  We claim that 
  \[|Inj(\hat{V},\gg)_S|_\nn=\coprod_{\alpha:\hat{\Lambda}\twoheadrightarrow V}|\Spec S|_\nn,\]
  where the subscipt \(\nn\) denotes the subset of points over \(\nn\) and the coproduct runs over the {\it continiuous} surjective homomorphisms.

  Since \(R\to S\) is faithfully flat, basechange along it commutes with passing to underlying classical stacks, so by \cite[thm.~4.7.1,\,thm.~4.3.2]{Lurie_2019_EllipticCohomologyIII} we get a commutative diagram
\[\begin{tikzcd}
	{|\gg\{\bb W\}|} && {|\coprod_{\alpha:\widehat{\Lambda}\to W}\gg_S^\circ\{\bb W\}|} \\
	{|\gg\{\bb G\}|} && {|\coprod_{(\gamma:\widehat{\Lambda}\to G)_G}\gg_S^\circ\{\bb Z_G(\im \gamma)\}|} \\
	{|\gg\{\bb V\}|} && {|\coprod_{\beta:\widehat{\Lambda}\to V}\gg_S^\circ\{\bb V\}|}
	\arrow[from=1-1, to=2-1]
	\arrow[from=1-3, to=1-1]
	\arrow[from=1-3, to=2-3]
	\arrow[from=2-3, to=2-1]
	\arrow[from=3-1, to=2-1]
	\arrow[from=3-3, to=2-3]
	\arrow[from=3-3, to=3-1]
\end{tikzcd}\]
  Note that, in the corresponding diagram of classical stacks, both squares are pullback squares.
  We can now use \cite[\href{https://stacks.math.columbia.edu/tag/01JT}{Tag 01JT}]{stacks-project} to construct lifts of \(x\), \(x_V\), and \(x_W\) along the horizontal maps.
  Fix a prime ideal
  \[\mathfrak{p}_V\subset k(x_V)\otimes_{k(\mm)}k(\nn)\]
  and denote the associated lift of \(x_V\) by \(y_V\).
  Since \(x_V\) and \(x_W\) both map to \(x\), there is a zig-zag of subconjugations of abelian subgroups of \(G\) and points in the corresponding stacks connecting \(x_V\) and \(x_W\):
\[\begin{tikzcd}
	{p_0\in|\gg\{\bb A_0\}|} & {p_{n}\in|\gg\{\bb A_{n}\}|} & \\
	{x_V\in|\gg\{\bb V\}|} & \ldots & {x_W\in|\gg\{\bb W\}|}
	\arrow[hook, from=2-1, to=1-1]
	\arrow[hook', from=2-2, to=1-1]
	\arrow[hook, from=2-2, to=1-2]
	\arrow[hook', from=2-3, to=1-2]
\end{tikzcd}\]

  We then choose lifts and restrict to transport \(\mathfrak{p}_V\) to a prime ideal \(\mathfrak{p}_W\subset k(x_W)\otimes_{k(\mm)}k(\nn)\) along the zig-zag
\[\begin{tikzcd}
	{\mathfrak{q}_0\subset k(p_0)\otimes_{k(\mm)}k(\nn)} & {\mathfrak{q}_n\subset k(p_n)\otimes_{k(\mm)}k(\nn)} & \\
	\\
	{\mathfrak{p}_V\subset k(x_V)\otimes_{k(\mm)}k(\nn)} & \ldots & {\mathfrak{p}_W\subset k(x_W)\otimes_{k(\mm)}k(\nn)}
	\arrow[hook, from=1-1, to=3-1]
	\arrow[hook', from=1-1, to=3-2]
	\arrow[hook, from=1-2, to=3-2]
	\arrow[hook', from=1-2, to=3-3]
\end{tikzcd}\]
  and denote the associated lift of \(x_W\) by \(y_W\).
  This construction ensures that \(y_V\) and \(y_W\) map to the same point in \(|\gg_S\{\bb G\}|\).

  By our claim \(y_V\) and \(y_W\) lie in components associated to surejctions \(\beta:\hat{\Lambda}\twoheadrightarrow V\) and \(\alpha:\hat{\Lambda}\twoheadrightarrow W\).
  Since they map to the same point, their images must be in the same component, so that, in particular, the images of the composites
  \(\hat{\Lambda}\xrightarrow{\beta}V\subset G\text{ and }\hat{\Lambda}\xrightarrow{\alpha}W\subset G,\)
  that is \(V\) and \(W\), must be conjugate in \(G\).

  Let us now prove the claim.
  Combining \cite[lem.~5.5.5]{BalderramaDaviesLinskens_2026_AmbidextrousGlobalSpectra} and \cite[thm.~4.7.1]{Lurie_2019_EllipticCohomologyIII} we see that
  \[Inj(\widehat{V},\gg)_S\simeq Inj(\widehat{V},\gg_S).\]
  By \cite[lem.~5.5.5]{BalderramaDaviesLinskens_2026_AmbidextrousGlobalSpectra} and \cite[thm.~4.3.2]{Lurie_2019_EllipticCohomologyIII} we can write
  \[|Inj(\widehat{V},\gg_S)|\simeq \coprod_{\alpha:\hat{\Lambda}\to V}|\spec\,S(\gg^\circ_S)^V|\setminus \bigcup_{\im \alpha\subset V'\nsubseteq V}|\spec\,S(\gg^\circ_S)^{V'}|,\]
  where the coproduct runs over the {\it continuous} homomorphisms.
  The claim now follows if we can show that, on points over \(\nn\), the restriction induces an equivalence
  \[|\spec\,S(\gg^\circ_S)^V|_\nn \simeq |\spec\,S(\gg^\circ_S)^e|_\nn=\{\nn\}.\]
  By the construction of tempered cohomology and since \(R\to S\) is flat we have a natural equivalence
  \[|\spec\,S(\gg^\circ_S)^V|=|\gg^\circ_S[\widehat{V}]|=|\gg^\circ[\widehat{V}]_S|=|\gg^\circ[\widehat{V}]^\heartsuit_{\pi_0S}|.\]
  Let \(a\) and \(b\) be points of \(\gg^\circ[\widehat{V}]^\heartsuit_{\pi_0S}\) over \(\nn\), which are then the images of the unique points of \(\spec\,k(a)\) and \(\spec\,k(b)\).
  Let \(\mathfrak{l}\subset k(a)\otimes_{k(\nn)}k(b)\) be a maximal ideal and denote the quotient field by \(L\).
  Then composing with the maps from \(\spec\,k(a)\) and \(\spec\, k(b)\) we get two maps \(\spec\, L\to \gg^\circ[\widehat{V}]^\heartsuit_{\pi_0S}\) over \(\spec\,\pi_0 S\) so that the images of the unique point of \(\spec\, L\) are  \(a\) and \(b\).

  By definition we have an equivalence of mapping spaces
  \[\Stk_{/\spec\,\pi_0 S}(\spec\, L, \gg^\circ[\widehat{V}]^\heartsuit_{\pi_0S})\simeq \Stk_{/\spec\, R_{\geq 0}}(\spec\, L, \gg^\circ[\widehat{V}])\simeq Mod_{\zz}^{cn}(\widehat{V},\gg^\circ(L)).\]
  The \(\mathbb{P}\)-divisable group \(\gg^\circ\) is connected over the closed point of \(R\) and \(L\) is a reduced \(\pi_0R/\mm\)-algebra, so by \cite[prop.~2.3.9]{Lurie_Ell2} we have \(\gg^\circ(L)=0\).
  Thus, the mapping space is contractible, so \(a=b\) and we are done.
\end{proof}
\begin{remark}\label[remark]{rmk:boundedHeightBoundedGeneration}
  Let \(G\) be a finite group, \(x\) a point in \(|\gg\{\bb G\}|\), and let \(S\) be an abelian subgroup in the support of \(x\).
  The proof of \cref{prop:temperedSupportsAreConjugate} also gives us the following:
  Let \(x\) be over \(\mathfrak{p}\in \spec\,R\) and denote the vector of heights of the \(\mathbb{P}\)-divisable group \((\gg_{\widehat{R}_\mathfrak{p}})^{et}\) by \(h(\mathfrak{p})\).
  The group \(S\) then must admit a continuous surjection from
  \[\prod_{\ell\in\mathbb{P}}\nb{\zz_\ell^\wedge}^{h(\mathfrak{p})_\ell}.\]
  In particular, if \(h(\mathfrak{p})\) is bounded above by some \(n\), then \(S\) can be generated by \(\leq n\) elements.
\end{remark}

\begin{lemma}\label[lemma]{lem:temperedHalfAbelianSupportStaysSupport}
  Let \(G\) be a finite group, let \(H\) be an abelian subgroup of \(G\), and let \(x\) be a point in \(|\gg\{\bb H\}|\).
  If \(S\subset G\) is in the support of the image of \(x\) in \(|\gg\{\bb G\}|\), then \(S\) is conjugate to a subgroup \(T\subset H\subset G\) such that \(T\) is in the support of \(x\). 
\end{lemma}
\begin{proof}
  Denote the image of \(x\) in \(|\gg\{\bb G\}|\) by \(y\), and let \(S'\) be in the support of \(x\).
  Then \(y\) also comes from \(S'\subset H\subset G\), so that the support of \(y\) contains a subgroup \(T\subset S'\subset H\subset G\).
  \cref{prop:temperedSupportsAreConjugate} tells us that \(T\) and \(S\) are conjugate in \(G\).

  Let us now show that we actually have \(T=S'\).
  By assumption, there is a point \(p\in|\gg\{\bb T\}|\) that maps to \(y\).
  Denote the image of \(p\) in \(|\gg\{\bb H\}|\) by \(x'\).
  Then both \(x\) and \(x'\) map to \(y\).
  In particular, \(T\) must be in the support of \(x'\).
  Moreover, there is a zig-zag of subconjugations of abelian subgroups of \(G\)
\[\begin{tikzcd}
	{A_1} & {A_3} & {A_n} & \\
	H & {A_2} & \ldots & H
	\arrow["{f_0}"{description}, hook, from=2-1, to=1-1]
	\arrow["{f_1}"{description}, hook, from=2-2, to=1-1]
	\arrow["{f_2}"{description}, from=2-2, to=1-2]
	\arrow[hook, from=2-3, to=1-2]
	\arrow[hook, from=2-3, to=1-3]
	\arrow["{f_n}"{description}, hook, from=2-4, to=1-3]
\end{tikzcd}\]
  and points \(p_i\) in the associated stacks \(\gg\{\bb A_i\}\) that connects \(x\) to \(x'\).

  By \cref{lem:temperedAbelianSupportStaysSupport} we then have that \(f_0(S')\) is in the support of \(p_1\).
  By assumption \(p_1\) comes from \(f_1(A_2)\), so \(f_1(A_2)\) contains a support of \(p_1\), too.
  Since \(A_1\) is abelian, \cref{prop:temperedSupportsAreConjugate} tells us that the support of \(p_1\) is a singelton, so \(f_0(S')\subset f_1(A_2)\).

  Now \(\gg\{\bb f_1\}\) is a closed immersion of affine stacks, so the preimage \(f_1^{-1}(f_0(S'))\) must be in the support of \(p_2\).
  Continuing like this we find that
  \[f_n^{-1}(\cdots(f_1^{-1}(f_0(S')))\cdots)=T.\]
  In particular, we have that the order of \(S'\) and \(T\) are equal.
  Thus, the inclusion \(T\subset S'\) must be an equality.
\end{proof}

\begin{lemma}\label[lemma]{lem:temperedSupportStaysSupport}
  Let \(H\subset G\) be finite groups, let \(x\) be a point in \(|\gg\{\bb H\}|\), and denote its image in \(|\gg\{\bb G\}|\) by \(y\).
  If \(S\subset H\) is in the support of \(x\), then \(S\) is also in the support of \(y\).
\end{lemma}
\begin{proof}
  Let \(z\) be a preimage of \(x\) in \(|\gg\{\bb S\}|\).
  Applying \cref{lem:temperedHalfAbelianSupportStaysSupport} to \(z\) and \(S\subset G\) we find that any support \(S'\subset G\) of \(y\) must be conjugate in \(G\) to a subgroup \(T\subset S\) such that \(T\) is in the support of \(z\).
  But the support of \(z\) must be \(S\) itself, so \(T=S\).
  Since supports are closed under conjugation, we see that \(S\) is also in the support of \(y\) and we are done.
\end{proof}

\begin{definition}
  Let \(G\) be a finite group and let \(\mathcal{F}\) be a family of subgroups of \(G\).
  We define the {\it vanishing locus of \(\mathcal{F}\)} to be
  \[Z_\gg(\mathcal{F})=\{x\in|\gg\{\bb G\}|\mid supp(x)\subset\mathcal{F}\}.\]
\end{definition}

\begin{lemma}\label[lemma]{lem:ClosedSubsetOfFamily}
  Let \(G\) be a finite group and \(\mathcal{F}\) be a family of subgroups of \(G\).
  \begin{enumerate}
    \item For every subgroup \(H\subset G\) the preimage of \(Z_\gg(\mathcal{F})\) under the map
      \[|\gg\{\bb H\}|\to|\gg\{\bb G\}|\]
      induced by the inclusion coincides with \(Z_\gg(\mathcal{F}\vert_H)\).
    \item The subset \(Z_\gg(\mathcal{F})\subset |\gg\{\bb G\}|\) is closed.
  \end{enumerate}
\end{lemma}
\begin{proof}
  For (1), denote the induced map of topological spaces by \(\iota\) and the preimage by \(Z'\).
  We have that
  \[Z'=\{x\in|\gg\{\bb H\}|\mid supp(\iota(x))\subset \mathcal{F}\}.\]
  Since \(\iota(x)\) comes from \(H\), there is a support \(S\subset H\subset G\) of \(x\).
  But then \(S\) must also be a support of \(x\), so \(supp(x)=(S)_H\subset (S)_G=supp(\iota(x))\).
  Thus, we have the inclusion \(Z'\subset Z_\gg(\mathcal{F}\vert_H)\).
  In the other direction, if \(supp(x)=(S)_H\subset \mathcal{F}\vert_H\), then by \cref{lem:temperedSupportStaysSupport} we have \(supp(\iota(x))=(S)_G\) and \(S\in\mathcal{F}\).
  Since families are closed under conjugation, we then also have that \(supp(\iota(x))\subset\mathcal{F}\), so \(Z_\gg(\mathcal{F}\vert_H)\subset Z'\).

  For (2), let us first consider the case when \(G\) is abelian.
  Then, for every subgroup \(H\subset G\), the map \(\gg\{\bb H\}\to\gg\{\bb G\}\) is a closed immersion between affines and we have
  \[Z_\gg(\mathcal{F})=\bigcup_{H\in\mathcal{F}}\im (|\gg\{\bb H\}|\to\gg\{\bb G\}|).\]
  This is a finite union of closed subsets and thus closed.

  In the general case, we know that the topological space \(|\gg\{\bb G\}|\) is a quotient of the topological space
  \[\coprod_{G/H\in \Orb_{\ab}(G)}|\gg\{\bb H\}|,\]
  and by (1) the preimage of \(Z_\gg(\mathcal{F})\) in this is given by
  \[\coprod_{G/H\in \Orb_{\ab}(G)}Z_\gg(\mathcal{F}\vert_H).\]
  Since each \(H\) in the above is abelian, we already know that \(Z_\gg(\mathcal{F}\vert_H)\subset |\gg\{\bb H\}|\) is closed, so the above is closed, so \(Z_\gg(\mathcal{F})\) is closed.
\end{proof}

\begin{definition}
  Let \(G\) be a finite group and let \(\mathcal{F}\) be a family of subgroups of \(G\).
  We denote the open substack of \(\gg\{\bb G\}\) associated to the open subset \(|\gg\{\bb G\}|\setminus Z_\gg(\mathcal{F})\), whose existence is guaranteed by \cref{rmk:open_substacks_exist}, by \(U_\gg(\mathcal{F})\).
\end{definition}

\begin{proposition}\label[proposition]{prop:affinizationVIFandZF}
  Let \(G\) be a finite group and \(\mathcal{F}\) be a family of subgroups of \(G\).
  Define the ideal \(I_\gg(\mathcal{F})\) as the intersection of all the restriction ideals for subgroups contained in \(\mathcal{F}\), that is
  \[I_\gg(\mathcal{F})=\bigcap_{H\in \mathcal{F}}\ker(\pi_0 R_\gg^{\bb G}\to \pi_0R_\gg^{\bb H}).\]
  Then \(Z_\gg(\mathcal{F})\) is the preimage of \(V(I_\gg(\mathcal{F}))\) under the affinization map
  \[|a|:|\gg\{\bb G\}|\to |\spec\,R_\gg^{\bb G}|.\]
\end{proposition}
\begin{proof}
  Since each map \(\pi_0R_\gg^{\bb G}\to \pi_0R_\gg^{\bb H}\) is finite (\cite[cor.~4.7.13(b)]{Lurie_2019_EllipticCohomologyIII}) and thus integral, we have that \(V(I_\gg(\mathcal{F}))\) is the union of the images of \(|\spec\,R_\gg^{\bb H}|\to |\spec\,R_\gg^{\bb G}|\) over all \(H\in \mathcal{F}\).

  A point \(x\in Z_\gg(\mathcal{F})\) is, by definition, in the image of \(|\gg\{\bb H\}|\to |\gg\{\bb G\}|\) with \(H\in\mathcal{F}\) abelian.
  For such \(H\) \(\gg\{\bb H\}\) is affine, so we see that \(|a|(x)\in V(I_\gg(\mathcal{F}))\).
  
  Now assume that \(x\notin Z_\gg(\mathcal{F})\).
  We need to show that then \(|a|(x)\notin V(I_\gg(\mathcal{F}))\).

  Let \(\mm\in |\spec\,R|\) be the image of \(x\).
  Since \(R\to R_{\mm}\) is a localization and \(R_\mm\to \widehat{R_\mm}\) is faithfully flat and local (since \(R\) is noetherian), \(x\) has a unique preimage \(x'\) in
  \[|\spec\,\widehat{R_\mm}\times_{\spec\,R}\gg\{\bb G\}|=|\gg_{\widehat{R_\mm}}\{\bb G\}|,\]
  \(x\in Z_\gg(\mathcal{F})\) if and only if \(x'\in Z_{\gg_{\widehat{R_\mm}}}(\mathcal{F})\), and \(|a|(x)\in V(I_\gg(\mathcal{F}))\) if an only if \(|a|(x')\in V(I_{\gg_{\widehat{R_\mm}}}(\mathcal{F}))\).
  Thus, we may assume that \(R\) is complete local and \(x\) lies over the closed point \(\mm\).

  Now \(\gg\) admits a connected-étale sequence
  \[0\to \gg^\circ \xrightarrow{f} \gg\xrightarrow{q} \gg^{\et}\to 0.\]
  If the characteristic of \(k(\mm)\) is zero we can simply set \(\gg^\circ=0\), else we use \cite[cor.~2.5.22]{Lurie_Ell2}.
  Now by \cite[prop.~2.7.15]{Lurie_2019_EllipticCohomologyIII} there exists a faithfully flat map \(R\to S\) and a colattice \(\Lambda\) such that \(\gg^{\et}_S\simeq \underline{\Lambda}_S\) and such that basechange of the connected-étale sequence
  \[0\to \gg^\circ_S\to\gg_S\to\gg^{\et}_S\to 0\]
  splits.
  By faithful flatness there exists a lift \(x'\) of \(x\) to \(|\gg_S\{\bb G\}|\), and denote the image of \(x'\) in \(|\spec\, S|\) by \(\nn\).

  Our assumption on \(x\) gives us that \(x'\notin Z_{\gg_S}(\mathcal{F})\), and it suffices to show that \(|a|(x')\notin V(I_{\gg_S}(\mathcal{F}))\).
  Now \cite[thm.~4.3.2]{Lurie_2019_EllipticCohomologyIII} identifies the affinization map with
  \[\coprod_{(\gamma:\widehat{\Lambda}\to G)_G}|a|_{[\gamma]}:\coprod_{(\gamma:\widehat{\Lambda}\to G)_G}|\gg^\circ_S\{\bb Z_G(\im\gamma)\}|\to\coprod_{(\gamma:\widehat{\Lambda}\to G)_G}|\spec\,S_{\gg^\circ_S}^{\bb Z_G(\im\gamma)}|.\]
  Let \(\alpha\) be a representative for the class of the component \(x'\) is contained in.
  As in the proof of \cref{prop:temperedSupportsAreConjugate} we see that the support of \(x'\) must then be given by \((\im\alpha)_G\).
  In particular, since \(x'\notin Z_{\gg_S}(\mathcal{F})\), we must have \(\im\alpha\notin\mathcal{F}\).

  For every \(H\in\mathcal{F}\) the image of
  \[|\spec\,S_{\gg_S}^{\bb H}|\to |\spec\,S_{\gg_S}^{\bb G}|\]
  is contained parts of the above disjoint union where some representative \(\gamma:\widehat{\Lambda}\to G\) factors over \(H\), so in parts where \(\im\gamma\in\mathcal{F}\).
  In particular, we see that \(|a|(x')\) can't be contained in \(V(I_{\gg_S}(\mathcal{F}))\).
\end{proof}

\begin{proposition}
  Let \(H\subset G\) be finite groups and let \(\mathcal{F}\) be a family of subgroups of \(G\).
  \begin{enumerate}
    \item The basechange of \(U_\gg(\mathcal{F})\) to \(\gg\{\bb H\}\) is \(U_\gg(\mathcal{F}\vert_H)\).
    \item The canonical map \(U_\gg(\mathcal{F})\to \gg\{\bb G\}\) is a locally quasi-compact open immersion.
  \end{enumerate}
\end{proposition}
\begin{proof}
  For (1), by \cref{cor:open_immersion_basechange} the basechange is the open immersion associated to the preimage of \(|\gg\{\bb G\}|\setminus Z_\gg(\mathcal{F})\), which, by \cref{lem:ClosedSubsetOfFamily}, is exactly \(|\gg\{\bb H\}|\setminus Z_\gg(\mathcal{F\vert_H})\).

  For (2), write
  \[\gg\{\bb G\}\simeq \colim_{G/H\in \Orb_{\ab}(G)}\gg\{\bb H\}.\]
  Since each \(\gg\{\bb H\}\) is affine, \cref{lem:open_immersion_basechange_to_affines} and (1) tell us that the basechange of \(U_\gg(\mathcal{F})\) agrees with an open spectral subscheme of \(\gg\{\bb H\}\simeq \Spec R(\gg)^H\).
  Since \(R\) is noetherian and \(\gg\{\bb H\}\) is finite over \(\spec\,R\), also \(\gg\{\bb H\}\) is noetherian, so every open is quasi-compact.
\end{proof}

\subsection{Family completion over affines}\label{ssec:family_completion_affine}
\begin{lemma}\label[lemma]{lem:temperedIsMNNAbNilpotent}
  Let \(G\) be a finite group.
  Then \(R(\gg)_G\) is nilpotent with respect to the family of abelian subgroups of \(G\), in the sense of \cite[def.~6.36]{MathewNaumannNoel_2017_NilpotenceDescentEquivariant}.
\end{lemma}
\begin{proof}
  By \cite[thm.~6.41]{MathewNaumannNoel_2017_NilpotenceDescentEquivariant} we need to show that, for every nonabelian subgroup \(H\subset G\), the geometric fixed point \((R(\gg)_G)^{\Phi H}\) vanish.
  But this is exactly the content of \cite[prop.~5.4.4]{BalderramaDaviesLinskens_2026_AmbidextrousGlobalSpectra}.
\end{proof}

\begin{lemma}\label{lem:factors_over_character_stack}
  There is a commutative diagram
\[\begin{tikzcd}
	{\QCoh(\gg\{\bb G\})} & {\Mod_{R(\gg)_G}(\Sp_G)} \\
	{\QCoh(\Spec \Gamma\gg\{\bb G\})} & {\Mod_{R(\gg)^G}(\Sp)}
	\arrow[from=1-1, to=1-2]
	\arrow["{a^*}", from=2-1, to=1-1]
	\arrow["\simeq"', from=2-1, to=2-2]
	\arrow[from=2-2, to=1-2]
\end{tikzcd}\]
  in \(\CAlg(\PrLst)\), natural in \(\bb G\in \Glo_\fin^\op\).
\end{lemma}
\begin{proof}
  First of all, note that the functor
  \[\Glo_\fin^\op\to\CAlg(\PrLst),\,\bb G\mapsto \Mod_{R(\gg)_G}(\Sp_G)\]
  is right Kan extended from its restriction to \(\Glo_\fab^\op\): this is a consequence of \cref{lem:temperedIsMNNAbNilpotent} and \cite[thm.~6.42]{MathewNaumannNoel_2017_NilpotenceDescentEquivariant}.
  Since \(\QCoh:\Stk^\op\to \CAlg(\PrLst)\) preserves limits, also
  \[\Glo_\fin^\op\to\CAlg(\PrLst),\,\bb G\mapsto \QCoh(\gg\{\bb G\})\]
  is right Kan extended from its restriction to \(\Glo_\fab^\op\).

  Denote by \(\mathbb{\Gamma}:\CAlg(\PrLst)\to \CAlg(\Sp)\) the functor that maps \(C\) to \(\mathrm{end}_C(\mathbb{1}_C)\), which is a right adjoint by \cite[thm.~4.8.5.11,\,thm.~4.8.5.16(4)]{Lurie_HA}.
  Moreover, its left adjoint \(\Theta\) is fully faithful by \cite[cor.~4.8.5.22]{Lurie_HA}.

  Let us first construct the natural transformation
  \[\phi_{\bb G}:\QCoh(\gg\{\bb G\})\to \Mod_{R(\gg)_G}(\Sp_G).\]
  Both are right Kan extended from \(\Glo_\fab^\op\), and on there the source is, by construction, equivalent to \(\Theta\mathbb{\Gamma}\) of the target.
  Thus, we may take \(\phi\) to be the right Kan extension of the adjunction counit.
  The square we want to construct then arises as the adjunction counit \(\Theta\mathbb{\Gamma}\phi_{\bb G}\to \phi_{\bb G}\).

  Since \(\Theta\) is fully faithful and \(\mathbb{\Gamma}\) commutes with limits, appyling  \(\mathbb{\Gamma}\) to \(\phi_{\bb G}\) yields an equivalence.

  Now note that we have equivalences \(\Theta\simeq \QCoh\circ\Spec\) and \(\Gamma\simeq \mathbb{\Gamma}\circ \QCoh\), so we can naturally identify the adjunction counit
  \[\Theta\mathbb{\Gamma}\QCoh(X)\to \QCoh(X)\]
  with \(\QCoh\) applied to the affinization map \(a_X:X\to \Spec \Gamma X\), for any stack \(X\).
  This yields the desired factorization. 
\end{proof}

\begin{lemma}\label[lemma]{lem:temperedAbelianProperCompletion}
  Let \(A\) be a finite abelian group, let \(\mathcal{P}\) be the family of proper subgroups of \(A\), and denote by \(I\mathcal{P}\) the intersection of the ideals
  \[I^A_B=\ker \pi_0R(\gg)^A\to \pi_0R(\gg)^B\]
  over all \(B\in \mathcal{P}\).
  Then we have an equivalence of recollements
  \[\mathcal{R}(\widetilde{E\mathcal{P}})\otimes_{\Sp_A}\Mod_{R(\gg)_A}(\Sp_A)\simeq \mathcal{R}(I\mathcal{P})\otimes_{\Mod_{R(\gg)^A}(\Sp)}\Mod_{R(\gg)_A}(\Sp_A).\]
\end{lemma}
\begin{proof}
  It will suffice to show that the classes of nilpotent objects agree.
  By \cref{prop:FCompleteImpliesIFComplete} we already know that \(\mathcal{P}\)-nilpotent \(R(\gg)_A\)-modules are \(I\mathcal{P}\)-nilpotent.

  If \(M\) is an \(I\mathcal{P}\)-nilpotent \(R(\gg)_A\)-module, then \(M^A\) is an \(I\mathcal{P}\)-complete \(R(\gg)^A\)-module.
  Note that \(I\mathcal{P}\) coincides with the ideal \(I\) defined in \cite[def.~15.10]{GepnerLinskensPol_2024_Global2ringsGenuine}.

  In particular, we have that \(\Phi^A(M)=0\), where \(\Phi^A\) is defined in \cite[def.~15.11]{GepnerLinskensPol_2024_Global2ringsGenuine}.

  Now \cite[prop.~15.14]{GepnerLinskensPol_2024_Global2ringsGenuine}
  tells us that \(M\) is contained in the localizing tensor ideal generated by
  \[\{R(\gg)_A\otimes Ind_{A_0}^A\mathbb{1}\simeq (A/A_0)_+\mid A_0\in\mathcal{P}\}.\]
  Thus, the underlying \(A\)-spectrum of \(M\) is \(\mathcal{P}\)-nilpotent by \cref{lem:generatorsForFNilpotents}, and so also \(M\) is \(\mathcal{P}\)-nilpotent.
\end{proof}

\begin{theorem}\label[theorem]{thm:betterTemperedFamilyCompletion}
Let \(G\) be a finite group, let \(\mathcal{F}\) be a family of subgroups of \(G\), and denote by \(I_\gg(\mathcal{F})\) the intersection of the ideals
\[I^G_H=\ker \pi_0R(\gg)^G\to \pi_0R(\gg)^H\]
over all \(H\in \mathcal{F}\).
Then we have an equivalence of recollements
\[\mathcal{R}(\widetilde{E\mathcal{F}})\otimes_{\Sp_G}\Mod_{R(\gg)_G}(\Sp_G)\simeq \mathcal{R}(I_\gg(\mathcal{F}))\otimes_{\Mod_{R(\gg)^G}(\Sp)}\Mod_{R(\gg)_G}(\Sp_G).\]
\end{theorem}
\begin{proof}
  We need to show that the classes of complete objects agree.
  From \cref{prop:FCompleteImpliesIFComplete} we already now that \(\mathcal{F}\)-complete \(R(\gg)_G\)-modules are \(I_\gg(\mathcal{F})\)-complete.
  For the converse we will use an induction on the subgroups of \(G\).

  To start the induction, assume that \(G=e\).
  There is only one family of subgroups, namely \(\mathcal{F}=\{e\}\).
  Then \(\widetilde{E\mathcal{F}}=0\) so every module is \(\mathcal{F}\)-complete.

  Now assume we have proven the theorem for every proper subgroup of \(G\).
  Again, if \(G\in\mathcal{F}\), then every \(R(\gg)_G\)-module is \(\mathcal{F}\)-complete, so we may assume that \(G\notin \mathcal{F}\).

  Let \(M\) be an \(I_\gg(\mathcal{F})\)-complete \(R(\gg)_G\)-module.
  We need to show that
  \[\underline{\map}_{R(\gg)_G}(R(\gg)_G\otimes \widetilde{E\mathcal{F}},M)=0.\]
  By mapping the isotropy seperation sequence into this, we reduce to proving that both
  \[X=\underline{\map}_{R(\gg)_G}(R(\gg)_G\otimes \widetilde{E\mathcal{F}}\otimes \widetilde{E\mathcal{P}},M)\text{ and }Y=\underline{\map}_{R(\gg)_G}(R(\gg)_G\otimes \widetilde{E\mathcal{F}}\otimes E\mathcal{P}_+,M)\]
  vanish, where \(\mathcal{P}\) is the family of proper subgroups of \(G\).

  By assumption we have \(\mathcal{F}\subset \mathcal{P}\), so \(\widetilde{E\mathcal{F}}\otimes \widetilde{E\mathcal{P}}=\widetilde{E\mathcal{P}}\), and any \(I_\gg(\mathcal{F})\)-complete object is also \(I\mathcal{P}\)-complete.

  If \(G\) is abelian, we can apply \cref{lem:temperedAbelianProperCompletion} to conclude that \(X=0\).

  If \(G\) is not abelian, we see that the family of abelian subgroups of \(G\) \(\mathcal{AB}\) is contained in \(\mathcal{P}\).
  In particular, we have \(\widetilde{E\mathcal{P}}=\widetilde{E\mathcal{P}}\otimes \widetilde{E\mathcal{AB}}\).
  Also, since \(R(\gg)_G\) is \(\mathcal{AB}\)-nilpotent in the sense of \cite[def.~6.41]{MathewNaumannNoel_2017_NilpotenceDescentEquivariant}, every \(R(\gg)_G\)-module is \(\mathcal{AB}\)-complete
  Thus, also in this case we have \(X=0\).

  Now for \(Y\):
  Since \(E\mathcal{P}_+\) is \(\mathcal{P}\)-nilpotent, \cref{lem:generatorsForFNilpotents} tells us that it is contained in the localizing stable subcategory generated by \(G/H_+\) for \(H\in\mathcal{P}\).
  Thus it suffices to prove that, for every proper subgroup \(H\subset G\),
  \[Y(H)=\underline{\map}_{R(\gg)_G}(R(\gg)_G\otimes \widetilde{E\mathcal{F}}\otimes G/H_+,M)\]
  vanishes.
  The underlying \(G\)-spectrum of this is equivalent to the induction from \(H\) to \(G\) of
  \[Y'(H)=\underline{\map}_{R(\gg)_H}(R(\gg)_H\otimes\widetilde{E\mathcal{F}\vert_H},\mathrm{Res}^G_HM).\]
  We claim that \(\mathrm{Res}^G_HM\) is \(I_\gg(\mathcal{F}\vert_H)\)-complete.
  Then our inductive hypothesis shows that \(\mathrm{Res}^G_HM\) is also \(\mathcal{F}\vert_H\)-complete, so \(Y'(H)=0\) and we would be done.

  To prove the claim, note that Frobenius reciprocity tells us that \(\mathrm{Ind}_H^G\) is canonically \(\Mod_{R(\gg)_H}(\Sp_H)\)-linear, so in particular \(\Mod_{R(\gg)^H}(\Sp)\)-linear.
  Its right adjoint \(\mathrm{Res}^G_H\) thus commutes with all completions associated to recollements that are tensored up from \(\Mod_{R(\gg)^H}(\Sp)\).
  \Cref{lem:ICompletionAndRestrictingScalars} then tells us that \(\mathrm{Res}^G_H M\) is complete for the ideal \(J=\mathrm{res}^G_H(I_\gg(\mathcal{F}))\pi_0R(\gg)^H\).
  Consider the commutative diagram
\[\begin{tikzcd}
	{\gg\{\bb H\}} & {\gg\{\bb G\}} \\
	{\spec\,R(\gg)^H} & {\spec\,R(\gg)^G}
	\arrow[from=1-1, to=1-2]
	\arrow[from=1-1, to=2-1]
	\arrow[from=1-2, to=2-2]
	\arrow[from=2-1, to=2-2]
\end{tikzcd}\]
By definition \(V(J)\) is the preimage of \(V(I_\gg(\mathcal{F}))\) under the bottom horizontal map, and let \(W\) be the preimage of \(V(J)\) under the left vertical map.
By \cref{prop:affinizationVIFandZF} the preimage of \(V(I_\gg(\mathcal{F}))\) under the right vertical map is \(Z_\gg(\mathcal{F})\), and by \cref{lem:ClosedSubsetOfFamily} the preimage of \(Z_\gg(\mathcal{F})\) under the top horizontal map is \(Z_\gg(\mathcal{F}\vert_H)\).
Thus, \(W=Z(\mathcal{F}\vert_H)\), which by \cref{prop:affinizationVIFandZF} is also the preimage of \(V(I_\gg(\mathcal{F}\vert_H))\).

Finally, \cref{lem:factors_over_character_stack} tells us that the symmetric monoidal functor \(\Mod_{R(\gg)^H(\Sp)}\to \Mod_{R(\gg)_H}(\Sp_H)\) factors over \(\QCoh(\gg\{\bb H\})\).
By the previous discussion and \cref{cor:idempotent_recollement_basechange} we thus find that
\begin{align*}
  \mathcal{R}(J)\otimes_{\Mod_{R(\gg)^H}(\Sp)}\Mod_{R(\gg)_H}(\Sp_H)&\simeq\mathcal{R}(i_*\O_{U_\gg(\mathcal{F}\vert_H)})\otimes_{\QCoh(\gg\{\bb H\})}\Mod_{R(\gg)_H}(\Sp_HH)\\
  &\simeq \mathcal{R}(I_\gg(\mathcal{F}\vert_H))\otimes_{\Mod_{R(\gg)^H}(\Sp)}\Mod_{R(\gg)_H}(\Sp_H).
\end{align*}
So \(\mathrm{Res}^G_HM\) is \(I_\gg(\mathcal{F}\vert_H)\)-complete.
\end{proof}

As a consequence, we find that tempered cohomology with bounded height is 'detected' by abelian groups of bounded generation:

\begin{proposition}\label{prop:rightKanExtension}
  Let \(n\geq 0\) such that, for each point \(\mathfrak{p}\) in \(\spec\, R\) and each \(\ell\in\mathbb{P}\), the \(\ell\)-divisable group \((\gg_{\widehat{R}_\mathfrak{p}})^{et}_{(\ell)}\) has height less or equal to \(n\).
  Then the functor
  \[\Glo_{\fab}^{\op}\to \CAlg(\Sp),\,\bb V\mapsto R_\gg^{\bb V}\]
  is right Kan extended from the full subcategory on objects \(\bb H\) where \(H\) can be generated by at most \(n\) elements.
\end{proposition}
\begin{proof}
  Let \(\mathcal{F}_n\) be tha global family of finite abelian groups that can be generated by \(n\) or less elements.
  By the pointwise formula for right Kan extensions we need to show that
  \[R_\gg^{\bb V}\xrightarrow{\sim}\lim_{(\Glo_{\mathcal{F}_n})_{/\bb V}}R_\gg^\bullet\]
  is an equivalence.
  By \cite[lem.~4.2,prop.~2.16]{GepnerMeier_EquivariantTMF1} the right hand side agrees with the limit of its restriction along
  \[\Orb^{\mathcal{F}_n}(V)\to (\Glo_{\mathcal{F}_n})_{/\bb V},\, V/S\mapsto (V/S)//V\to *//V,\]
  where the source is the full subcategory of \(\Orb(V)\) on objects \(V/S\) with \(S\in\mathcal{F}_n\).
  
  The colimit of
  \[\Orb^{\mathcal{F}_n}(V)\to \Spc_V\]
  is the \(V\)-space \(E\mathcal{F}_n\vert_V\).
  Thus the map we are interested in can be identified with
  \[R_\gg^{*//V}\to R_\gg^{E\mathcal{F}_n\vert_V//V},\]
  which is equivalent to
  \[R(\gg)^V\to \underline{\map}((E\mathcal{F}_n\vert_V)_+,R(\gg)_V)^V.\]
  By \cref{thm:betterTemperedFamilyCompletion} we have that this map agrees with the \(V\)-fixed points of the \(I_\gg(\mathcal{F}_n\vert_V)\)-completion of \(R(\gg)_V\).
  Finally, by \cref{rmk:boundedHeightBoundedGeneration} and our assumption we have that \(\mathcal{F}_n\vert_V\) contains a support for every prime ideal in \(\pi_0R(\gg)^V\), so that \(I_\gg(\mathcal{F}_n\vert_V)\) is contained in the nilradical of \(\pi_0R(\gg)^V\).
  It clearly also contains \((0)\), so \(I_\gg(\mathcal{F}_n\vert_V)\)-completion is equivalent to \((0)\)-completion, which is the identity, and we are done.
\end{proof}

\subsection{Family completion over locally noetherian geometric stacks}\label{ssec:LNG}
Having done the affine case, we now want to give a version for more general base stacks \(M\).
Let us first define the class of stacks where our arguments will work:
\begin{definition}\label[definition]{def:LNG_stack}
  Let \(M\) be a stack.
  We say that \(M\) is a {\it locally noetherian geometric} stack, or {\it LNG}-stack, if it admits a small colimit presentation
  \[M\simeq \colim_{i\in I}\Spec R_i\]
  with each \(R_i\) noethrian and each transition map \(R_i\to R_j\) flat.
  In this case, we will call such a presentation an {\it LNG-presentation} of \(M\).
\end{definition}
\begin{remark}
  Note that LNG-stacks are, in particular, geometric stacks in the sense of \cite[def.~3.4.1.1]{BalderramaDaviesLinskens_2025_AffinenessReconstructionComplexperiodic}.
  Examples of interest include \(\mathrm{M}_{\mathrm{Ell}}^{\mathrm{or}}\) and \(\mathrm{M}_{\mathrm{Tori}}^{\mathrm{or}}\simeq(\Spec \KU)//C_2\).
\end{remark}

To even give a statement, we will first need to introduce/recall some notation.

\begin{construction}\label[construction]{cons:OMGG}
  From \cite[prop.~5.2.6]{BalderramaDaviesLinskens_2026_AmbidextrousGlobalSpectra} we get an object
\[\oo_M(\gg)\in \CAlg_\pi^{gl}(\Mod_{\oo_M}(\mathrm{Shv}(M,\Sp))).\]
  For any finite group \(G\) \cite[cons.~2.3.9]{BalderramaDaviesLinskens_2026_AmbidextrousGlobalSpectra}
  yields a functor
  \[(-)_G:\CAlg_\pi^{gl}(\Mod_{\oo_M}(\mathrm{Shv}(M,\Sp)))\to \CAlg(\Fun^\times(\mathrm{Span}(\Fin_G),\Mod_{\oo_M}(\mathrm{Shv}(M,\Sp)))).\]
  The value of \(\oo_M(\gg)_G\) on \(G/H\) lies, by \cite[prop.~5.3.3]{BalderramaDaviesLinskens_2026_AmbidextrousGlobalSpectra}, in the full subcategory
  \[\QCoh(M)\subset \Mod_{\oo_M}(\mathrm{Shv}(M,\Sp)).\]
  This inclusion is symmetric monoidal and commutes with colimits, so by stability it commutes with finite products.
  In particular, we find that \(\oo_M(\gg)_G\) canonically lifts to
  \[\CAlg(\Fun^\times(\mathrm{Span}(\Fin_G),\QCoh(M)))\simeq \CAlg(\Sp_G\otimes \QCoh(M)).\]

  The functoriality in \(\gg\) is as follows: denote by
  \[\underline{\CAlg_G}\to \Stk\]
  the Cartesian unstraightening of
  \[\Stk^{op}\to \widehat{\Cat_\infty},\,M\mapsto \CAlg(\Sp_G\otimes\QCoh(M)).\]
  Then \(\gg\mapsto \oo_M(\gg)_G\) assembles into a Cartesian section
\[\begin{tikzcd}
	& {\underline{\CAlg_G}} \\
	{\mathrm{PDiv}^{or}_{cart}} & \Stk
	\arrow[from=1-2, to=2-2]
	\arrow[from=2-1, to=1-2]
	\arrow[from=2-1, to=2-2]
\end{tikzcd}\]

Finally, recall from \cite[def.~5.3.4]{BalderramaDaviesLinskens_2026_AmbidextrousGlobalSpectra} the notation \(\gg(\bb G)\) for the relative affine stack over \(M\) associated to \(\O_M(\gg)^G\in \QCoh(M)\).
\end{construction}

\begin{lemma}\label[lemma]{lem:VIF_flat_basechange}
  Let \(G\) be a finite group, \(\mathcal{F}\) a family of subgroups of \(G\), \(f:\Spec S\to \Spec R\) a flat map between noetherian affine stacks, and \(\gg\) an orient \(\mathbb{P}\)-divisible group over \(\Spec R\).
  Then we have
  \[|f_G|^{-1}(V(I_\gg(\mathcal{F})))=V(I_{f^*\gg}(\mathcal{F})).\]
\end{lemma}
\begin{proof}
  As noted in the proof of \cref{prop:affinizationVIFandZF}, \(V(I_\gg(\mathcal{F}))\) equals the union of the images of
  \[|\Spec R(\gg)^H|\to |\Spec R(\gg)^G|\]
  for \(H\in \mathcal{F}\).
  Consider the commutative square
\[\begin{tikzcd}
	{\Spec S(f^*\gg)^H} & {\Spec R(\gg)^H} \\
	{\Spec S(f^*\gg)^G} & {\Spec R(\gg)^G}
	\arrow[from=1-1, to=1-2]
	\arrow[from=1-1, to=2-1]
	\arrow[from=1-2, to=2-2]
	\arrow[from=2-1, to=2-2]
\end{tikzcd}\]
  By \cite[thm.~4.7.1]{Lurie_2019_EllipticCohomologyIII} it is a pullback square, and the horizontal maps are flat by assumption.
  Thus, it remains a pullback when passing to \(\pi_0\), so \cite[\href{https://stacks.math.columbia.edu/tag/01JT}{Tag 01JT}]{stacks-project} tells us that the diagonal map in the diagram 
\[\begin{tikzcd}
	{|\Spec S(f^*\gg)^H|} && \\
	& {|\Spec S(f^*\gg)^G|\times_{|\Spec R(\gg)^G|}|\Spec R(\gg)^H|} & {|\Spec R(\gg)^H|} \\
	& {|\Spec S(f^*\gg)^G|} & {|\Spec R(\gg)^G|}
	\arrow[from=1-1, to=2-2]
	\arrow[curve={height=-12pt}, from=1-1, to=2-3]
	\arrow[curve={height=12pt}, from=1-1, to=3-2]
	\arrow[from=2-2, to=2-3]
	\arrow[from=2-2, to=3-2]
	\arrow[from=2-3, to=3-3]
	\arrow[from=3-2, to=3-3]
\end{tikzcd}\]
  is surjective.
  This shows that \(|f_G|^{-1}(V(I^G_H))=V(I^G_H)\) and we are done.
\end{proof}

\begin{proposition}\label[proposition]{prop:VIF_closed_over_geometric_stacks}
  Let \(M\) be an LNG-stack, let \(\gg\) be an oriented \(\mathbb{P}\)-divisible group over \(M\), let \(G\) be a finite group, and let \(\mathcal{F}\) be a family of subgroups of \(G\).
  Then \(V_\gg(\mathcal{F})\), the union of the images of
  \[|\gg(\bb H)|\to |\gg(\bb G)|\]
  for \(H\in\mathcal{F}\), is a closed subset.
  Moreover, the open immersion \(U'_\gg(\mathcal{F})\to \gg(\bb G)\) associated to its complement is locally quasi-compact.
\end{proposition}
\begin{proof}
  Choose an LNG-presentation \(M\simeq \colim_{i\in I}\Spec R_i\) of \(M\), and denote the inclusion maps by \(\iota_i:\Spec R_i\to M\).

  By universality of colimits we get a commutative square
\[\begin{tikzcd}
	{\coprod_{i\in I}|\Spec R_i(\iota_i^*\gg)^H|} & {\coprod_{i\in I}|\Spec R_i(\iota_i^*\gg)^G|} \\
	{|\gg(\bb H)|} & {|\gg(\bb G)|}
	\arrow[from=1-1, to=1-2]
	\arrow["p"', from=1-1, to=2-1]
	\arrow["q"', from=1-2, to=2-2]
	\arrow[from=2-1, to=2-2]
\end{tikzcd}\]
  where the vertical maps are quotient maps.
  By \cref{lem:VIF_flat_basechange} the image of the top map is closed under the equivalence relation associated to \(q\), so it agrees with the preimage of the image of the bottom map.
  Thus, the image of the bottom map is closed, and so is \(V_\gg(\mathcal{F})\).
  Additionally, the preimage of the of \(V_\gg(\mathcal{F})\) under each \(|\iota_i|\) agrees with \(V(I_{\iota_i^*\gg}(\mathcal{F}))\).

  Denoting the open immersion associated to the complement of \(V_\gg(\mathcal{F})\) by \(U_\gg(\mathcal{F})\to \gg(\bb G)\), the above shows that its base change along \(\iota_i\) agrees with the open immersion into \(\Spec R_i(\iota_i^*\gg)^G\) associated to \(|\Spec R_i(\iota_i^*\gg)^G|\setminus V(I_{\iota_i^*\gg}(\mathcal{F}))\).
  Since, by assumption, \(R_i\) is noetherian,  \(R_i(\iota_i^*\gg)^G\) is noetherian, so \(I_{\iota_i^*\gg}(\mathcal{F})\) is finitely generated, so this basechanged open immersion is quasi-compact.
  Thus, \(U_\gg(\mathcal{F})\to \gg(\bb G)\) is locally quasi-compact.
\end{proof}
\begin{remark}\label[remark]{rmk:LNG_to_local_affines_family_recollement_basechange}
  The proof also shows that, for such a small colimit presentation, the basechange of \(U_\gg(\mathcal{F})\to \gg(\bb G)\) along each \(\iota_i:\Spec R_i\to M\) agrees with the quasi-compact open immersion associated to the complement of \(V(I_{\iota_i^*\gg}(\mathcal{F}))\) in \(|\Spec R_i(\iota_i^*\gg)^G|\).
\end{remark}

\begin{definition}
  To match with the affine case, we denote the recollement of \(\QCoh(\gg(\bb G))\) associated to \(U'_\gg(\mathcal{F})\to \gg(\bb G)\) by \(\mathcal{R}(I_\gg(\mathcal{F}))\).
\end{definition}
\begin{theorem}\label[theorem]{thm:family_completion_LNG_version}
  Let \(M\) be an LNG-stack, and let \(\gg\) be an oriented \(\mathbb{P}\)-divisible group over \(M\).
  For any finite group \(G\) and family \(\mathcal{F}\) of subgroups of \(G\) the basechange of \(\mathcal{R}(I_\gg(\mathcal{F}))\) along
  \[\QCoh(\gg(\bb G))\to \Mod_{\oo_M(\gg)_G}(\Sp_G\otimes_{\Sp}\QCoh(M))\]
  and the basechange of \(\mathcal{R}(\widetilde{E\mathcal{F}})\) along
  \[\Sp_G\to \Mod_{\oo_M(\gg)_G}(\Sp_G\otimes_{\Sp}\QCoh(M))\]
  are equivalent.
\end{theorem}
\begin{proof}
  Choose an LNG-presentation \(M\simeq\colim_{i\in I}\Spec R_i\) of \(M\).

  By \cref{rmk:LNG_to_local_affines_family_recollement_basechange} and \cref{thm:betterTemperedFamilyCompletion} the further basechanges of these recollement along
  \[\Mod_{\oo_M(\gg)_G}(\Sp_G\otimes_{\Sp}\QCoh(M))\to \Mod_{R_i(\iota_i^*\gg)_G}(\Sp_G)\]
  are equivalent.
  But these functors are jointly conservative, so we conclude.
\end{proof}

\begin{example}
  Let us illustrate a possible application of \cref{thm:family_completion_LNG_version}.
  The classifying stack \(\mathrm{M}_{\mathrm{Ell}}^{\mathrm{or}}\) of oriented ellipitc curves is LNG, so we may apply \cref{thm:family_completion_LNG_version} to the oriented \(\mathbb{P}\)-divisible group \(\mathrm{E}^\mathrm{or}\) of torsion in the universal oriented elliptic curve.

  Since, by a famous result of Mathew and Meier (\cite[thm.~7.2]{MathewMeier_2015_AffinenessChromaticHomotopy}), \(\mathrm{M}_{\mathrm{Ell}}^{\mathrm{or}}\) is \(0\)-affine, we find that the target category in question is equivalent to
  \[\Mod_{\mathrm{TMF}_G}(\Sp_G),\]
  that is to modules over genuine \(G\)-equivariant topological modular forms.

  Since \(\mathrm{M}_{\mathrm{Ell}}^{\mathrm{or}}\) admits an LNG-presentation such that, over each affine, the heights of \(\mathrm{E}^\mathrm{or}\) are bounded above by \(2\), \cref{rmk:boundedHeightBoundedGeneration} tells us that, for every finite abelian group \(G\),
  \[U'_{\mathrm{E}^\mathrm{or}}(\mathcal{F}_2\vert_G)\to \mathrm{E}^\mathrm{or}(\bb G)\]
  is the open immersion associated to the empty open subset.
  Here \(\mathcal{F}_2\) denotes the global family of finite abelian groups that can be generated by \(2\) or less elements.
  As a consequence, we find that (similarly to \cref{prop:rightKanExtension})
  \[\Glo_\fin^\op\to \CAlg(\Sp),\,\bb G\to \mathrm{TMF}^G\]
  is right Kan extended from its restriction to \(\Glo_{\mathcal{F}_2}^\op\).
  A generalization of this result to compact Lie groups was previously obtained by Gepner and Meier in (so-far) unpublished work.
\end{example}
\begin{remark}
  Assume that \(M\) is \(0\)-affine.
  Then also \(\gg(\bb G)\) is \(0\)-affine, so that
  \[\QCoh(\gg(\bb G))\simeq \Mod_{\Gamma\O_M(\gg)^G}(\Sp).\]
  It is unclear whether or not \(\mathcal{R}(I_\gg(\mathcal{F}))\) is then equivalent to a recollement associated to a finitely generated ideal in \(\pi_0\Gamma\O_M(\gg)^G\).
  Thus, while the affine case has concrete computational ramifications, it seems that the \(0\)-affine case is mostly useful to prove qualitative statements.
\end{remark}

\appendix

\section{Some spectral algebraic geometry}\label{sec:stacks}
We dedicate this appendix to develope locality statements for various classes of maps of stacks in the sense of \cite[§~3]{BalderramaDaviesLinskens_2025_AffinenessReconstructionComplexperiodic}.
Of particular interest to us are the open immersions.
Defining them and proving their basic properties could be done much terser, but we ended up essentially proving (special cases of) most of the subsequent results anyway.
Since these could be of their own interest, we take a more small-stepped approach.

\subsection{Underlying locales and spaces of stacks}
\begin{definition}
    Given a spectral Deligne-Mumford stack \(\X=(\mathcal{X},\O_\X)\), we denote by \(\mathrm{Sub}(\X)\) the (essentially small) subcategory of \((-1)\)-truncated objects of \(\mathcal{X}\), which is a locale.

    A morphism of spectral Deligne-Mumford stacks \(f:\X\to \Y\) has an underlying geometric morphism \(f^*\dashv f_*:\mathcal{X}\to\mathcal{Y}\), and we denote by \(\mathrm{Sub}(f)=f^*:\mathrm{Sub}(\Y)\to \mathrm{Sub}(\X)\) the corresponding morphism of frames.
    This defines a functor
    \[\mathrm{Sub}:\SpDM^\nc\to \mathrm{Loc},\]
    and we call \(\mathrm{Sub}(\X)\) the {\it underlying locale} of \(\X\).
\end{definition}
\begin{lemma}
    The underlying locale functor is a sheaf for the étale topology on \(\SpDM^\nc\).
\end{lemma}
\begin{proof}
    By construction, \(\mathrm{Sub}\) factors over the forgetful functor \(\SpDM^\nc\to \mathrm{RTop}\) to \(\infty\)-topoi and geometric morphisms, and an étale cover of spectral Deligne-Mumford stacks induces an étale cover of \(\infty\)-topoi.

    The functor that associates to an \(\infty\)-topos the underlying set of its frame of \((-1)\)-truncated objects is representable on \(\mathrm{RTop}\):
    we have that
    \[\mathrm{Sub}(\mathcal{X})\simeq \mathrm{Loc}(\mathrm{Sub}(\mathcal{X}),\{0<u<1\})\simeq \mathrm{RTop}(\mathcal{X},\Shv(\{0<u<1\})).\]
    Now, since the étale topology on \(\mathrm{RTop}\) is subcanonical \cite[prop.~3.6]{Carchedi_2020_HigherOrbifoldsDeligneMumford}, it follows that the composite
    \[\mathrm{RTop}^{op}\xrightarrow{\mathrm{Sub}}\mathrm{Loc}^{op}\simeq \mathrm{Frm}\to \mathrm{Set}\]
    is an étale sheaf.
    But the functor \(\mathrm{Frm}\to \mathrm{Set}\) preserves and reflects products and equalizers by \cite[§~IV.3.1,\,IV.3.2]{PicadoPultr_2012_FramesLocalesTopology}, so we conclude.
\end{proof}
\begin{lemma}\label{lem:zariski_space_fpqc_sheav}
    The functor
    \[\CAlg\to \Top^{op},\,A\mapsto |\Spec A|\]
    that maps an \(\mathbb{E}_\infty\)-ring \(A\) to the Zariski space of \(\pi_0 A\) is a sheaf for the fpqc topology.
\end{lemma}
\begin{proof}
    According to \cite[prop.~A.3.3.1]{Lurie_SAG} we need to verify two things:
    \begin{enumerate}
        \item For a finite family \(\{A_i\}_{i=1,\ldots,n}\) we have that the natural map
        \[\coprod_{i=1}^n|\Spec A_i|\to |\Spec\prod_{i=1}^nA_i|\]
        is a homeomorphism, and
        \item for a faithfully flat map \(A\to B\) the natural diagram
        \[|\Spec B\otimes_A B|\rightrightarrows|\Spec B|\to |\Spec A|\]
        is a coequalizer in \(\Top\).
    \end{enumerate}
    
    The first point follows from the classical statement, since \(\pi_0:\CAlg\to \CAlg^\heartsuit\) commutes with products.

    Regarding the second point, note that, since \(A\to B\) is faithfully flat, we have that \(\pi_0A\to \pi_0B\) is faithfully flat, and that \(\pi_0(B\otimes_A B)\simeq (\pi_0B)\otimes_{(\pi_0 A)}(\pi_0 B)\).
    Thus, we may assume that \(A\) and \(B\) are discrete rings.

    By \cite[\href{https://stacks.math.columbia.edu/tag/01JT}{Tag 01JT}]{stacks-project} we have that the natural map
    \[|\Spec B\otimes_A B|\to |\Spec B|\times_{|\Spec A|}|\Spec B|\]
    is surjective, and by \cite[\href{https://stacks.math.columbia.edu/tag/00HP}{Tag 00HP}]{stacks-project} that \(|\Spec B|\to |\Spec A|\) is surjective.
    It follows that the natural map \(f:X\to |\Spec A|\), where \(X\) is the coequalizer in \(\Top\), is bijective on points.

    Let us now show that \(f\) is an open map, thus a homeomorphism.
    By \cite[prop.~1.6.2.2]{Lurie_SAG} the composite
    \[\CAlg\to \Top^\op\xrightarrow{\mathrm{Open}}\mathrm{Loc}^\op\simeq \mathrm{Frm}\to\mathrm{Set},\]
    which maps \(A\) to the set of open subsets of \(|\Spec A|\), is an fpqc-sheaf.
    But the functor \(\mathrm{Frm}\to \mathrm{Set}\) preserves and reflects products and equalizers by \cite[§~IV.3.1,\,IV.3.2]{PicadoPultr_2012_FramesLocalesTopology}, so already the functor that sends \(A\) to the frame of open subsets of \(|\Spec A|\) is an fpqc-sheaf.
    Also, as a left adjoint, \(\mathrm{Open}:\Top\to \mathrm{Loc}\) commutes with colimits, so we get a diagram of frames
\[\begin{tikzcd}
	& {\mathrm{Open}(|\Spec A|)} && \\
	{\mathcal{P}(|\Spec A|)} & {\mathrm{Open}(X)} & {\mathrm{Open}(|\Spec B|)} & {\mathrm{Open}(|\Spec B\otimes_A B|)} \\
	{\mathcal{P}(X)} & {\mathcal{P}(|\Spec B|)}
	\arrow[hook, from=1-2, to=2-1]
	\arrow["{f^{-1}}"', hook, from=1-2, to=2-2]
	\arrow["{q^{-1}}", hook, from=1-2, to=2-3]
	\arrow["{f^{-1}}"', hook, two heads, from=2-1, to=3-1]
	\arrow["{q^{-1}}", hook, from=2-1, to=3-2]
	\arrow["{p^{-1}}", hook, from=2-2, to=2-3]
	\arrow[hook, from=2-2, to=3-1]
	\arrow[shift right, from=2-3, to=2-4]
	\arrow[shift left, from=2-3, to=2-4]
	\arrow[hook, from=2-3, to=3-2]
	\arrow["{p^{-1}}"', hook, from=3-1, to=3-2]
\end{tikzcd}\]
    where \(\mathcal{P}\) denotes the frame of all subsets.
    By the discussion above both \(p^{-1}\) and \(q^{-1}\) are equalizers for the two parallel arrows, so they have the same image, so \(f^{-1}:\mathrm{Open}(|\Spec A|)\to \mathrm{Open}(X)\) is a bijection, so \(f\) is a homeomorphism.
\end{proof}

\begin{lemma}
    Restricted to affine spectral Deligne-Mumford stacks, the underlying locale is a sheaf for the fpqc-topology.
\end{lemma}
\begin{proof}
    By \cite[prop.~1.6.2.2]{Lurie_SAG} The composite functor
    \[\CAlg\xrightarrow{\Spet}(\SpDM^\nc)^\op\xrightarrow{\mathrm{Sub}}\mathrm{Loc}^\op\simeq \mathrm{Frm}\to \mathrm{Set}\]
    is a sheaf for the fpqc topology.
    But the functor \(\mathrm{Frm}\to \mathrm{Set}\) preserves and reflects products and reflects equalizers by \cite[§~IV.3.1,\,IV.3.2]{PicadoPultr_2012_FramesLocalesTopology}, so we conclude.
\end{proof}
\begin{corollary}\label{cor:h_underlying_locale}
    The underlying locale functor \(\mathrm{Sub}:\SpDM^\nc\to\mathrm{Loc}\) factors uniquely through \(h:\SpDM^\nc\to \Stk\), and the resulting functor \(\Stk\to\mathrm{Loc}\) commutes with colimits.
\end{corollary}
\begin{definition}\label{def:underlying_locale_stack}
    We denote the resulting functor \(\Stk\to \mathrm{Loc}\) of \cref{cor:h_underlying_locale} by \(\mathrm{Sub}\), and call \(\mathrm{Sub}(X)\) the {\it underlying locale} of \(X\).

    By \cite[prop.~2.1.2.2]{BalderramaDaviesLinskens_2025_AffinenessReconstructionComplexperiodic} and \cref{lem:zariski_space_fpqc_sheav} we also obtain a colimit preserving functor \(|-|:\Stk\to \Top\), and we call \(|X|\) the {\it underlying topological space} of \(X\).
\end{definition}

\begin{remark}\label{rmk:spaces_vs_locales}
    As the underlying locale of any affine \(\Spec A\) is spatial, and spatial locales are closed under colimits, it follows that \(\mathrm{Sub}(X)\) is spatial for any \(X\in\Stk\).
    In particular, we can recover \(\mathrm{Sub}(X)\) as \(\mathrm{Open}(|X|)\).
    On the other hand, the best we can say about the natural map \(|X|\to \mathrm{Pt}(\mathrm{Sub}(X))\) is  that it is sobrification.
    It follows that for \(\X\in \SpDM^\nc\) the natural map \(|h_\X|\to |\X|\) is sobrification.
\end{remark}
\begin{example}
    While, by construction, \(X\mapsto \mathrm{Sub}(X)\)  and \(X\mapsto |X|\) commute with colimits, this does not hold for \(X\mapsto \mathrm{Pt}(\mathrm{Sub}(X))\):
    Let \(I\) be the category with objects \(\{x\}\coprod\mathbb{C}\coprod \{y\}\) and morphisms identities and unique morphisms \(u_c:c\to x,\,v_c:c\to y\) for each \(c\in \mathbb{C}\).
    Let \(F:I^\op\to \CAlg^\heartsuit\subset \CAlg\) be the diagram defined as
    \[F(x)=\mathbb{C}[s],\,F(y)= \mathbb{C}[t],\,F(c\in\mathbb{C})=\mathbb{C},\,F(u_c)=(\mathbb{C}[s]\xrightarrow{s\mapsto c}\mathbb{C}),\,F(v_c)=(\mathbb{C}[t]\xrightarrow{t\mapsto c}\mathbb{C}).\]
    Then \(T=\colim_{i\in I}\mathrm{Pt}(\mathrm{Sub}(\Spec F(i)))\) has underlying set
    \[\{\eta_s=(0)\subset \mathbb{C}[s]\}\coprod \{[(s-c),(t-c)]\mid c\in \mathbb{C}\}\coprod\{\eta_t=(0)\subset \mathbb{C}[t]\}.\]
    If \(\eta_s\in U\subset T\) is an open neighborhood, then its preimage in \(|\Spec\mathbb{C}[s]|\) is open and contains the generic point, so \(U\) must contain all but finitely many of the points \([(s-c)]=[(t-c)]\).
    On the other hand, the closure of \(\{\eta_t\}\) must contain all the points \([(t-c)]=[(s-c)]\).
    Since \(\mathbb{C}\) is famously infinite, this tells us that \(U\cap \overline{\{\eta_t\}}\neq \emptyset\), so the closed complement of \(U\) can't contain \(\eta_t\), so also \(\eta_t\in U\).
    The same argument also shows that every open containing \(\eta_t\) also contains \(\eta_s\).
    In particular, \(T\) is not a sober space, so can't agree with \(\mathrm{Pt}(\mathrm{Sub}(\colim_{i\in I}\Spec F(i)))\).
    
    What does hold is that \(X\mapsto \mathrm{Pt}(\mathrm{Sub}(X))\) commutes with colimits up to sobrification.
\end{example}

\subsection{Local properties of morphisms of stacks}
\begin{definition}\label{def:stacks_effective_epi_etc}
    Let \(f:X\to Y\) be a map of stacks.
    Denote by \(\check{C}^\bullet(f):\Delta^\op\to \Stk_{/Y}\) the Čech nerve of \(f\).
    Define the {\it image of \(f\)}, \(\mathrm{Im}(f)\to Y\), to be the colimit of \(\check{C}^\bullet(f)\).

    We call \(f\) an {\it effective epimorphism} if \(\mathrm{Im}(f)\to Y\) is an equivalence.

    Let \(P\) be a class of maps of stacks.
    We say that \(P\) is {\it good} if it contains all equivalences, is closed under composition, and is closed under basechange.
    Call a good class of maps {\it local} if, for a pullback square
\[\begin{tikzcd}
	{X'} & X \\
	{Y'} & Y
	\arrow[from=1-1, to=1-2]
	\arrow["{f'}"', from=1-1, to=2-1]
	\arrow["\lrcorner"{anchor=center, pos=0.125}, draw=none, from=1-1, to=2-2]
	\arrow["f", from=1-2, to=2-2]
	\arrow["g"', from=2-1, to=2-2]
\end{tikzcd}\]
    of stacks with \(g\) an effective epimorphism, we have \(f'\in P\) implies \(f\in P\).

    Finally, call a good class of maps {\it finitary} if, for any (small) family \(\{f_i\in P\}_{i\in I}\), we have \(\coprod_{i\in I}f_i\in P\).
\end{definition}

\begin{remark}
    It follows from \cite[props.~2.1.2.1,\,2.1.2.9]{BalderramaDaviesLinskens_2025_AffinenessReconstructionComplexperiodic} that \(\Stk\) is a {\it regular} category as in \cite[def.~2.1.1]{Stefanich_2023_Derived$infty$categoriesExact}.
    Thus, by \cite[prop.~2.1.6]{Stefanich_2023_Derived$infty$categoriesExact}, the class of effective epimorphisms is local.
    Moreover, since coproducts are disjoint and universal in \(\Stk\), effective epimorphisms are finitary.
\end{remark}
\begin{remark}
  What we denote by \(\mathrm{Im}(f)\) is called the {\it descent stack} and deontet by \(\mathrm{D}_f\) in \cite[def.~2.3.1.1]{BalderramaDaviesLinskens_2025_AffinenessReconstructionComplexperiodic}.
\end{remark}

\begin{lemma}\label{lem:colim_coproduct_effective_epi}
    Let \(X\in \Stk\).
    Then, for any small colimit presentation \(X\simeq \colim_{i\in I}\Spec R_i\) by affines, the canonical map
    \[\coprod_{i\in I}\Spec R_i\to X\]
    is an effective epimorphism.
\end{lemma}
\begin{proof}
    Let \(\kappa\) be a regular cardinal such that each \(R_i\) is \(\kappa\)-compact in \(\CAlg\).
    Then, since \(\Spec:\CAlg^\op\to \Stk\) is fully faithful, we may lift the whole diagram to a functor \(I\to (\CAlg^\kappa)^\op\).
    Composing this with \(\Spec_\kappa:(\CAlg^\kappa)^\op\to \Shv^\fpqc_\kappa\) and denoting the resulting colimit by \(X_\kappa\), \cite[lem.~6.2.3.13]{Lurie_HTT} tells us that
    \[\coprod_{i\in I}\Spec_\kappa R_i\to X_\kappa\]
    is an effective epimorphism.

    Since the inclusion \((\CAlg^\kappa)^\op\to \CAlg^\op\) is an exact morphism of sites (\cite[def.~2.1.1.2]{BalderramaDaviesLinskens_2025_AffinenessReconstructionComplexperiodic}) for the fpqc topologies, the induced functor
    \[\Shv^\fpqc_\kappa\simeq \Stk((\CAlg^\kappa)^\op,\tau_\fpqc)\to \Stk\]
    commutes with small colimits and finite limits.
    In particular, it sends effective epimorphisms to effective epimorphisms.
    This shows that already \(\coprod_{i\in I}\Spec R_i\to X\) is an effective epimorphism.
\end{proof}

\begin{lemma}\label{lem:effective_epi_to_affine_local_section}
    Let \(f:X\to \Spec A\) be a map of stacks with affine target.
    Then \(f\) is an effective epimorphism if and only if there exists a faithfully flat map \(A\to B\) such that \(\Spec B\to \Spec A\) factors though \(f\).
\end{lemma}
\begin{proof}
    For the 'if' part, let \(\kappa\) be a regular cardinal such that both \(A\) and \(B\) are \(\kappa\)-compact.
    Then \(\Spec_\kappa B\to \Spec_\kappa A\) is an effective epimorphism in \(\Shv^\fpqc_\kappa\).
    Since the inclusion \((\CAlg^\kappa)^\op\to \CAlg^\op\) is an exact morphism of sites (\cite[def.~2.1.1.2]{BalderramaDaviesLinskens_2025_AffinenessReconstructionComplexperiodic}) for the fpqc topologies, the induced functor
    \[\Shv^\fpqc_\kappa\simeq \Stk((\CAlg^\kappa)^\op,\tau_\fpqc)\to \Stk\]
    commutes with small colimits and finite limits.
    In particular, it sends effective epimorphisms to effective epimorphisms, so \(\Spec B\to \Spec A\) is an effective epimorphism.
    Now it follows from \cite[prop.~2.1.6(6)]{Stefanich_2023_Derived$infty$categoriesExact} that also \(f\) is an effective epimorphism.

    For the 'only if' part, choose a small colimit presentation \(X\simeq \colim_{i\in I}\Spec R_i\), and let \(\kappa\) be a regular cardinal such that \(A\) and all of the \(R_i\) are \(\kappa\)-compact.
    We then may lift the effective epimorphism
    \[q:\coprod_{i\in I}\Spec R_i\to \Spec A\]
    to a map
    \[q_\kappa:\coprod_{i\in I}\Spec_\kappa R_i\to \Spec_\kappa A\]
    in \(\Shv^\fpqc_\kappa\).
    We claim that, for \(\kappa\) large enough, \(q_\kappa\) is an effective epimorphism.
    In fact, since the colimit
    \[\Stk\simeq \colim_\lambda\Shv_\lambda^\fpqc\]
    is filtered, an inverse of \(\mathrm{Im}(q)\to \Spec A\) must exist at some finite stage.
    
    For such a \(\kappa\), there exists a faitfully flat map \(A\to B\), with \(B\) also \(\kappa\) compact, such that \(\Spec_\kappa B\to \Spec_\kappa A\) factors over \(q_\kappa\) by \cite[cor.~B.7]{HesselholtPstragowski_2024_DiracGeometryII}.
    Mapping this factorization into \(\Stk\) gives the desired result.
\end{proof}

\begin{definition}
    Let \(P\) be a good class of maps of stacks.
    We say that \(P\) is {\it detected on affines} if a map \(f:X\to Y\) is in \(P\) if and only if, for all maps \(\Spec A\to Y\) with affine source, the basechange \(X_A\to \Spec A\) is in \(P\).
\end{definition}

\begin{proposition}\label{prop:local_finitary_vs_detected_on_affines}
    Let \(P\) be a good class of maps of stacks.
    Then \(P\) is detected on affines if \(P\) is local and finitary.
    The converse holds if \(P\) is additionally closed under finite coproducts of maps with affine target and if, for any map \(f:X\to \Spec A\) and any faithfully flat map \(A\to B\), the basechange \(X_B\to \Spec B\) being in \(P\) implies that \(f\) is in \(P\).
\end{proposition}
\begin{proof}
    Let us first assume that \(P\) is local and finitary, and let \(f:X\to Y\) be a map of stacks.
    If \(f\) is in \(P\), then every basechange of \(f\) to an affine is in \(P\) by definition.

    Conversely, assume that every basechange of \(f\) to an affine is in \(P\).
    Fix a small colimit presentation \(Y\simeq \colim_{i\in I}\Spec R_i\).
    Then each basechange \(X_i\to \Spec R_i\) is in \(P\), so
    \[\coprod_{i\in I}X_i\to \coprod_{i\in I}\Spec R_i\]
    is in \(P\) since \(P\) is finitary.
    But this is the basechange of \(f\) along the effective epimorphism \(\coprod_{i\in I}\Spec R_i\to Y\), so \(f\) is in \(P\) since \(P\) is local.

    Now let us assume that \(P\) is detected on affines.
    Let \(\{f_i:X_i\to Y_i\}_{i\in I}\) be a family of maps in \(P\), and let \(g:\Spec A\to \coprod_{i\in I}Y_i\).
    Fix small colimit presentations \(Y_i\simeq \colim_{j\in J_i}\Spec R_{i,j}\).
    The resulting map
    \[q:\coprod_{i\in I}\coprod_{j\in J_i}\Spec R_{i,j}\to \coprod_{i\in I} Y_i\]
    is an effective epimorphism, so there is a faithfully flat map \(h:A\to B\) such that
    \[\Spec B\xrightarrow{h} \Spec A\xrightarrow{g} \coprod_{i\in I} Y_i\]
    factors over \(q\), say \(g\circ h=q\circ p\).

    Let \(\kappa\) be such that all of \(B\) and the \(R_{i,j}\) are \(\kappa\)-compact.
    Since the defining colimit
    \[\Stk\simeq \colim_\lambda\Shv_\lambda^\fpqc\]
    is filtered, there is some \(\lambda\geq \kappa\) such that \(p\) lifts to a map
    \[p_\lambda:\Spec_\lambda B\to \coprod_{i\in I}\coprod_{j\in J_i}\Spec_\lambda R_{i,j}\]
    in \(\Shv_\lambda^\fpqc\).
    Now \cite[prop.~A.3.1.3]{Lurie_SAG} tells us that \(\Shv_\lambda^\fpqc\) is a coherent topos and that \(\Spec_\lambda B\) is a coherent object in it.
    In particular, there must be a finite subset \(S\subset \prod_{i\in I}J_i\) such that \(p_\lambda\) factors over \(\coprod_{(i,j)\in S}\Spec_\lambda R_{i,j}\).
    It follows that the same is true for the map \(p\).

    From this we conclude that there must be a finite subset \(F\subset I\) such that \(g\) factors over \(\coprod_{i\in F}Y_i\).
    Thus, \(\Spec A\simeq \coprod_{i\in F}Z_i\), and on each \(Z_i\) \(g\) factors over the inclusion of \(Y_i\).
    But as disjoint summands of an affine, each \(Z_i\) is affine it self.
    Thus the basechange along \(g\) is the finite coproduct of the basechanges to the affines \(Z_i\), and these are all in \(P\).
    If \(P\) is closed under finite coproducts of maps with affine target, then also the basechange along \(g\) must be in \(P\).
    Thus, \(P\) is finitary.

    Finally let \(f:X\to Y\) and \(g:Y'\to Y\) be maps of stacks, with \(g\) an effective epimorphism, and assume thate base change of \(f\) along \(g\) is in \(P\).
    Let \(h:\Spec A \to Y\) be any map from an affine.
    Then we get a faithfully flat map \(i:A\to B\) such that \(h\circ i\) factors over \(g\).
    In particular, the base change of \(f\) to \(\Spec B\) is in \(P\), so our assumptions show that also the basechange of \(f\) to \(\Spec A\) is in \(P\).
    Since \(h\) was arbitrary and \(P\) is detected on affines, this shows that \(f\) is in \(P\).
    Thus, \(P\) is local.
\end{proof}

\begin{proposition}
    The class of affine maps of stacks is local and finitary.
\end{proposition}

\begin{proof}
    It follows directly from the definition that the class of affine maps is good and detected on affines.
    Thus, it suffices to prove that it is closed under finite coproducts and local for faithfully flat maps between affines.

    Clearly the empty coproduct of maps is affine.
    Let \(f:X\to \Spec A\) and \(f':X'\to \Spec A'\) be affine maps.
    In particular, both \(X\) and \(X'\) are affine stacks, say \(\Spec B\) and \(\Spec B'\).
    Since finite coproducts of affine stacks are affine,
    we see that \(f\coprod f'\) has affine source and target, thus is an affine map.

    Now let \(f:X\to \Spec A\) and \(g:\Spec B\to \Spec A\) be faithfully flat such that the basechange of \(f\) along \(g\) is affine, say \(h:X_B=\Spec C\to \Spec B\).

    We want to show that \(X\) is affine, or equivalently that the canonical map
    \[a_X:X\to \Spec \mathrm{End}_{\QCoh(X)}(\O_X)\]
    is an equivalence.
    Note that this is canonically a map in \(\Stk_{/\Spec A}\).

    Since \(\Spec B\to \Spec A\) is affine, \cite[props.~2.2.2.5,~2.2.1.11]{BalderramaDaviesLinskens_2025_AffinenessReconstructionComplexperiodic} tell us that the natural square
\[\begin{tikzcd}
	{\QCoh(X)} & {\QCoh(X_B)} \\
	{\QCoh(\Spec A)} & {\QCoh(\Spec B)}
	\arrow[from=1-1, to=1-2]
	\arrow[from=2-1, to=1-1]
	\arrow[from=2-1, to=2-2]
	\arrow[from=2-2, to=1-2]
\end{tikzcd}\]
    is a pushout in \(\CAlg(\PrLst)\).
    Consequently, the basechange of \(a_X\) to \(\Stk_{/\Spec B}\) is the affinization map
    \[a_{X_B}:X_B\to \Spec \mathrm{End}_{\QCoh(X_B)}(\O_{X_B}),\]
    which is an equivalence.

    In fact, the above works for any map \(\Spec R\to \Spec A\) that factors over \(\Spec B\to \Spec A\).
    Since \(A\to B\) is faithfully flat, we have that
    \[\Spec A\simeq \colim_{\bullet\in\Delta^{op}}\Spec(B^{\otimes_A\bullet +1}),\]
    so by descent we get
    \[\Stk_{/\Spec A}\simeq \lim_{\bullet\in\Delta^{op}} \Stk_{/ \Spec(B^{\otimes_A\bullet +1})}.\]
    By the previous discussion the basechange of \(a_X\) to each \(\Spec (B^{\otimes_A\bullet +1})\) is an equivalence, so this shows that already \(a_X\) is an equivalence.
\end{proof}

\begin{lemma}\label{lem:monomorphism_local_finitary}
    The class of monomorphisms of stacks is local and finitary.
\end{lemma}
\begin{proof}
    By definition, a map \(f:X\to Y\) of stacks is a monomorphism if the relative diagonal \(X\to X\times_Y X\) is an equivalence.
    Equivalently, for any stack \(T\), the induced map
    \[\Stk(T,X)\to\Stk(T,Y)\]
    is \((-1)\)-truncated.
    It follows the class of monomorphisms is good.

    Let \(g:Y'\to Y\) be an effective epimorphism, and assume that the basechange \(f':X'\to Y'\) of \(f\) along \(g\) is a monomorphism.
    In particular, the further basechange to each \(\check{C}^\bullet(g)\) is a monomorphism.

    Note that the basechange of a relative diagonal is the relative diagonal of the basechange.
    Thus, the basechange of the map \(X\to X\times_Y X\) in \(\Stk_{/Y}\) to each \(\Stk_{/\check{C}^\bullet(g)}\) is an equivalence, so by descent already \(f\) must be a monomorphism.

    Now let \(\{f_i:X_i\to Y_i\}_{i\in I}\) be a (small) family of monomorphisms.
    Since coproducts are disjoint in \(\Stk\), the base change of
    \[\coprod_{i\in I}X_i\to (\coprod_{i\in I}X_i)\times_{(\coprod_{i\in I}Y_i)}(\coprod_{i\in I}X_i)\]
    to each \(Y_j\) is the relative diagonal of \(f_j\), which is an equivalence by assumption.
    Thus, by descent, already \(\coprod_{i\in I}f_i\) is a monomorphism.
\end{proof}

Recall the definition of a closed immersion of stacks from \cite[def.~5.5.3]{BalderramaDaviesLinskens_2026_AmbidextrousGlobalSpectra}.
\begin{lemma}\label{lem:closed_immersion_local_finitary}
    The class of closed immersions of stacks is local and finitary.
\end{lemma}
\begin{proof}
    From the definition it is clear that the class of closed immersions is good and detected on affines.
    Thus, it will suffice to verify the extra conditions of \cref{prop:local_finitary_vs_detected_on_affines}.

    Let \(\{f_i:X_i\to \Spec R_i\}_{i\in I}\) be a finite family of closed immersions.
    Since closed immersions are affine maps, we have \(X_i\simeq \Spec S_i\), and the induced maps \(\pi_0 R_i\to \pi_0 S_i\) are surjective.

    Since \(\Spec\) commutes with finite coproducts, we must show that the induced map
    \[\pi_0\prod_{i\in I}R_i\to \pi_0\prod_{i\in I}S_i\]
    is surjective.
    But \(\pi_0\) commutes with finite products.

    Now let \(A\to C\) be any map, let \(A\to B\) be faithfully flat, and assume that \(\pi_0 B\to \pi_0B\otimes_A C\) is surjective.
    That \(A\to B\) is faithfully flat means that \(\pi_0A\to \pi_0B\) is injective, \(\pi_0B\) and \(\pi_0B/A\simeq \pi_0B/\pi_0A\) are flat \(\pi_0A\)-modules, and
    \[\pi_*B\simeq \pi_0B\otimes_{\pi_0 A}\pi_*A.\]
    A quick diagram chase now shows that also \(\pi_0A\to \pi_0 C\) is surjective.

    Finally, let \(A\to C\) be any map, let \(A\to B\) be faithfully flat, and assume that \(\Spec B\otimes_A C\to \Spec B\) is a closed immersion.
    We need to show that, for any map \(A\to T\), the induced map \(\pi_0T\to \pi_0 T\otimes_A C\) is surjective.

    The map \(T\to B\otimes_A T\) is still faithfully flat, and by assumption the map \(B\otimes_A T\to (B\otimes_A C)\otimes_A T\simeq (B\otimes_A C)\otimes_B (B\otimes_A T)\) is surjective on \(\pi_0\).
    By the previous argument we conclude that also \(\pi_0T\to \pi_0 T\otimes_A C\) is surjective.
\end{proof}
\begin{remark}
    Note that for a map \(\Spec B\to \Spec A\) to be a closed immersion of stacks, it is not enough for \(\pi_0A\to \pi_0B\) to be surjective, not even when \(A\) and \(B\) are connective.
    
    A counter example is given by the map \(i:\ko\to \ku\):
    This is surjective on \(\pi_0\), but its basechange along \(\ko\to \ku\to \KU\) is, by Wood's theorem, a map \(\KU\to \KU\oplus \KU\), which can't be surjective on \(\pi_0\).

    Another counter example is the map \(\mathbb{Z}\to \mathbb{F}_p\), whose basechange along \(\mathbb{Z}\to \mathbb{F}_p\to \map(\Sigma^\infty_+ S^1,\mathbb{F}_p)\)
    can't be surjective on \(\pi_0\).

    Though, if the map \(A\to B\) admits a section as \(A\)-modules, then \(\Spec B\to \Spec A\) is a closed immersion.
\end{remark}
\begin{example}
    Despite the name, a closed immersion of stacks does not necessarily induce a closed embedding on underlying topological spaces:

    Let \(*\to BU(1)\) be any point and consider the induced map
    \[\F_2^{BU(1)}\to \F_2^*\simeq \F_2.\]
    This admits a section as \(\mathbb{E}_\infty\) rings, so
    \[i:\Spec \F_2\to \Spec \F_2^{BU(1)}\]
    is a closed immersion.
    Consider the base change of this along the map \(\F_2^{BU(1)}\to (\F_2P)^{BU(1)}\), where \(\F_2P\) is the even periodic version of \(\F_2\).

    We have \(\pi_* \F_2^{BU(1)}\simeq \F_2[x]\) with \(|x|=-2\), and \(\pi_*\F_2P\simeq \F_2[u^\pm]\) with \(|u|= 2\).
    There is a cofiber sequence
    \[\Sigma^{-2}\F_2^{BU(1)}\xrightarrow{\cdot x}\F_2^{BU(1)}\to \F_2\]
    in \(\F_2^{BU(1)}\)-modules, so the basechange of \(\F_2^{BU(1)}\to \F_2\) along \(\F_2^{BU(1)}\to (\F_2P)^{BU(1)}\) is given by \(\F_2P\).
    If we further basechange along the localization
    \[(\F_2P)^{BU(1)}\to (\F_2P)^{BU(1)}[1/ux]\]
    we get the zero ring.

    Now let \(Z\to X\) be the row-wise colimit of the diagram
\[\begin{tikzcd}
	{\Spec \F_2} & {\Spec 0} & {\Spec \F_2} \\
	{\Spec \F_2^{BU(1)}} & {\Spec (\F_2P)^{BU(1)}[1/ux]} & {\Spec \F_2^{BU(1)}}
	\arrow[from=1-1, to=2-1]
	\arrow[from=1-2, to=1-1]
	\arrow[from=1-2, to=1-3]
	\arrow["\lrcorner"{anchor=center, pos=0.125, rotate=-90}, draw=none, from=1-2, to=2-1]
	\arrow[from=1-2, to=2-2]
	\arrow["\lrcorner"{anchor=center, pos=0.125}, draw=none, from=1-2, to=2-3]
	\arrow[from=1-3, to=2-3]
	\arrow[from=2-2, to=2-1]
	\arrow[from=2-2, to=2-3]
\end{tikzcd}\]
    where both squares are pullbacks and all vertical maps are closed immersions.
    By universality of colimits, and since closed immersions are local, we find that \(Z\to X\) is a closed immersion.
    But the map \(|Z|\to |X|\) is the row-wise colimit of the diagram
\[\begin{tikzcd}
	{*} & \emptyset & {*} \\
	{*} & {*} & {*}
	\arrow[from=1-1, to=2-1]
	\arrow[from=1-2, to=1-1]
	\arrow[from=1-2, to=1-3]
	\arrow[from=1-2, to=2-2]
	\arrow[from=1-3, to=2-3]
	\arrow[from=2-2, to=2-1]
	\arrow[from=2-2, to=2-3]
\end{tikzcd}\]
    which is \(*\coprod*\to *\), which is not even injective.
    It follows that the map \(\mathrm{Sub}(Z)\to \mathrm{Sub}(X)\) is not a closed embedding of locales.
\end{example}

\begin{lemma}\label[lemma]{lem:colim_of_adjointables_through_adjointables}
  Let \(F:I\to \Fun(\Delta^1\times \Delta^1,\Stk)\) be a small diagram such that, for every object \(i\in I\) and every map \(e:i\to j\) in \(I\), the squares
\[\begin{tikzcd}
	{F_{00}(i)} & {F_{10}(i)} & {F_{00}(i)} & {F_{10}(i)} & {F_{01}(i)} & {F_{11}(i)} \\
	{F_{01}(i)} & {F_{11}(i)} & {F_{00}(j)} & {F_{10}(j)} & {F_{01}(j)} & {F_{11}(j)}
	\arrow[from=1-1, to=1-2]
	\arrow[from=1-1, to=2-1]
	\arrow[from=1-2, to=2-2]
	\arrow[from=1-3, to=1-4]
	\arrow[from=1-3, to=2-3]
	\arrow[from=1-4, to=2-4]
	\arrow[from=1-5, to=1-6]
	\arrow[from=1-5, to=2-5]
	\arrow[from=1-6, to=2-6]
	\arrow[from=2-1, to=2-2]
	\arrow[from=2-3, to=2-4]
	\arrow[from=2-5, to=2-6]
\end{tikzcd}\]
  are adjointable.
  Then the colimit
\[\begin{tikzcd}
	{\colim F_{00}} & {\colim F_{10}} \\
	{\colim F_{01}} & {\colim F_{11}}
	\arrow[from=1-1, to=1-2]
	\arrow[from=1-1, to=2-1]
	\arrow[from=1-2, to=2-2]
	\arrow[from=2-1, to=2-2]
\end{tikzcd}\]
  is adjointable.
\end{lemma}
\begin{proof}
  We need to show that the Beck-Chevalley transformation is an equivalence for each object \(M\in \QCoh(\colim F_{01})\).
  Since, by construction, the map
  \[\coprod_{i\in I}F_{10}(i)\to \colim F_{10}\]
  is an effective epimorphism, the functor
  \[\QCoh(\colim F_{10})\to \QCoh(\coprod_{i\in I}F_{10}(i))\simeq \prod_{i\in I}\QCoh(F_{10}(i))\]
  is conservative, so it suffices to show that to it is an equivalence after pulling back to each \(F_{10}(i)\).
  Consider the diagram
\[\begin{tikzcd}
	& {F_{00}(i)} && {F_{10}(i)} \\
	{\colim F_{00}} && {\colim F_{10}} \\
	& {F_{01}(i)} && {F_{11}(i)} \\
	{\colim F_{01}} && {\colim F_{11}}
	\arrow[from=1-2, to=1-4]
	\arrow[from=1-2, to=2-1]
	\arrow[from=1-2, to=3-2]
	\arrow[from=1-4, to=2-3]
	\arrow[from=1-4, to=3-4]
	\arrow[from=2-1, to=2-3]
	\arrow[from=2-1, to=4-1]
	\arrow[from=2-3, to=4-3]
	\arrow[from=3-2, to=3-4]
	\arrow[from=3-2, to=4-1]
	\arrow[from=3-4, to=4-3]
	\arrow[from=4-1, to=4-3]
\end{tikzcd}\]
  where the back square is adjointable by assumption, and the top and bottom squares are adjointable by \cite[prop.~2.2.1.9]{BalderramaDaviesLinskens_2025_AffinenessReconstructionComplexperiodic} (note that the \(0\)-(semi)affineness assumption is not needed for the adjointability result).
  Chasing through the diagram gives the desired result.
\end{proof}

\begin{lemma}\label[lemma]{lem:universal_zero_affine_local_finitary}
  The class of universally \(0\)-affine maps of stacks is local and finitary.
\end{lemma}
\begin{proof}
  It is clear from the definition that it is a good class, and detected on affines.

  Let \(\{f_i:X_i\to \Spec R_i\}_{i=1,\ldots,n}\) be a finite collection of universally \(0\)-affine morphisms with affine targets.
  For a diagram
\[\begin{tikzcd}
	Z & {\Spec B} \\
	Y & {\Spec A} \\
	{\coprod_{i=1}^nX_i} & {\coprod_{i=1}^n\Spec R_i}
	\arrow["h"{description}, from=1-1, to=1-2]
	\arrow[from=1-1, to=2-1]
	\arrow["\lrcorner"{anchor=center, pos=0.125}, draw=none, from=1-1, to=2-2]
	\arrow[from=1-2, to=2-2]
	\arrow["g"{description}, from=2-1, to=2-2]
	\arrow[from=2-1, to=3-1]
	\arrow["\lrcorner"{anchor=center, pos=0.125}, draw=none, from=2-1, to=3-2]
	\arrow[from=2-2, to=3-2]
	\arrow[from=3-1, to=3-2]
\end{tikzcd}\]
  with both squares pullbacks, we need to show that both \(g\) and \(h\) are \(0\)-affine, and that the top square is adjointable.

  Note that we have decompositions
  \[\Spec A\simeq \coprod_{i=1}^n\Spec A_i\text{ and }\Spec B\simeq \coprod_{i=1}^n \Spec B_i,\]
  and the right vertical maps are coproducts of maps \(\Spec B_i\to \Spec A_i\to \Spec R_i\).
  Since each \(f_i\) is universally \(0\)-affine, we conclude that the top square is a coproduct of adjointable squares, and thus adjointable itself by \cref{lem:colim_of_adjointables_through_adjointables}.

  This also gives us that both \(h\) and \(g\) are coproducts of \(0\)-affine morphisms, so they are \(0\)-affine them selves.

  Now, let \(q:\Spec B\to \Spec A\) be a faithfully flat map, and let \(f:X\to \Spec A\) be a map such that its basechange \(X_B\to \Spec B\) is universally \(0\)-affine.
  Consider a diagram
\[\begin{tikzcd}
	{X_S} & {\Spec S} \\
	{X_R} & {\Spec R} \\
	X & {\Spec A}
	\arrow[from=1-1, to=1-2]
	\arrow[from=1-1, to=2-1]
	\arrow["\lrcorner"{anchor=center, pos=0.125}, draw=none, from=1-1, to=2-2]
	\arrow[from=1-2, to=2-2]
	\arrow[from=2-1, to=2-2]
	\arrow[from=2-1, to=3-1]
	\arrow["\lrcorner"{anchor=center, pos=0.125}, draw=none, from=2-1, to=3-2]
	\arrow[from=2-2, to=3-2]
	\arrow[from=3-1, to=3-2]
\end{tikzcd}\]
  where both squares are pullbacks.
  Since \(q\) is affine and \(X_B\to \Spec B\) is universally \(0\)-affine, taking the basechange of the top square along all the maps \(\check{C}^\bullet(q)\to \Spec A\) yields a functor \(F:\Delta^{op}\to \Fun(\Delta^1\times\Delta^1,\Stk)\) that fulfills the assumptions of \cref{lem:colim_of_adjointables_through_adjointables}.
  By universality of colimits, the colimit of \(F\) is the top square in the diagram above, so it is adjointable.

  Finally, restricting \(F\) to \(F_{00}\to F_{10}\) and \(F_{01}\to F_{11}\) and applying \cite[propo.~2.2.1.9]{BalderramaDaviesLinskens_2025_AffinenessReconstructionComplexperiodic} yields that both \(X_R\to \Spec R\) and \(X_S\to \Spec S\) are \(0\)-affine.
\end{proof}

\subsection{Open immersions of stacks}

\begin{definition}\label{def:open_immersion_stk}
    We call a map of stacks \(i:U\to X\) an {\it open immersion} if
    \begin{enumerate}
        \item the induced map \(|i|:|U|\to |X|\) on underlying topological spaces is an open embedding, and
        \item the natural square
\[\begin{tikzcd}
	{\Stk(T,U)} & {\Stk(T,X)} \\
	{\mathrm{Top}(|T|,|U|)} & {\mathrm{Top}(|T|,|X|)}
	\arrow[from=1-1, to=1-2]
	\arrow[from=1-1, to=2-1]
	\arrow["\lrcorner"{anchor=center, pos=0.125}, draw=none, from=1-1, to=2-2]
	\arrow[from=1-2, to=2-2]
	\arrow[from=2-1, to=2-2]
\end{tikzcd}\]
        is cartesian for every \(T\in\Stk\).
    \end{enumerate}
    Said differently, \(i\) is a \(|-|\)-cartesian edge over an open embedding.
    
    It is a {\it quasi-compact} open immersion if, for every quasi-compact open \(C\subset |X|\), its preimage \(|i|^{-1}(C)\) is quasi-compact in \(|U|\).
    That is, if \(|i|:|U|\to |X|\) is a quasi-compact map.
\end{definition}
\begin{remark}\label{rmk:open_immersion_on_affines_are_sub_schemes}
    The second point shows that an open immersion is a monomorphism of stacks, and open immersions with target \(X\) are uniquely determined by their associated open subset of \(|X|\).
    Also note that, for \(X=\Spec A\) an affine stack, this definition agrees with \cite[def.~3.3.1.1]{BalderramaDaviesLinskens_2025_AffinenessReconstructionComplexperiodic}, that is, a map into \(\Spec A\) is an open immersion if and only if it is equivalent to applying \(h:\SpSch^\nc\to \Stk\) to
    \[(U,\O_A\vert_U)\to (|\Spec \pi_0 A|,\O_A)\]
    for some open subset \(U\subset |\Spec \pi_0 A|\).
\end{remark}
\begin{lemma}\label{lem:open_immersion_basechange_to_affines}
    Let \(i:U\to X\) be a (quasi-compact) open immersion, and let \(f:\Spec A\to X\) be a map from an affine.
    Then the basechange \(i_A:U_A\to \Spec A\) of \(i\) along \(f\) is a (quasi-compact) open immersion, and
    \[|i_A|(|U_A|)=|f|^{-1}(|i|(|U|)).\]
\end{lemma}
\begin{proof}
    By construction \(U_A\) is a \((-1)\)-truncated object of \(\Stk_{/\Spec A}\), and by inspection we find that
    \[\Map_{\Stk_{/\Spec A}}(g:T\to \Spec A,i_A:U_A\to\Spec A)\simeq \begin{cases}*& \text{if } |f\circ g|(|T|)\subset |i|(|U|),\\ \emptyset & \text{else.}\end{cases}\]

    Let \(V\to \Spet A\) be the nonconnective spectral open subscheme associated to the open subset \(|f|^{-1}(|i|(|U|))\subset |\Spec A|\simeq |\Spet A|\).
    Since \(h_{(-)}:\SpDM^\nc\to\Stk\) is fully faithful on spectral schemes, and \(\Stk\) is generated under colimits by affines, we find that
    \(h_V\) and \(U_A\) are equivalent \((-1)\)-truncated objects of \(\Stk_{/\Spec A}\), so we conclude.

    The claim about quasi-compactness follows since quasi-compact maps of topological spaces are stable under basechange.
\end{proof}
\begin{lemma}\label{lem:colim_open_embeddings_to_affines}
    Let \(F=(U_i\xrightarrow{\eta_i}\Spec R_i):I\to \Fun(\Delta^1,\Stk)\) be a small diagram such that, for any edge \(e:i\to i'\), the square
\[\begin{tikzcd}
	{U_i} & {U_{i'}} \\
	{\Spec R_i} & {\Spec R_{i'}}
	\arrow[from=1-1, to=1-2]
	\arrow["{\eta_i}"', from=1-1, to=2-1]
	\arrow["\lrcorner"{anchor=center, pos=0.125}, draw=none, from=1-1, to=2-2]
	\arrow["{\eta_{i'}}", from=1-2, to=2-2]
	\arrow[from=2-1, to=2-2]
\end{tikzcd}\]
    is cartesian, and the maps \(\eta_i\) are open immersions.
    Denote the colimit of \(F\) by \(\eta:U\to X\) and the canonical morphisms \(\Spec R_i\to X\) by \(\iota_i\).

    Then \(|\eta|:|U|\to |X|\) is an open embedding with image \(\bigcup_{i\in I}|\iota_i\circ\eta_i|(|U_i|)\),
    and for every \(i\in I\) we have \(|\eta_i|(|U_i|)=|\iota_i|^{-1}(|\eta|(|U|))\).
\end{lemma}
\begin{proof}
    First of all, observe that the functor \(|-|:\Stk\to \Top\) preserves colimits by construction and pullback squares of the form \(F(e:i\to i')\) by \cref{lem:open_immersion_basechange_to_affines}.
    We then get a commutative diagram
\[\begin{tikzcd}
	{\coprod_{i\in I}|U_i|} & {\coprod_{i\in I}|\Spec R_i|} \\
	{|U|} & {|X|}
	\arrow["{|\eta_\bullet|}", hook, from=1-1, to=1-2]
	\arrow["p"', two heads, from=1-1, to=2-1]
	\arrow["q", two heads, from=1-2, to=2-2]
	\arrow["{|\eta|}"', from=2-1, to=2-2]
\end{tikzcd}\]
    where \(p\) and \(q\) are quotient maps.
    Then, in the equivalence relation associated to \(q\), a point is equivalent to one in the image of \(|\eta_\bullet|\) if and only if it is itself in that image.
    Also, two points are equivalent in the relation associated to \(p\) if and only if their images under \(|\eta_\bullet|\) are equivalent with respect to \(q\).
    Putting these together proves that \(|\eta|\) is an open embedding with image as claimed, and the above diagram is a pullback in \(\Top\).
\end{proof}

\begin{lemma}\label{lem:underlying_space_basechange_open_immersion}
    The functor \(|-|:\Stk\to\Top\) preserves pullbacks where one of the legs is an open immersion.
\end{lemma}
\begin{proof}
    Consider a pullback square
\[\begin{tikzcd}
	V & U \\
	Y & X
	\arrow["{f'}"{description}, from=1-1, to=1-2]
	\arrow["{i'}"{description}, from=1-1, to=2-1]
	\arrow["\lrcorner"{anchor=center, pos=0.125}, draw=none, from=1-1, to=2-2]
	\arrow["i"{description}, from=1-2, to=2-2]
	\arrow["f"{description}, from=2-1, to=2-2]
\end{tikzcd}\]
    in \(\Stk\) with \(i\) an open immersion.
    We need to show that \(|V|\to |Y|\) is an open embedding with image \(|f|^{-1}(|i|(|U|))\).
    
    Choose a small colimit presentation
    \(Y\simeq \colim_{j\in J}\Spec R_j\) and denote the canonical maps by \(\iota_j:\Spec R_j\to Y\).
     Since colimits are universal in \(\Stk\), we find that \(i':V\to Y\) is the colimit of a diagram
     \[F:J\to \Fun(\Delta^1,\Stk),\,j\mapsto (\eta_j:V_j\to \Spec R_j),\]
     where \(V_j\) is the basechange of of \(U\) along \(f\circ \iota_j\).

     By \cref{lem:open_immersion_basechange_to_affines} the diagram \(F\) is of the form considered in \cref{lem:colim_open_embeddings_to_affines}, so \(|i'|\) is an open embedding.

    It will suffice to prove that the preimage of \(|i|(|U|)\) in \(Y'\) is \(j(V')\).
    This agree with the image in \(Y'\) of the preimage in \(\coprod_{j\in J}|\Spec R_j|\) of \(|i|(|U|)\).
    By \cref{lem:open_immersion_basechange_to_affines} this agrees with the image of \(\cup_{j\in J} |\eta_j|(|V_j|)\) in \(Y'\), which is exactly \(j(V')\) by \cref{lem:colim_open_embeddings_to_affines}.
\end{proof}
\begin{corollary}\label{cor:open_immersion_basechange}
    Let \(i:U\to X\) be an open immersion of stacks, and let \(f:Y\to X\) be any map.
    Then the basechange \(i':V\to Y\) of \(i\) along \(f\) is an open immersion, and
    \[|i'|(|V|)=|f|^{-1}(|i|(|U|)).\]
    If \(i\) is quasi-compact, then \(i'\) is quasi-compact.
\end{corollary}

\begin{proposition}\label{prop:open_immersion_detected_on_affines}
    Open immersions are detected on affines.
\end{proposition}
\begin{proof}
    As open embeddings of topological spaces contain all equivalences and are closed under compositions, it follows that the class of open immersions of stacks is good.
    
    The 'only if' part follows from \cref{lem:open_immersion_basechange_to_affines}.

    In the other direction, choose a small colimit presentation \(X\simeq \colim_{i\in I}\Spec R_i\) and denote the canonical maps by \(\iota_j:\Spec R_j\to X\).

    Since colimits are universal in \(\Stk\), we find that \(\eta:U\to X\) is the colimit of a diagram
    \[F:I\to \Fun(\Delta^1,\Stk),\,i\mapsto (\eta_i:U_i\to \Spec R_i),\]
    where \(\eta_i\) is the basechange of \(\eta\) along \(\iota_i\).
    Our assumption tells us that this diagram is of the form considered in \cref{lem:colim_open_embeddings_to_affines}, so \(|\eta|:|U|\to |X|\) is an open embedding, \(|\eta|(|U|)=\bigcup_{i\in I}|\iota_i\circ \eta_i|(|U_i|)\), and \(|\iota_i|^{-1}(|\eta|(|U|))=|\eta_i|(|U_i|)\).
    This takes care of the first point of \cref{def:open_immersion_stk}.

    For the second point of \cref{def:open_immersion_stk}, let us first note that \(\eta:U\to X\) is a monomorphism by \cref{lem:monomorphism_local_finitary}.

    Also note that it suffices to check the second point of \cref{def:open_immersion_stk} on affines \(T\simeq \Spec B\), since they generate \(\Stk\) under colimits and \(|-|\) preserves colimits.

    We then get a diagram
\[\begin{tikzcd}
	{\Stk(\Spec B, U)} && \\
	& PB & {\Stk(\Spec B, X)} \\
	& {\Top(|\Spec B|,|U|)} & {\Top(|\Spec B|, |X|)}
	\arrow["\epsilon"{description}, from=1-1, to=2-2]
	\arrow["{\eta_*}"{description}, curve={height=-12pt}, from=1-1, to=2-3]
	\arrow[curve={height=18pt}, from=1-1, to=3-2]
	\arrow["j"{description}, from=2-2, to=2-3]
	\arrow[from=2-2, to=3-2]
	\arrow[from=2-3, to=3-3]
	\arrow["{|\eta|_*}"', from=3-2, to=3-3]
\end{tikzcd}\]
    where the square is cartesian and both \(\eta_*\) and \(j\) are \((-1)\)-truncated maps of spaces.
    Thus, for \(\epsilon\) to be an equivalence (which we need to show), it suffices for \(\pi_0\epsilon\) to be surjective.

    By construction we have
    \[\pi_0PB\simeq \{[f]\in \pi_0\Stk(\Spec B, X)\mid |f|(|\Spec B|)\subset |\eta|(|U|)\}.\]
    Given such an \(f\) we need to show that it factors through \(\eta:U\to X\).
    Consider the pullback square
\[\begin{tikzcd}
	{U_B} & U \\
	{\Spec B} & X
	\arrow[from=1-1, to=1-2]
	\arrow["{\eta_B}"{description}, from=1-1, to=2-1]
	\arrow["\lrcorner"{anchor=center, pos=0.125}, draw=none, from=1-1, to=2-2]
	\arrow["\eta"{description}, from=1-2, to=2-2]
	\arrow["f"{description}, from=2-1, to=2-2]
\end{tikzcd}\]
    If \(\eta_B\) was an equivalence this would give us the desired factorization.

    By \cref{lem:colim_coproduct_effective_epi} and \cref{lem:effective_epi_to_affine_local_section} there is a faithfully flat map \(B\to C\) and a map \(g:\Spec C\to \coprod_{i\in I}\Spec R_i\) such that
    \[(\Spec C\to \Spec B\xrightarrow{f} X)\simeq(\Spec C\xrightarrow{g}\coprod_{i\in I}\Spec R_i\to X).\]
    By construction there exists a lift \(g\) to a map
    \[g_\kappa:\Spec_\kappa C\to \coprod_{i\in I}\Spec_\kappa R_i\]
    in \(\Shv_\kappa^\fpqc\).
    Now \cite[prop.~A.3.1.3]{Lurie_SAG} tells us that \(\Shv_\kappa^\fpqc\) is a coherent topos and that \(\Spec_\kappa C\) is a coherent object in it.
    Therefore, there is a finite subset \(S\subset I\) such that \(g_\kappa\) factors over
    \[\Spec_\kappa \prod_{i\in S}R_i\simeq \coprod_{i\in S}\Spec_\kappa R_i\to\coprod_{i\in I}\Spec_\kappa R_i.\]
    Mapping back to \(\Stk\) we find that \(g\) factors over
    \[\Spec \prod_{i\in S}R_i\simeq \coprod_{i\in S}\Spec R_i\to\coprod_{i\in I}\Spec R_i.\]
    Call the resulting map \(g':\Spec C\to \coprod_{i\in S}\Spec R_i\).

    By assumption we have \(|f|(|\Spec B|)\subset |\eta|(|U|)\), so \(|\coprod_{i\in S}\iota_i|(|g'|(|\Spec C|))\subset |\eta|(|U|)\), so
    \[|g'|(|\Spec C|)\subset \coprod_{i\in S}|\eta_i|(|U_i|).\]
    Since coproducts are disjoint, and since \(\coprod_{i\in S}\Spec R_i\) is affine, \cref{lem:open_immersion_basechange_to_affines} tells us that \(\coprod_{i\in S}U_i\to \coprod_{i\in S}\Spec R_i\) is an open immersion with image \(\coprod_{i\in S}|\eta_i|(|U_i|)\).

    Thus, again by \cref{lem:open_immersion_basechange_to_affines}, the basechange of \(U\to X\) to \(\Spec C\) is an open immersion with image \(|\Spec C|\), so an equivalence.
    Consequently, the basechange of \(\eta_B\) along the effective epimorphism \(\Spec C\to \Spec B\) is an equivalence, so by descent already \(\eta_B\) is an equivalence.
\end{proof}
\begin{lemma}\label{lem:open_immersions_local_finitary}
    The class of open immersions is local and finitary.
\end{lemma}
\begin{proof}
    In view of \cref{prop:open_immersion_detected_on_affines} we only need to verify the extra conditions of \cref{prop:local_finitary_vs_detected_on_affines}.

    Let \(\{\eta_i:U_i\to \Spec R_i\}_{i\in I}\) be a finite family of open immersions.
    By \cref{rmk:open_immersion_on_affines_are_sub_schemes} each \(U_i\) is in the image of the fully faithful functor \(\SpSch^\nc\to\Stk\), which also commutes with finite coproducts, so the coproduct of the \(\eta_i\) is equivalent to applying this functor to
    \[\coprod_{i\in I}(\Spet R_i)\vert_{|\eta_i|(|U_i|)}\to \coprod_{i\in I}\Spet R_i.\]
    The target of this map is equivalent to \(\Spet \prod_{i\in I}R_i\), and this identifies the map with an open immersion in \(\SpSch^\nc\).
    Thus, by \cref{rmk:open_immersion_on_affines_are_sub_schemes}, \(\coprod_{i\in I}\eta_i\) is an open immersion.

    Now let \(\eta: U\to \Spec A\) be a map of stacks, and let \(f:\Spec B\to \Spec A\) be a faithfully flat map such that the basechange \(\eta_B:U_B\to \Spec B\) of \(\eta\) along \(f\) is an open immersion.
    Then \cref{cor:open_immersion_basechange} tells us that the further basechange to each \(\Spec (B^{\otimes_A n+1})\) is an open immersion, so, by universality, \cref{lem:monomorphism_local_finitary}, \cref{lem:open_immersion_basechange_to_affines}, and \cref{lem:colim_open_embeddings_to_affines}, \(|\eta|\) is an open embedding and \(\eta\) is a monomorphism.

    Let \(V\to \Spet A\) be the open substack associated to \(|\eta|(|U|)\subset |\Spec A|=|\Spet A|\).
    Then by \cref{rmk:open_immersion_on_affines_are_sub_schemes} \(h_V\to \Spec A\) is an open immersion, and the monomorphism \(\eta\) factors through it.
    By \cref{cor:open_immersion_basechange} and \cref{lem:colim_open_embeddings_to_affines} the basechange of the resulting map \(U\to h_V\) to each \(\Spec (B^{\otimes_A n+1})\) is an equivalence, so by descent \(U\to h_V\) is an equivalence, so \(\eta\) is an open immersion.
\end{proof}
\begin{corollary}\label{cor:colim_open_immersions_to_affines}
    Let \(F=(U_i\xrightarrow{\eta_i}\Spec R_i):I\to \Fun(\Delta^1,\Stk)\) be as in \cref{lem:colim_open_embeddings_to_affines}.
    Then \(\eta:U\to X\), the colimit of \(F\), is an open immersion.
\end{corollary}

\begin{remark}\label{rmk:open_substacks_exist}
    This also shows that for any stack \(X\) and open \(|U|\subset|X|\) there is an open immersion with image \(|U|\).
    Said differently, \(|-|:\Stk\to\Top\) admits cartesian lifts of open embeddings.
\end{remark}


\section{Lifting operadic structures across monomorphisms}
\label{sec:lifting_monos}

Throughout, fix a fibration of \(\infty\)-operads \(q\colon \mathcal{C}^\otimes \to \mathcal{B}^\otimes\), that is, a map of \(\infty\)-operads which is a categorical fibration \cite[def.~2.1.2.10]{Lurie_HA}.
The reader may keep in mind the case \(\mathcal{B}^\otimes = \mathrm{Comm}^\otimes = \Fin_*\) with \(\mathcal{C}^\otimes\) a symmetric monoidal \(\infty\)-category, which recovers the setting of interest; in this case, \(q\) is even a coCartesian fibration of \(\infty\)-operads.
For \(b \in \mathcal{B}^\otimes\) we write \(\mathcal{C}_b = \mathcal{C}^\otimes_b\) for the fibre, and we abbreviate \(\mathcal{C} = \mathcal{C}_{\langle 1\rangle}\).

We will use repeatedly the following consequence of \cite[prop.~2.1.2.22]{Lurie_HA}: for every object \(X \in \mathcal{C}^\otimes\) and every inert morphism \(\beta\colon q(X) \to Y\) of \(\mathcal{B}^\otimes\), there is an inert morphism \(\bar\beta\colon X \to Y'\) of \(\mathcal{C}^\otimes\) with \(q(\bar\beta) = \beta\), and the inert morphisms of \(\mathcal{C}^\otimes\) are exactly the \(q\)-coCartesian morphisms whose image in \(\mathcal{B}^\otimes\) is inert.
In particular \(q\) admits \(q\)-coCartesian lifts of inert morphisms of \(\mathcal{B}^\otimes\), even though it need not be a coCartesian fibration.

The application we have in mind is the following.
Let \(\mathcal{A}\) be a monoidal \(\infty\)-category and \(\mathcal{M}\) an \(\infty\)-category left-tensored over \(\mathcal{A}\), encoded by a coCartesian fibration of \(\infty\)-operads \(\mathcal{C}^\otimes \to \mathcal{LM}^\otimes\) \cite[def.~4.2.1.19]{Lurie_HA}.
For \(R \in \operatorname{Alg}(\mathcal{A})\) and \(M \in \operatorname{LMod}_R(\mathcal{M})\), we write \(p\colon \operatorname{LMod}_R(\mathcal{M}) \to \mathcal{M}\) for the forgetful functor.

\begin{lemma}\label{lem:lifting_monos_into_modules}
In the situation above, let \(i\colon U \to p(M)\) be a monomorphism in \(\mathcal{M}\).
The space of lifts of \(i\) to a morphism \(U' \to M\) in \(\operatorname{LMod}_R(\mathcal{M})\) is empty or contractible.
It is non-empty if and only if
\[
  R \otimes U \longrightarrow R \otimes p(M) \longrightarrow p(M)
\]
factors through \(i\), where the second map is the action of \(R\) on \(p(M)\).
\end{lemma}

We will deduce this from a general lifting criterion, \cref{lem:lifting_monos_into_O_algebras}, valid for an arbitrary fibration of \(\infty\)-operads.
First, we need to establish some technology.

\begin{lemma}\label{lem:fiber_decomposition}
Let \(b \in \mathcal{B}^\otimes_{\langle n\rangle}\), and choose inert morphisms \(b \to b_i\) with \(b_i \in \mathcal{B}^\otimes_{\langle 1\rangle}\) for \(1 \leq i \leq n\).
Then the inert morphisms of \(\mathcal{C}^\otimes\) lying over \(b \to b_i\) induce an equivalence
\[
  \mathcal{C}_b \xrightarrow{\ \sim\ } \prod_{i=1}^{n} \mathcal{C}_{b_i}.
\]
\end{lemma}
\begin{proof}
This is the equivalence established in the course of proving \cite[prop.~2.1.2.12]{Lurie_HA}, in the implication (a)\(\Rightarrow\)(b).
That argument forms the square of Segal equivalences
\[
\begin{tikzcd}
\mathcal{C}^\otimes_{\langle n\rangle} \arrow[r, "q_{\langle n\rangle}"] \arrow[d, "\simeq"'] &
\mathcal{B}^\otimes_{\langle n\rangle} \arrow[d, "\simeq"] \\
\mathcal{C}^n \arrow[r, "(q_{\langle 1\rangle})^n"'] & \mathcal{B}^n,
\end{tikzcd}
\]
which commutes because \(q\) preserves inert morphisms, and passes to the fibres of the horizontal maps over \(b\) and its image \((b_1,\dots,b_n)\) to obtain \(\mathcal{C}_b \simeq \prod_i \mathcal{C}_{b_i}\).
Of the coCartesian hypothesis on \(q\) it uses only that \(q\) admits coCartesian lifts of inert morphisms in \(\mathcal{B}^\otimes\), and that the fibres of the horizontal maps compute their homotopy fibres; the latter holds because \(q_{\langle n\rangle}\) and \((q_{\langle 1\rangle})^n\) are base changes of the categorical fibration \(q\), so that these strict pullbacks are homotopy pullbacks.
\end{proof}

\begin{definition}\label{def:relative_arrow_operad}
Let \(\Ar_{\mathcal{B}}(\mathcal{C})^\otimes\) denote the pullback
\[
  \Ar_{\mathcal{B}}(\mathcal{C})^\otimes
  := \Ar(\mathcal{C}^\otimes) \times_{\Ar(\mathcal{B}^\otimes)} \mathcal{B}^\otimes,
\]
formed along the diagonal \(\delta\colon \mathcal{B}^\otimes \to \Ar(\mathcal{B}^\otimes)\), with structure map \(r\) to \(\operatorname{Fin}_*\) given by the composite through \(\mathcal{B}^\otimes\).
Its objects lying over \(b\) are the morphisms of the fibre \(\mathcal{C}_b\).
We write \(s,\, t\colon \Ar_{\mathcal{B}}(\mathcal{C})^\otimes \to \mathcal{C}^\otimes\) for evaluation at the source and at the target; both are functors over \(\mathcal{B}^\otimes\).
Given \(\phi \in \Ar_{\mathcal{B}}(\mathcal{C})^\otimes\) over \(b \in \mathcal{B}^\otimes_{\langle n\rangle}\), its \emph{components} \(\phi_i\) (\(1 \leq i \leq n\)) are the images of \(\phi\) under the equivalence of \cref{lem:fiber_decomposition} for the arrow categories, \(\Ar(\mathcal{C}_b) \simeq \prod_i \Ar(\mathcal{C}_{b_i})\).
We define \((\mathcal{C}^{\mathrm{mono}}_{/\mathcal{B}})^\otimes \subseteq \Ar_{\mathcal{B}}(\mathcal{C})^\otimes\) to be the full subcategory spanned by those \(\phi\) all of whose components \(\phi_i\) are monomorphisms in the \(\infty\)-category \(\mathcal{C}^\otimes\).
\end{definition}

\begin{lemma}\label{lem:relative_arrow_mapping_spaces}
The projection \(\Ar_{\mathcal{B}}(\mathcal{C})^\otimes \to \Ar(\mathcal{C}^\otimes)\) is fully faithful, with essential image the arrows of \(\mathcal{C}^\otimes\) lying over equivalences of \(\mathcal{B}^\otimes\).
In particular, for objects \(\Phi, \Psi\) there is a natural equivalence
\[
  \operatorname{Map}_{\Ar_{\mathcal{B}}(\mathcal{C})^\otimes}(\Phi, \Psi)
  \simeq
  \operatorname{Map}_{\mathcal{C}^\otimes}(s\Phi, s\Psi)
  \times_{\operatorname{Map}_{\mathcal{C}^\otimes}(s\Phi, t\Psi)}
  \operatorname{Map}_{\mathcal{C}^\otimes}(t\Phi, t\Psi),
\]
under which \(t\) is the projection to the last factor.
\end{lemma}

\begin{proof}
Evaluation at \(1 \in \Delta^1\) is left adjoint to \(\delta\) with invertible counit, so \(\delta\) is fully faithful; since \(\Ar(q)\) is a categorical fibration, the defining pullback is a homotopy pullback, and the first assertion follows, the essential image consisting of those arrows whose image under \(\Ar(q)\) is equivalent to a degenerate one.
The displayed formula is the standard description of mapping spaces in an arrow category.
\end{proof}

\begin{remark}\label{rmk:mono_notions}
If \(q\) is a coCartesian fibration, i.e.\ \(\mathcal{C}^\otimes\) is a \(\mathcal{B}\)-monoidal \(\infty\)-category, then a morphism of a fibre \(\mathcal{C}_{b_0}\), \(b_0 \in \mathcal{B}^\otimes_{\langle 1\rangle}\), is a monomorphism in \(\mathcal{C}^\otimes\) as soon as it is one in the fibre, since every multimorphism space of \(\mathcal{C}^\otimes\) with target in \(\mathcal{C}_{b_0}\) is corepresented in \(\mathcal{C}_{b_0}\).
For \(\mathcal{B}^\otimes = \mathrm{Comm}^\otimes\) and \(\mathcal{C}\) symmetric monoidal this recovers the naive notion of a pointwise list of monomorphisms.
In general the fibrewise notion is genuinely too weak, since multimorphism spaces of an \(\infty\)-operad are not corepresented; correspondingly, \((\mathcal{C}^{\mathrm{mono}}_{/\mathcal{B}})^\otimes\) is only an operad and not a \(\mathcal{B}\)-monoidal subcategory.
\end{remark}

\begin{lemma}\label{lem:mono_operad_is_operad}
The functor \(\mathrm{pr}_2\) exhibits \(\Ar_{\mathcal{B}}(\mathcal{C})^\otimes\) and \((\mathcal{C}^{\mathrm{mono}}_{/\mathcal{B}})^\otimes\) as \(\infty\)-operads, fibered over \(\mathcal{B}^\otimes\).
A morphism is inert if and only if its images under \(s\) and \(t\) are inert in \(\mathcal{C}^\otimes\).
Moreover, \(s\), \(t\), and the inclusion \(\iota:(\mathcal{C}^{\mathrm{mono}}_{/\mathcal{B}})^\otimes\subset \Ar_{\mathcal{B}}(\mathcal{C})^\otimes\) are maps of \(\infty\)-operads over \(\mathcal{B}^\otimes\).
\end{lemma}

\begin{proof}
We verify the conditions of \cite[def.~2.1.1.10]{Lurie_HA} for \(\Ar_{\mathcal{B}}(\mathcal{C})^\otimes\).
For (1), let \(\Phi\) lie over \(b\) and let \(\alpha\) be an inert morphism of \(\operatorname{Fin}_*\) under \(r(\Phi)\).
Choosing an inert morphism \(\beta\colon b \to b'\) of \(\mathcal{B}^\otimes\) over \(\alpha\), the inert-lift property recalled above provides inert morphisms \(s\Phi \to c\) and \(t\Phi \to c'\) of \(\mathcal{C}^\otimes\) over \(\beta\), which are moreover \(q\)-coCartesian.
Factoring \(s\Phi\to t\Phi \to c'\) through the \(q\)-coCartesian morphism \(s\Phi \to c\) produces an arrow \(\Phi'\colon c \to c'\) of the fibre \(\mathcal{C}_{b'}\), and the pair \((s\Phi \to c,\, t\Phi \to c')\) defines an edge \(\phi:\Phi \to \Phi'\) of \(\Ar_{\mathcal{B}}(\mathcal{C})^\otimes\) over \(\beta\) whose images under \(s\) and \(t\) are inert.
The mapping space formula of \cref{lem:relative_arrow_mapping_spaces} then shows that \(\phi\) is a \(\mathrm{pr}_2:\Ar_\mathcal{B}(\mathcal{C})^\otimes\to \mathcal{B}^\otimes\)-coCartesian lift of \(b\), so an \(r\)-coCartesian lift of \(\alpha\).

For (2), the mapping space formula of \cref{lem:relative_arrow_mapping_spaces} and the fibre decomposition reduce the Segal condition to that of \(\mathcal{C}^\otimes\), computed with the inert morphisms just constructed.

For (3), the existence of tuples follows from condition (3) for \(\mathcal{B}^\otimes\) together with \cref{lem:fiber_decomposition}.

Since the components of an inert pushforward of \(\Phi\) are, up to equivalence, among the components of \(\Phi\), the full subcategory \((\mathcal{C}^{\mathrm{mono}}_{/\mathcal{B}})^\otimes\) contains the targets of the canonical inert morphisms and inherits the operad structure.
The assertions about \(\iota, s, t\) are immediate from the pointwise description of inert morphisms.
\end{proof}

\begin{lemma}\label{lem:mono_operad_target_faithful}
The composite \((\mathcal{C}^{\mathrm{mono}}_{/\mathcal{B}})^\otimes \xrightarrow{\ \iota\ } \Ar_{\mathcal{B}}(\mathcal{C})^\otimes \xrightarrow{\ t\ } \mathcal{C}^\otimes\) is a faithful functor, i.e.\ induces inclusions of connected components on mapping spaces.
\end{lemma}

\begin{proof}
Let \(\Phi, \Psi \in (\mathcal{C}^{\mathrm{mono}}_{/\mathcal{B}})^\otimes\) with \(\Psi\) over \(b' \in \mathcal{B}^\otimes_{\langle n\rangle}\), and fix a morphism \(\bar f\colon r(\Phi) \to b'\) of \(\mathcal{B}^\otimes\).
By \cref{lem:fiber_decomposition} and the mapping space formula of \cref{lem:relative_arrow_mapping_spaces}, the component of \(\operatorname{Map}_{(\mathcal{C}^{\mathrm{mono}}_{/\mathcal{B}})^\otimes}(\Phi, \Psi)\) over \(\bar f\) decomposes as a product, over \(i \in \langle n\rangle^\circ\), of arrow mapping spaces into the unary component \(\Psi_i\); and on each factor \(t\) is the base change of postcomposition with \(\Psi_i\),
\[
  \operatorname{Map}_{\mathcal{C}^\otimes}(s\Phi, s\Psi_i)
  \longrightarrow
  \operatorname{Map}_{\mathcal{C}^\otimes}(s\Phi, t\Psi_i),
\]
which is \((-1)\)-truncated since \(\Psi_i\) is a monomorphism in \(\mathcal{C}^\otimes\).
Hence \(t \circ \iota\) is faithful.
\end{proof}

\begin{remark}\label{rmk:only_targets}
As in the symmetric monoidal case, the proof uses only that the \emph{targets} \(\Psi_i\) are monomorphisms; at the unary level it is the computation of \cite[cor.~A.9]{Ramzi_2023_ElementaryProofNaturality} for the arrow category \(\Ar(\mathcal{C})\) with evaluation at the target.
\end{remark}

\begin{lemma}\label{lem:lifting_monos_into_O_algebras}
Let \(f\colon \mathcal{O}^\otimes \to \mathcal{C}^\otimes\) be a map of \(\infty\)-operads, let \(g\colon K \to \mathcal{O}^\otimes_{\langle 1\rangle}\) be an essentially surjective functor, let \(h\colon K \to (\mathcal{C}^{\mathrm{mono}}_{/\mathcal{B}})^\otimes_{\langle 1\rangle}\) be any functor, and fix an equivalence between \(f \circ g\) and \(t \circ \iota \circ h\).
Regard \(\mathcal{O}^\otimes\) as an operad over \(\mathcal{B}^\otimes\) via \(q \circ f\).
Then the space of diagonal fillers
\[
\begin{tikzcd}
K \arrow[r, "h"] \arrow[d, "g"'] &
(\mathcal{C}^{\mathrm{mono}}_{/\mathcal{B}})^\otimes \arrow[d, "t\circ\iota"] \\
\mathcal{O}^\otimes \arrow[r, "f"'] \arrow[ur, dashed] & \mathcal{C}^\otimes
\end{tikzcd}
\]
where the dotted arrow is a map of \(\infty\)-operads over \(\mathcal{B}^\otimes\) is either empty or contractible.
It is non-empty if and only if the following condition is fulfilled:

For every \(m \geq 0\), every tuple \((k_1,\dots,k_m)\) of objects of \(K\), every object \(k'\) of \(K\), and every active morphism \(o\colon (g(k_1),\dots,g(k_m)) \to g(k')\) in \(\mathcal{O}^\otimes\), the multimorphism
  \[
    (sh(k_1),\ldots,sh(k_n)) \xrightarrow{ (h(k_\bullet))} (th(k_1),\ldots,th(k_n))\xrightarrow{\ f(o)\ } th(k')
  \]
  of \(\mathcal{C}^\otimes\) factors through the monomorphism \(h(k')\colon sh(k') \to th(k')\).

\end{lemma}

\begin{proof}
Write \(\mathcal{S} = \mathcal{B}^\otimes\), regarded as the base of the slices below.
Note that \(t \circ \iota\) is a map of operads over \(\mathcal{S}\) by \cref{lem:mono_operad_is_operad}.

\emph{Step 1: transposition.}
By \cite[rmk.~2.1.4.10]{Lurie_HA} the forgetful functor \(\operatorname{Op}_\infty \to \operatorname{Cat}_\infty\), \(\mathcal{P}^\otimes \mapsto \mathcal{P}^\otimes_{\langle 1\rangle}\), admits a left adjoint \(\operatorname{Fr}(-)^\otimes\).
By \cite[prop.~2.1.4.11]{Lurie_HA} \(\operatorname{Fr}(K)^\otimes\) is the trivial operad on \(K\), so by \cite[rmk.~2.1.1.14]{Lurie_HA} its objects are the finite tuples of objects of \(K\).
Transposing \(g\), \(h\) and the fixed equivalence yields a square in \(\operatorname{Op}_\infty\)
\[
\begin{tikzcd}
\operatorname{Fr}(K)^\otimes \arrow[r, "\tilde h"] \arrow[d, "\tilde g"'] &
(\mathcal{C}^{\mathrm{mono}}_{/\mathcal{B}})^\otimes \arrow[d, "t \circ \iota"] \\
\mathcal{O}^\otimes \arrow[r, "f"'] & \mathcal{C}^\otimes
\end{tikzcd}
\]
whose space of fillers diagonal fillers in \((\operatorname{Op}_\infty)_{/\mathcal{S}}\) is equivalent to the space in question.
All four maps are over \(\mathcal{S}\), and \(\tilde g\) is essentially surjective by by construction.
For a tuple \(k\) and an inert \(\alpha\) we write \(\operatorname{can}_\alpha\colon k \to k|_\alpha\) for the canonical inert morphism of \(\operatorname{Fr}(K)^\otimes\) covering \(\alpha\).

\emph{Step 2: from \((\operatorname{Op}_\infty)_{/\mathcal{S}}\) to \(\operatorname{Cat}_\infty\).}
A functor over \(\mathcal{S}\) between operads is a map of operads if and only if it preserves inert morphisms, a property stable under equivalence; hence the space of fillers in \((\operatorname{Op}_\infty)_{/\mathcal{S}}\) is the union of those components of the space of fillers in \((\operatorname{Cat}_\infty)_{/\mathcal{S}}\) whose underlying functor preserves inert morphisms.
Moreover, the forgetful map from fillers in \((\operatorname{Cat}_\infty)_{/\mathcal{S}}\) to fillers in \(\operatorname{Cat}_\infty\) is an equivalence: promoting a filler and its coherences to the slice over \(\mathcal{S}\) amounts to filling inner horns \(\Lambda^3_2\) in \(\operatorname{Cat}_\infty\), whose space of solution is contractible.
Thus the operad-map fillers of the original square form a union of components of the space of diagonal fillers in \(\operatorname{Cat}_\infty\) of the square above.

\emph{Step 3: the lifting problem in \(\operatorname{Cat}_\infty\).}
Since \(\tilde g\) is essentially surjective and \(t \circ \iota\) is faithful (\cref{lem:mono_operad_target_faithful}), \cite[lem.~A.6]{Ramzi_2023_ElementaryProofNaturality} shows that the space of fillers in \(\operatorname{Cat}_\infty\) is empty or contractible, and is non-empty if and only if for all objects \(k, k'\) of \(\operatorname{Fr}(K)^\otimes\) and all \(o\colon \tilde g(k) \to \tilde g(k')\), the morphism \(f(o)\) lies in the image of \(\operatorname{Map}_{(\mathcal{C}^{\mathrm{mono}}_{/\mathcal{B}})^\otimes}(\tilde h(k), \tilde h(k')) \to \operatorname{Map}_{\mathcal{C}^\otimes}(f\tilde g(k), f\tilde g(k'))\); call such \(o\) \emph{liftable}.
In particular the operad-map filler space is empty or contractible.
Liftability is closed under composition, and every \(\tilde g(w)\) is liftable by commutativity of the square.

For \(o\) with unary target, over an active \(\beta\), the mapping space formula of \cref{lem:relative_arrow_mapping_spaces} identifies a preimage of \(f(o)\) with a factorization as in condition (1); so (1) holds if and only if every active morphism with unary target is liftable.
For general \(o\) over \(\varphi\colon \langle m\rangle \to \langle n\rangle\): if \(n = 0\), both mapping spaces are contractible by the Segal condition \cite[rmk.~2.1.2.6]{Lurie_HA} and \(o\) is liftable; if \(n \geq 1\), computing the Segal decompositions of both sides with the canonical inert lifts \(\tilde h(\operatorname{can}_{\rho^i})\) and their images shows that \(o\) is liftable if and only if each \(\tilde g(\operatorname{can}_{\rho^i}) \circ o\) is, and these have unary target.
Each such factors, by \cite[prop.~2.1.2.4]{Lurie_HA}, as an inert followed by an active morphism with unary target. 
The inert piece is liftable since it is equivalent to \(\tilde{g}\) applied to an inert, the active piece by condition (1), so \(o\) is liftable.

\emph{Step 4: morphism of operads.}
Let \(F\) be the essentially unique filler, promoted over \(\mathcal{S}\).
By \cite[rmk.~2.1.2.9]{Lurie_HA}, \(F\) preserves all inert morphisms if and only if it preserves those over the maps \(\rho^i\colon \langle n\rangle \to \langle 1\rangle\).
Since \(\tilde g\) is essentially surjective and inert morphisms and \(F\)-images are stable under composition with equivalences, it suffices to treat inert \(o\colon \tilde g(k) \to g(k')\) over \(\rho^i\) with both endpoints \(g\)-images.
Such an \(o\) factors as
\[
  \tilde g(k) \xrightarrow{\ \tilde g(\operatorname{can}_{\rho^i})\ } g(k_i) \xrightarrow{\ u\ } g(k'),
\]
where \(u\) is an equivalence of \(\mathcal{O}^\otimes_{\langle 1\rangle}\) comparing the two inert lifts of \(\rho^i\).
Thus, we find that \(F(o)\) is equivalent to the composite
\[\tilde{h}(k)\xrightarrow{\tilde{h}({\mathrm{can}_{\rho^i}})}h(k_i)\simeq F(g(k_i))\xrightarrow{F(u)}F(g(k')),\]
so it is inert itself.
\end{proof}

We record two consequences.
The first is \cref{lem:lifting_monos_into_modules}; the second is the analogous statement for subalgebras.

\begin{proof}[Proof of \cref{lem:lifting_monos_into_modules}]
Take \(\mathcal{B}^\otimes = \mathcal{O}^\otimes = \mathcal{LM}^\otimes\), let \(q\colon \mathcal{C}^\otimes \to \mathcal{LM}^\otimes\) be the coCartesian fibration encoding \(\mathcal{M}\) over \(\mathcal{A}\), and let \(f\) be the section classifying \((R, M)\).
Let \(g\) identify the two colors \(\{\mathfrak{a}, \mathfrak{m}\}\) with \(\mathcal{LM}^\otimes_{\langle 1\rangle}\), and set \(h(\mathfrak{a}) = (\operatorname{id}_R\colon R \to R)\) and \(h(\mathfrak{m}) = (i\colon U \to p(M))\); by \cref{rmk:mono_notions}, \(i\) is a monomorphism in \(\mathcal{C}^\otimes\) since it is one in the fibre \(\mathcal{M}\).
The space of lifts of \(i\) to \(\operatorname{LMod}_R(\mathcal{M})\) is equivalent to the space of operad-map fillers of the resulting diagram: a filler additionally chooses an algebra structure at \(\mathfrak{a}\) refining \(R\), but this datum is contractible, being an algebra map \(A \to R\) over \(\operatorname{id}_{p(R)}\), automatically an equivalence by conservativity of \(\operatorname{Alg}(\mathcal{A}) \to \mathcal{A}\).
Thus \cref{lem:lifting_monos_into_O_algebras} applies and the space of lifts is empty or contractible.
Condition (2) is automatic: \(\{\mathfrak{a}, \mathfrak{m}\}\) is discrete and \(g\) is an equivalence, so any equivalence \(u\colon g(k) \to g(k')\) forces \(k = k'\) and \(u \simeq \operatorname{id}\), whose factorization is the identity.
For condition (1), the active morphisms with target \(\mathfrak{a}\) have source a tuple of copies of \(\mathfrak{a}\), and factor since \(h(\mathfrak{a}) = \operatorname{id}_R\) is an equivalence (including the nullary case, the unit \(\mathbf{1} \to R\)); those with target \(\mathfrak{m}\) have source \((\mathfrak{a},\dots,\mathfrak{a},\mathfrak{m})\) and require that
\[
  R^{\otimes n} \otimes U \longrightarrow R^{\otimes n} \otimes p(M) \longrightarrow p(M)
\]
factor through \(i\).
By associativity of the action this holds for all \(n \geq 1\) as soon as it holds for \(n = 1\), while \(n = 0\) is the identity on \(U\).
Hence a filler exists if and only if \(R \otimes U \to R \otimes p(M) \to p(M)\) factors through \(i\).
Conversely, evaluating a filler on the binary action \(\varphi \in \operatorname{Mul}_{\mathcal{LM}}(\{\mathfrak{a}, \mathfrak{m}\}, \mathfrak{m})\) produces exactly this factorization, so the condition is also necessary.
\end{proof}

\begin{corollary}\label{cor:subalgebras}
Let \(\mathcal{A}\) be a monoidal \(\infty\)-category, \(A \in \operatorname{Alg}(\mathcal{A})\), and \(i\colon U \to A\) a monomorphism in \(\mathcal{A}\).
The space of lifts of \(i\) to a morphism \(U' \to A\) in \(\operatorname{Alg}(\mathcal{A})\) is empty or contractible, and non-empty if and only if the unit \(\mathbf{1} \to A\) and the composite \(U \otimes U \to A \otimes A \to A\) both factor through \(i\).
\end{corollary}

\begin{proof}
Apply \cref{lem:lifting_monos_into_O_algebras} with \(\mathcal{B}^\otimes = \mathcal{O}^\otimes = \operatorname{Assoc}^\otimes\), the coCartesian fibration \(\mathcal{C}^\otimes \to \operatorname{Assoc}^\otimes\) encoding \(\mathcal{A}\), and \(f\) the section classifying \(A\); here \(K\) has a single object with \(h = (i\colon U \to A)\), and fillers over \(\operatorname{Assoc}^\otimes\) are algebra maps.
Condition (2) is automatic as before.
In condition (1) the active morphisms are the multiplications \(\langle n\rangle \to \langle 1\rangle\); the cases \(n = 0\) and \(n = 2\) give the stated factorizations of the unit and of \(U \otimes U \to A\), and the remaining cases follow from these by associativity.
\end{proof}

\printbibliography

\end{document}